\documentclass[a4paper,12pt]{article}
\usepackage{amsthm,amsmath,amssymb,
bbm,geometry,epsfig,hyperref,enumerate,
comment,nicefrac,listings,graphicx,color,amsfonts,booktabs, 
subcaption,float}

\usepackage{courier}
\newtheorem{lemma}{Lemma}[section]

\newtheorem{proposition}[lemma]{Proposition}
\newtheorem{theorem}[lemma]{Theorem}

\newtheorem{algo}[lemma]{Framework}
\newtheorem{framework}[lemma]{Framework}

\providecommand{\N}{{\ensuremath{\mathbb{N}}}}

\providecommand{\R}{{\ensuremath{\mathbb{R}}}}
\providecommand{\B}{\mathcal{B}}

\renewcommand{\P}{\mathbbm{P}}

\providecommand{\bS}{\mathbb{S}}

\renewcommand{\S}{\mathcal{S}}
\providecommand{\E}{{\ensuremath{\mathbbm{E}}}}

\providecommand{\N}{{\ensuremath{\mathbbm{N}}}}
\providecommand{\bU}{{\ensuremath{\mathbb{U}}}}

\providecommand{\bX}{{\ensuremath{\mathbbm{X}}}}

\providecommand{\R}{{\ensuremath{\mathbbm{R}}}}

\providecommand{\E}{{\ensuremath{\mathbb{E}}}}
\providecommand{\cE}{{\ensuremath{\mathcal{E}}}}
\newcommand{\F}{{\ensuremath{\mathcal{F}}}}
\newcommand{\bF}{{\ensuremath{\mathbb{F}}}}

\newcommand{\X}{{\ensuremath{\mathcal{X}}}}

\hypersetup{linktocpage=true}

\title{\vspace{-0.5cm}Solving the Kolmogorov PDE \\ by means of deep learning}

\author{Christian Beck$^1$,
	Sebastian Becker$^2$,
	Philipp Grohs$^3$,\\
	Nor Jaafari$^4$,
	and 
	Arnulf Jentzen$^{5}$
	\bigskip
	\\
	\small{$^1$ Department of Mathematics, ETH Zurich,}\\
	\small{Zurich, Switzerland, e-mail: 
		christian.beck@math.ethz.ch}
	\smallskip
	\\
	\small{$^2$ ZENAI AG, Zurich, Switzerland, e-mail: sebastian.becker@zenai.ch}
	\smallskip
	\\
	\small{$^3$ Faculty of Mathematics and Research Platform Data Science,}\\ \small{University of Vienna, Vienna, Austria, e-mail: philipp.grohs@univie.ac.at}
	\smallskip
	\\
	\small{$^4$ ZENAI AG, Zurich, Switzerland, e-mail: nor.jaafari@zenai.ch}
	\smallskip
	\\
	\small{$^5$ Department of Mathematics, ETH Zurich, Zurich,}\\
	\small{Switzerland, e-mail: arnulf.jentzen@sam.math.ethz.ch}}

\begin{document}

\maketitle
\vspace{-0.5cm}
\begin{abstract}
	Stochastic differential equations (SDEs) and the Kolmogorov partial differential equations (PDEs) associated to them have been widely used in models from engineering, finance, and the natural sciences. 
	In particular, SDEs and Kolmogorov PDEs, respectively, are highly employed in models for the approximative pricing of financial derivatives. 
	Kolmogorov PDEs and SDEs, respectively, can typically not be solved explicitly and it has been and still is an active topic of research to design and analyze numerical methods which are able to approximately solve Kolmogorov PDEs and SDEs, respectively. Nearly all approximation methods for Kolmogorov PDEs in the literature suffer under the curse of dimensionality or only provide approximations of the solution of the PDE at a single fixed space-time point. In this paper we derive and propose a numerical approximation method which aims to overcome both of the above mentioned drawbacks and intends to deliver a numerical approximation of the Kolmogorov PDE on an entire region $[a,b]^d$ without suffering from the curse of dimensionality.
	Numerical  results  on  examples  including  the heat equation, the Black-Scholes model, the stochastic Lorenz equation, and the Heston model suggest that the proposed approximation algorithm is quite effective in high dimensions in terms of both accuracy and speed.
\end{abstract}

\tableofcontents

\section{Introduction}

Stochastic differential equations (SDEs) and the Kolmogorov partial differential equations (PDEs) associated to them have been widely used in models from engineering, finance, and the natural sciences. 
In particular, SDEs and Kolmogorov PDEs, respectively, are highly employed in models for the approximative pricing of financial derivatives. 
Kolmogorov PDEs and SDEs, respectively, can typically not be solved explicitly and it has been and still is an active topic of research to design and analyze numerical methods which are able to approximately solve Kolmogorov PDEs and SDEs, respectively (see, e.g.,
\cite{Giles_MultilevelMonteCarlo2008}, 
\cite{TalayGrahamBook}, 
\cite{Gyoengy_ANoteOnEulersApproximations1998}, 
\cite{HighamIntroduction}, 
\cite{Higham2011stochastic}, 
\cite{HighamMaoStuart_StrongConvergenceOfEuler2002}, 
\cite{HutzenthalerJentzen2015_Memoirs},
\cite{Kloeden_SystematicDerivation2002}, 
\cite{KloedenPlaten1992}, 
\cite{KloedenPlatenSchurz2012}, \cite{Maruyama_ContinuousMarkovProcessesAndStochasticEquations1955},
\cite{Milstein1974},
\cite{Milstein1995}, 
\cite{MilsteinTretyakovBook},
\cite{GronbachRitterMinimal2008}, 
\cite{Roessler2009}). 
In particular, there are nowadays several different types of numerical approximation methods for Kolmogorov PDEs in the literature including deterministic numerical approximation methods such as finite differences based approximation methods (cf., for example, \cite{BrennanSchwartz1977},  \cite{BrennanSchwartz1978}, \cite{HanWuFD03}, \cite{KUSHNERFD76}, \cite{YhaoDavisonCorlessFD07}, \cite{SCHWARTZ1977}) and finite elements based approximation methods (cf., for example, \cite{brenner2007mathematical}, \cite{ciarlet1991basic}, \cite{zienkiewicz1977finite}) as well as random numerical approximation methods based on Monte Carlo methods (cf., for example,~\cite{Giles_MultilevelMonteCarlo2008}, \cite{TalayGrahamBook}) and discretizations of the underlying SDEs (cf., for example,
\cite{Gyoengy_ANoteOnEulersApproximations1998}, 
\cite{HighamIntroduction}, 
\cite{Higham2011stochastic}, 
\cite{HighamMaoStuart_StrongConvergenceOfEuler2002}, 
\cite{HutzenthalerJentzen2015_Memoirs},
\cite{Kloeden_SystematicDerivation2002},  
\cite{KloedenPlaten1992}, 
\cite{KloedenPlatenSchurz2012}, \cite{Maruyama_ContinuousMarkovProcessesAndStochasticEquations1955},
\cite{Milstein1974},
\cite{Milstein1995},
\cite{MilsteinTretyakovBook},
\cite{GronbachRitterMinimal2008}, 
\cite{Roessler2009}). 
The above mentioned deterministic approximation methods for PDEs work quite efficiently in one or two space dimensions but cannot be used in the case of high-dimensional PDEs as they suffer from the so-called curse of dimensionality (cf.\  Bellman~\cite{Bellman}) in the sense that the computational effort of the considered approximation algorithm grows exponentially in the PDE dimension. The above mentioned random numerical approximation methods involving Monte Carlo approximations typically overcome this curse of dimensionality but only provide approximations of the Kolmogorov PDE at a single fixed space-time point. 

The key contribution of this paper is to derive and propose a numerical approximation method which aims to overcome both of the above mentioned drawbacks and intends to deliver a numerical approximation of the Kolmogorov PDE on an entire region $[a,b]^d$ without suffering from the curse of dimensionality. The numerical scheme, which we propose in this work, is inspired by recently developed deep learning based approximation algorithms for PDEs in the literature (cf., for example,  \cite{BeckEJentzen2017}, \cite{DeepStopping}, \cite{EHanJentzen2017a}, \cite{EYu2017}, \cite{FujiiTakahashiTakahashi2017}, \cite{EHanJentzen2017b}, \cite{Labordere2017}, \cite{raissi2018forward}, \cite{SirignanoSpiliopoulos2017}). To derive the proposed approximation scheme we first reformulate the considered Kolmogorov PDE as a suitable infinite dimensional stochastic optimization problem (see items~\eqref{it:exuniqueSolutions}--\eqref{it:UboundaryCond} in Proposition~\ref{proposition:minimizingPropertyApplied_new} below for details). This infinite dimensional stochastic optimization problem is then temporally discretized by means of suitable discretizations of the underlying SDE and it is spatially discretized by means of fully connected deep artificial neural network approximations (see~\eqref{eq:toMinimizeApprox} in Subsection~\ref{subsec:sgd} as well as Subsections~\ref{subsec:discretizationOfX}--\ref{subsec:DNNapproximations} below). The resulting finite dimensional stochastic optimization problem is then solved by means of stochastic gradient descent type optimization algorithms (see~\eqref{eq:plainGradientDescent} in Subsection~\ref{subsec:sgd}, Framework~\ref{algo:special} in Subsection~\ref{subsec:desc_algo}, Framework~\ref{algo:general_algorithm} in Subsection~\ref{subsec:generalOptDescription},  as well as \eqref{eq:Adam1}--\eqref{eq:Adam2} in Subsection~\ref{sec:adam}). We test the proposed approximation method numerically in the case of several examples of SDEs and PDEs, respectively (see Subsections~\ref{sec:paraboliceq}--\ref{sec:heston} below for details). The obtained numerical results indicate that the proposed approximation algorithm is quite effective in high dimensions in terms of both accuracy and speed.

The remainder of this article is organized as follows. In Section~\ref{sec:derivationalgo} we derive the proposed approximation algorithm (see Subsections~\ref{subsec:kolmogorovEq}--\ref{subsec:sgd} below) and we present a detailed description of the proposed approximation algorithm in a special case (see Subsection~\ref{subsec:desc_algo} below) as well as in the general case (see Subsection~\ref{subsec:generalOptDescription} below). In Section~\ref{sec:examples} we test the proposed algorithm numerically in the case of several examples of SDEs and PDEs, respectively. The employed source codes for the numerical simulations in Section~\ref{sec:examples} are postponed to Section~\ref{sec:source_code}.

\section{Derivation and description of the proposed approximation algorithm}
\label{sec:derivationalgo}

%%%%%%%%%%%%%%%%%%%%%%%%%%%%%%%%%%%%%%%%%%%%%%%%%%%%%%%%%%%%%%%%%%%%%%%%%%%%
%%%%%%%%%%%%%%% KOLMOGOROV EQUATIONS %%%%%%%%%%%%%%%%%%%%%%%%%%%%%%%%%%%%%%%
%%%%%%%%%%%%%%%%%%%%%%%%%%%%%%%%%%%%%%%%%%%%%%%%%%%%%%%%%%%%%%%%%%%%%%%%%%%%

In this section we describe the approximation
problem which we intend to solve (see Subsection~\ref{subsec:kolmogorovEq} below)
and we derive (see Subsections~\ref{subsec:connection}--\ref{subsec:sgd} below)
and specify (see Subsections~\ref{subsec:desc_algo}--\ref{subsec:generalOptDescription} below)
the numerical scheme which we suggest to use to solve this approximation problem
(cf., for example, E et al.~\cite{EHanJentzen2017a},
Han et al.~\cite{EHanJentzen2017b}, 
Sirignano \& Spiliopoulos~\cite{SirignanoSpiliopoulos2017},
Beck et al.~\cite{BeckEJentzen2017},
Fujii, Takahashi, A., \& Takahashi, M.~\cite{FujiiTakahashiTakahashi2017},
and Henry-Labordere~\cite{Labordere2017} 
for related derivations and related approximation schemes).

\subsection{Kolmogorov partial differential equations (PDEs)}
\label{subsec:kolmogorovEq}
Let 
  $ T \in (0,\infty) $, 
  $ d \in \N $, 
let 
  $\mu\colon \R^d\to\R^d$ 
  and 
  $\sigma\colon \R^d\to\R^{d\times d}$ 
  be Lipschitz continuous functions,  
let 
  $
    \varphi\colon \R^d \to \R
  $
  be a function, 
  and 
let 
  $ u = (u(t,x))_{(t,x)\in [0,T]\times\R^d}
  \in C^{1,2}([0,T]\times\R^d,\R)$
  be a function with at most polynomially growing partial derivatives
  which satisfies for every 
  $t\in [0,T]$, $x\in\R^d$ that 
  $u(0,x) = \varphi(x)$ and 
  \begin{equation}\label{eq:kolmogorovPDE}
   \tfrac{\partial u}{\partial t}(t,x) 
   = 
   \tfrac12 \operatorname{Trace}_{\R^d}\!\big(
   \sigma(x)[\sigma(x)]^{*}(\operatorname{Hess}_x u)(t,x)\big)
   + 
   \langle 
    \mu(x),(\nabla_x u)(t,x)
   \rangle_{\R^d}. 
  \end{equation}
Our goal is to approximately calculate the function 
$\R^d\ni x\mapsto u(T,x)\in\R$ on some subset of $\R^d$. 
To fix ideas we consider real numbers $a,b\in\R$ with 
$a<b$ and we suppose that our goal is to approximately 
calculate the function $[a,b]^d\ni x\mapsto u(T,x)\in\R$.

%%%%%%%%%%%%%%%%%%%%%%%%%%%%%%%%%%%%%%%%%%%%%%%%%%%%%%%%%%%%%%%%%%%%%%%%%%%%
%%%%%%%%%%%%%%%%%%% CONNECTION TO SDES %%%%%%%%%%%%%%%%%%%%%%%%%%%%%%%%%%%%%
%%%%%%%%%%%%%%%%%%%%%%%%%%%%%%%%%%%%%%%%%%%%%%%%%%%%%%%%%%%%%%%%%%%%%%%%%%%%

\subsection{On stochastic differential equations and Kolmogorov PDEs}
\label{subsec:connection}
In this subsection we provide a probabilistic representation for the solutions 
of the PDE~\eqref{eq:kolmogorovPDE}, that is, we recall 
the classical Feynman-Kac formula for the PDE 
\eqref{eq:kolmogorovPDE} (cf., for example, {\O}ksendal~\cite[Chapter 8]{Oksendal_SDE2003}). 

Let 
  $(\Omega,\F,\P)$ be a probability space 
  with a normal filtration $(\bF_t)_{t\in [0,T]}$, 
let 
  $W\colon [0,T]\times\Omega \to \R^d$ 
  be a standard $(\Omega,\F,\P,(\bF_t)_{t\in [0,T]})$-Brownian motion, 
and for every $x\in\R^d$ let 
  $X^x=(X^x_t)_{t\in [0,T]}\colon [0,T]\times\Omega\to\R^d$ 
  be an $(\bF_t)_{t\in [0,T]}$-adapted stochastic process with 
  continuous sample paths which satisfies that for every 
  $t\in [0,T]$ it holds $\P$-a.s.~that
  \begin{equation}\label{eq:SDEforXstartingAtx}
   X^x_t = x + \int_0^t \mu(X^x_s)\,ds + \int_0^t \sigma(X^x_s)\,dW_s. 
  \end{equation}
  The Feynman-Kac formula 
  (cf., for example, Hairer et al.~\cite[Corollary 4.17 and Remark 4.1]
  {HairerHutzenthalerJentzen_LossOfRegularity2015}) and \eqref{eq:kolmogorovPDE}
  hence yield that for every $x\in\R^d$ it holds that 
  \begin{equation}\label{eq:feynman-kac}
   u(T,x) = \E\big[u(0,X^x_T)\big] = \E\big[\varphi(X^x_T)\big]. 
  \end{equation}

%%%%%%%%%%%%%%%%%%%%%%%%%%%%%%%%%%%%%%%%%%%%%%%%%%%%%%%%%%%%%%%%%%%%%%%%%%%%
%%%%%%%%%%%%%%%% FORMULATION AS MINIMIZATION PROBLEM %%%%%%%%%%%%%%%%%%%%%%%
%%%%%%%%%%%%%%%%%%%%%%%%%%%%%%%%%%%%%%%%%%%%%%%%%%%%%%%%%%%%%%%%%%%%%%%%%%%%
  
\subsection{Formulation as minimization problem} 

In the next step we exploit \eqref{eq:feynman-kac}
to formulate a minimization problem which is uniquely solved by 
the function $[a,b]^d\ni x\mapsto u(T,x)\in\R$ (cf.~\eqref{eq:kolmogorovPDE} above). 
For this we first recall the $L^2$-minimization property of the expectation 
of a real-valued random variable (see Lemma~\ref{lemma:minimizingProperty} below). 
Then we extend this minimization result to certain random fields 
(see Proposition~\ref{proposition:minimizingProperty} below). 
Thereafter, we apply Proposition~\ref{proposition:minimizingProperty}
to random fields in the context of the Feynman-Kac representation~\eqref{eq:feynman-kac} 
to obtain Proposition~\ref{proposition:minimizingPropertyApplied_new} below. 
Proposition~\ref{proposition:minimizingPropertyApplied_new} 
provides a minimization problem 
(see, for instance,~\eqref{eq:functionMinimizedByU} below) 
which has the function $[a,b]^d\ni x\mapsto u(T,x)\in\R$ 
as the unique global minimizer. 

Our proof of Proposition~\ref{proposition:minimizingPropertyApplied_new} is based on the elementary auxiliary results in Lemmas~\ref{lem:projection_new}--\ref{lemma:equivalence_of_expectation_representations}. For completeness we also present the proofs of Lemmas~\ref{lem:projection_new}--\ref{lemma:equivalence_of_expectation_representations} here. The statement and the proof of Lemma~\ref{lem:projection_new} are based on the proof of Da Prato \& Zabczyk \cite[Lemma 1.1]{DaPratoZabczyk2008}.

\begin{lemma}\label{lemma:minimizingProperty}
 Let $(\Omega,\F,\P)$ be a probability space 
 and let $X\colon \Omega\to\R$ be an $\F$/$\B(\R)$-measurable random 
 variable which satisfies $\E[|X|^2] < \infty$. Then 
 \begin{enumerate}[(i)]
  \item\label{it:minimizingProperty_1} 
  it holds for 
  every $y\in\R$ that 
 \begin{equation}\label{eq:lemma_orthogonality_property}
  \E\big[ | X - y |^2\big]
  = 
  \E\big[ | X - \E[X] |^2 \big] + |\E[X]-y|^2,  
 \end{equation}
  \item\label{it:minimizingProperty_2}
  it holds that there exists a unique real number $z\in\R$ 
  such that 
  \begin{equation}\label{eq:lemma_minimizer_existence}
   \E\big[ | X - z |^2 \big] = \inf_{y\in\R} \E\big[ | X - y |^2 \big], 
  \end{equation}
  and 
  \item\label{it:minimizingProperty_3}
  it holds that 
  \begin{equation}\label{eq:lemma_minimizer_uniqueness}
   \E\big[ | X - \E[X] |^2 \big] = \inf_{y\in\R} \E\big[ | X - y |^2 \big].  
  \end{equation}
 \end{enumerate}
\end{lemma}

\begin{proof}[Proof of Lemma~\ref{lemma:minimizingProperty}]
 Observe that the fact that $\E[|X|]<\infty$ ensures that
 for every $y\in\R$ it holds that
 \begin{equation}\label{eq:proof_orthogonality}
  \begin{split}
  \E\big[ | X - y |^2 \big] 
  & = 
  \E\big[ | X - \E[X] + \E[X] - y |^2 \big] 
  \\
  & = 
  \E\big[ | X - \E[X] |^2 + 2 (X-\E[X]) (\E[X] - y) + |\E[X] - y|^2 \big]
  \\
  & =
  \E\big[ | X -\E[X] |^2 \big] + 2(\E[X]-y) \E\big[ X - \E[X]\big] + |\E[X] - y|^2 \\
  & = 
  \E\big[ | X - \E[X] |^2 \big] + |\E[X] - y|^2.
  \end{split}
 \end{equation}
 This establishes item~\eqref{it:minimizingProperty_1}. 
 Item~\eqref{it:minimizingProperty_2} and item~\eqref{it:minimizingProperty_3} 
 are immediate consequences of item~\eqref{it:minimizingProperty_1}. 
 The proof of Lemma~\ref{lemma:minimizingProperty} 
 is thus completed.
\end{proof}

\begin{proposition}\label{proposition:minimizingProperty}
 Let 
  $a\in\R$, $b\in (a,\infty)$, 
 let 
  $(\Omega,\F,\P)$ be a probability space,  
 let 
  $ X = ( X_x )_{ x \in [a,b]^d } \colon [a,b]^d \times \Omega \to \R $ 
  be a 
  $ ( \mathcal{B}( [a,b]^d ) \otimes \mathcal{F} ) $/$ \mathcal{B}( \R ) $-measurable function, 
 assume 
  for every $x\in [a,b]^d$ that $\E[|X_x|^2]<\infty$, 
 and 
 assume 
  that the function $[a,b]^d\ni x\mapsto \E[X_x]\in\R$ is continuous. 
 Then 
 \begin{enumerate}[(i)]
  \item\label{it:minimizingProperty_prop_1} 
  it holds that there exists a unique continuous function 
  $u\colon [a,b]^d\to\R$ such that 
  \begin{equation}\label{eq:existenceAndUniquenessOfMinimizer}
   \int_{[a,b]^d} \E\big[ | X_x - u(x) |^2 \big] \,dx
   = 
   \inf_{v\in C([a,b]^d,\R)} \bigg(\int_{[a,b]^d} \E\big[ | X_x - v(x) |^2 \big] \,dx\bigg)
  \end{equation} 
  and
  \item\label{it:minimizingProperty_prop_2} 
  it holds for every $x\in [a,b]^d$ that $u(x)=\E[X_x]$.
 \end{enumerate}
\end{proposition}

\begin{proof}[Proof of Proposition~\ref{proposition:minimizingProperty}]
Observe that item~\eqref{it:minimizingProperty_1} 
in Lemma~\ref{lemma:minimizingProperty} 
and the hypothesis that
$ \forall \, x \in [a,b]^d \colon \E[ |X_x|^2 ] < \infty $
ensure that for every function $ u \colon [a,b]^d \to \R $ 
and every $x\in [a,b]^d$ it holds that
\begin{equation}
\label{eq:minimizingPropertyQuantified}
  \E\big[ | X_x - u(x) |^2\big]
  = 
  \E\big[ | X_x - \E[X_x] |^2 \big] + | \E[X_x]- u(x) |^2 .
\end{equation}
Fubini's theorem (see, e.g., Klenke~\cite[Theorem 14.16]{Klenke_2014}) 
hence proves that for every continuous function
$ u \colon [a,b]^d \to \R $ it holds that
\begin{equation}
\label{eq:minimizingPropertyQuantified_2}
  \int_{ [a,b]^d }
  \E\big[ | X_x - u(x) |^2\big] \, dx
  = 
  \int_{ [a,b]^d }
  \E\big[ | X_x - \E[X_x] |^2 \big] \, dx 
  + 
  \int_{ [a,b]^d }
  | \E[X_x]- u(x) |^2 \, dx .
\end{equation}
The hypothesis that the function 
$ [a,b]^d \ni x \mapsto \E[ X_x ] \in \R $ 
is continuous therefore demonstrates that
\begin{align}
\begin{split}
 &\int_{[a,b]^d} 
  \E\big[ | X_x - \E[X_x] |^2 \big] \, dx
\\&\geq
  \inf_{v\in C([a,b]^d,\R)} 
  \left(
    \int_{[a,b]^d}
    \E\big[ | X_x - v(x) |^2 \big] \, dx
  \right)
\\&=
  \inf_{v\in C([a,b]^d,\R)} 
  \left(
    \int_{ [a,b]^d }
    \E\big[ | X_x - \E[X_x] |^2 \big] \, dx 
    + 
    \int_{ [a,b]^d }
    | \E[X_x]- v(x) |^2 \, dx
  \right)
\\&\geq 
  \inf_{v\in C([a,b]^d,\R)} 
  \left(
    \int_{ [a,b]^d }
    \E\big[ | X_x - \E[X_x] |^2 \big] \, dx 
  \right)
\\&=
  \int_{ [a,b]^d }
  \E\big[ | X_x - \E[X_x] |^2 \big] \, dx .
\end{split} 
\end{align}
Hence, we obtain that
\begin{equation}
\label{eq:minimizingPropertyQuantified_3}
  \int_{[a,b]^d} 
  \E\big[ | X_x - \E[X_x] |^2 \big] \, dx
  =
  \inf_{v\in C([a,b]^d,\R)} 
  \left(
    \int_{[a,b]^d}
    \E\big[ | X_x - v(x) |^2 \big] \, dx
  \right) .
\end{equation}
Again the fact that the function 
$[a,b]^d\ni x\mapsto \E[X_x]\in\R$ is continuous therefore
proves that there exists a continuous function $u\colon [a,b]^d\to\R$ 
such that
\begin{equation}
\label{eq:minimizingPropertyQuantified_4}
  \int_{[a,b]^d} \E\big[ | X_x - u(x) |^2 \big] \,dx
  = 
  \inf_{v\in C([a,b]^d,\R)} 
  \left(
    \int_{[a,b]^d} \E\big[ | X_x - v(x) |^2 \big] \, dx
  \right) .
\end{equation}
Next observe that \eqref{eq:minimizingPropertyQuantified_2}
and~\eqref{eq:minimizingPropertyQuantified_3}
yield that for every continuous function $ u \colon [a,b]^d \to \R $ 
with
\begin{equation}
  \int_{[a,b]^d} 
  \E\big[ | X_x - u(x) |^2 \big] \, dx
  = 
  \inf_{v\in C([a,b]^d,\R)} 
  \left(
    \int_{[a,b]^d}\E\big[ | X_x - v(x) |^2 \big] \, dx
  \right) 
\end{equation}
it holds that 
\begin{align}
\begin{split}
 &\int_{[a,b]^d} 
  \E\big[ | X_x - \E[X_x] |^2 \big] \, dx
\\&=
  \inf_{v\in C([a,b]^d,\R)} 
  \left(
    \int_{[a,b]^d}
    \E\big[ | X_x - v(x) |^2 \big] \, dx
  \right)
  =
  \int_{ [a,b]^d] }
  \E\big[ | X_x - u(x) |^2\big] \, dx
\\&= 
  \int_{ [a,b]^d }
  \E\big[ | X_x - \E[X_x] |^2 \big] \, dx 
  + 
  \int_{ [a,b]^d }
  | \E[X_x]- u(x) |^2 \, dx .
\end{split}
\end{align}
Hence, we obtain
that for every continuous function $ u \colon [a,b]^d \to \R $ 
with
\begin{equation}
  \int_{[a,b]^d} 
  \E\big[ | X_x - u(x) |^2 \big] \, dx
  = 
  \inf_{v\in C([a,b]^d,\R)} 
  \left(
    \int_{[a,b]^d}\E\big[ | X_x - v(x) |^2 \big] \, dx
  \right) 
\end{equation}
it holds that 
\begin{equation}
  \int_{[a,b]^d} | \E[X_x] - u(x) |^2 \, dx = 0 .
\end{equation}
This and again the hypothesis that the function
$ [a,b]^d \ni x \mapsto \E[ X_x ] \in \R $ 
is continuous yield that for 
every continuous function $ u \colon [a,b]^d \to \R $ 
with
\begin{equation}
  \int_{[a,b]^d} 
  \E\big[ | X_x - u(x) |^2 \big] \, dx
  = 
  \inf_{v\in C([a,b]^d,\R)} 
  \left(
    \int_{[a,b]^d}\E\big[ | X_x - v(x) |^2 \big] \, dx
  \right) 
\end{equation}
and every $ x \in [a,b]^d $
it holds that 
$ u(x) = \E[ X_x ] $. 
Combining this with~\eqref{eq:minimizingPropertyQuantified_4}
completes the proof of Proposition~\ref{proposition:minimizingProperty}.  
\end{proof}

\begin{lemma}[Projections in metric spaces]
	\label{lem:projection_new}
	Let $(E,d)$ be a metric space, 
	let $n \in \N$, $e_1, e_2, \ldots, e_n \in E$,
	and let $P \colon E \rightarrow E $
	be the function which satisfies for every $ x \in E $ that
	\begin{equation}
	\label{eq:projection_1_new} 
	P(x)
	=
	e_{
		\min\{ 
		k \in \{1,2,\ldots,n\} 
		\colon 
		d(x,e_k) 
		= 
		\min\{ 
		d(x,e_1), 
		d(x,e_2),
		\ldots,
		d(x,e_n)
		\}
		\}
	} .  
	\end{equation}
	Then
	\begin{enumerate}[(i)]
		\item\label{it:projection_2_new} 
		it holds for every $ x \in E $ that
		\begin{equation}
		d(x, P(x))
		=
		\min_{k\in\{1,2,\ldots, n\}} d(x,e_k)
		\end{equation}
		and
		\item\label{it:projection_1_new} 
		it holds for every $A\subseteq E$ that 
		$P^{-1}(A)\in \B(E)$.
	\end{enumerate}
\end{lemma}
\begin{proof}[Proof of Lemma~\ref{lem:projection_new}]
	Throughout this proof let 
	$ D = (D_1, \ldots, D_n) \colon E \rightarrow \R^n $
	be the function which satisfies for every $ x \in E $
	that
	\begin{equation}\label{eq:definition_of_D}
	D(x)
	=
	\left(
	D_1(x),
	D_2(x),
	\ldots,
	D_n(x)
	\right)
	=
	\left(
	d(x, e_1),
	d(x, e_2),
	\ldots,
	d(x, e_n)
	\right) .
	\end{equation}
	Note that~\eqref{eq:projection_1_new} 
	ensures that for every $x\in E$ it holds that 
	\begin{equation}
	d(x,P(x)) 
	= 
	d(x,e_{
		\min\{ 
		k \in \{1,2,\ldots,n\} 
		\colon 
		d(x,e_k) 
		= 
		\min\{ 
		d(x,e_1), 
		d(x,e_2),
		\ldots,
		d(x,e_n)
		\}
		\}
	})
	= 
	\min_{k\in\{1,2,\ldots,n\}} d(x,e_k) . 
	\end{equation}
	This establishes item~\eqref{it:projection_2_new}. 
	It thus remains to prove item~\eqref{it:projection_1_new}. For this observe that the fact that the function 
	$d\colon E\times E\to [0,\infty)$ is continuous ensures 
	that the function $D\colon E\to\R^n$ is continuous. Hence, 
	we obtain that the function $D\colon E\to \R^n$ is 
	$ \mathcal{B}(E)$/$\mathcal{B}(\R^n)$-measurable. 
	Next note that item~\eqref{it:projection_2_new} 
	demonstrates that for every $k\in \{1,2,\ldots,n\}$, 
	$x\in P^{-1}(\{e_k\})$ it holds that 
	\begin{equation}
	d(x,e_k) 
	= 
	d(x,P(x)) 
	= 
	\min_{l\in\{1,2,\ldots,n\}} d(x,e_l). 
	\end{equation}
	Hence, we obtain that for every $k\in\{1,2,\ldots,n\}$, 
	$x\in P^{-1}(\{e_k\})$ it holds that 
	\begin{equation}\label{eq:first_ineq_for_k_and_min}
	k\geq \min\{l\in\{1,2,\ldots,n\}\colon d(x,e_l) 
	= \min\{d(x,e_1),d(x,e_2),\ldots,d(x,e_n)\}\} . 
	\end{equation}
	Moreover, note that~\eqref{eq:projection_1_new} 
	ensures that for every 
	$ k \in \{1,2,\ldots,n\}$, $x\in P^{-1}(\{e_k\})$ 
	it holds that 
	\begin{align}
	\begin{split}
	&\min\!\left\{l\in\{1,2,\ldots,n\}\colon d(x,e_l) 
	= \min_{u\in\{1,2,\ldots,n\}} d(x,e_u)\right\} \\
	&\in 
	\big\{ l \in \{1,2,\ldots,n\} \colon e_l = e_k\big\}
	\subseteq\big\{k, k+1, \ldots, n\big\}. 
	\end{split}
	\end{align}
	Therefore, we obtain that for every $k\in\{1,2,\ldots,n\}$, $x\in P^{-1}(\{e_k\})$  
	with $e_k\notin (\cup_{l\in \N\cap [0,k)} \{e_l\})$ it holds that 
	\begin{equation}
	\min\!\left\{l\in\{1,2,\ldots,n\}\colon d(x,e_l) 
	= \min_{u\in\{1,2,\ldots,n\}} d(x,e_u)\right\} 
	\geq k. 
	\end{equation}
	Combining this with~\eqref{eq:first_ineq_for_k_and_min} yields 
	that for every $k\in\{1,2,\ldots,n\}$, $x\in P^{-1}(\{e_k\})$ with 
	$e_k\notin  (\cup_{l\in \N\cap [0,k)} \{e_l\})$  
	it holds that 
	\begin{equation}
	\min\!\left\{l\in\{1,2,\ldots,n\}\colon d(x,e_l) 
	= \min_{u\in\{1,2,\ldots,n\}} d(x,e_u)\right\} 
	=
	k . 
	\end{equation}
	Hence, we obtain that for every $k\in\{1,2,\ldots,n\}$ 
	with 
	$e_k\notin  (\cup_{l\in \N\cap [0,k)} \{e_l\})$ it holds 
	that 
	\begin{equation}
	P^{-1}(\{e_k\}) \subseteq
	\left\{x\in E\colon 
	\min\!\left\{l\in\{1,2,\ldots,n\}\colon d(x,e_l) 
	= \min_{u\in\{1,2,\ldots,n\}} d(x,e_u)\right\} 
	=
	k\right\} . 
	\end{equation}
	This and \eqref{eq:projection_1_new} show that 
	for every $k\in\{1,2,\ldots,n\}$ with 
	$e_k\notin  (\cup_{l\in \N\cap [0,k)} \{e_l\})$ it holds 
	that
	\begin{equation}
	P^{-1}(\{e_k\}) 
	=
	\left\{x\in E\colon 
	\min\!\left\{l\in\{1,2,\ldots,n\}\colon d(x,e_l) 
	= \min_{u\in\{1,2,\ldots,n\}} d(x,e_u)\right\} 
	=
	k\right\} .
	\end{equation}
	Combining \eqref{eq:definition_of_D} with 
	the fact that the function $D\colon E\to\R^n$ 
	is $\B(E)$/$\B(\R^n)$-measurable therefore demonstrates that 
	for every $ k \in \{1,2,\ldots,n\}$ with 
	$e_k \notin (\cup_{l \in \N \cap [0,k)} \{e_l\}) $
	it holds that
	\begin{align}
	\begin{split}
	&
	P^{-1}(\{e_k\})
	\\&=
	\left\{ 
	x \in E 
	\colon
	\min\!\left\{ 
	l \in \{1, 2, \ldots, n \}
	\colon
	d(x, e_l)
	=
	\min_{u\in\{1,2,\ldots,n\}} d( x, e_u )
	\right\} 
	= k
	\right\}
	\\&=
	\left\{ 
	x \in E 
	\colon
	\min\!\left\{ 
	l \in \{1, 2, \ldots, n \}
	\colon
	D_l(x)
	=
	\min_{ u \in \{1,2,\ldots,n\} }
	D_u(x)
	\right\}
	= 
	k
	\right\} 
	\\&=
	\left\{ 
	x \in E 
	\colon
	\left( 
	\begin{array}{c}
	\forall \, l \in\N\cap[0,k) 
	\colon
	D_k(x) < D_l(x) \,\, \text{and} \\
	\forall \, l \in \{1,2,\ldots,n\}
	\colon
	D_k(x) \leq D_l(x)
	\end{array}
	\right)
	\right\} 
	\\&=
	\left[ 
	\bigcap_{l=1}^{k-1}
	\underbrace{
		\{ 
		x \in E 
		\colon 
		D_k(x) < D_l(x) 
		\}
	}_{ \in \mathcal{B}(E) }
	\right]
	\bigcap 
	\left[ 
	\bigcap_{l=1}^{n}
	\underbrace{
		\{ 
		x \in E 
		\colon 
		D_k(x) \leq D_l(x) 
		\}
	}_{ \in \mathcal{B}(E) }
	\right] 
	\in 
	\mathcal{B}(E) .
	\end{split}
	\end{align}
	Hence, we obtain that for every
	$ f \in \{e_1,e_2,\ldots,e_n\} $
	it holds that
	\begin{equation}
	P^{-1}(\{f\})
	\in
	\mathcal{B}(E) .
	\end{equation}
	Therefore, we obtain that 
	for every $ A \subseteq E $ it holds that
	\begin{align}
	\begin{split}
	P^{-1}(A)
	&=
	P^{-1}\!\left(
	A \cap \{e_1,e_2,\ldots,e_n\}
	\right)
	\\&=
	\cup_{f \in A \cap \{e_1,e_2,\ldots,e_n\}}
	\underbrace{ 
		P^{-1}(\{f\})
	}_{ \in \mathcal{B}(E) }
	\in \mathcal{B}(E) .
	\end{split}
	\end{align}
	This establishes item~\eqref{it:projection_1_new}.
	The proof of Lemma~\ref{lem:projection_new} is thus completed.
\end{proof}

\begin{lemma}
	\label{lem:productMeasurable_new}
	Let $ (E, d) $ be a separable metric space, 
	let $ (\cE, \delta) $ be a metric space,
	% let $ (\Omega, \F) $, $ (G, \mathcal{G}) $ be measurable spaces, 
	let $ (\Omega, \F) $ be a measurable space,
	let $ X \colon E \times \Omega \rightarrow \cE $
	be a function, 
	assume for every
	$ e \in E $ that the function 
	$ \Omega \ni \omega \mapsto X(e, \omega) \in \cE $
	is $\F$/$\mathcal{B}(\cE)$-measurable, 
	and assume for every $ \omega \in \Omega $
	that the function $ E \ni e \mapsto X(e, \omega) \in \cE $
	is continuous. Then it holds that the function
	$ X \colon E\times \Omega \to \cE $ is $(\mathcal{B}(E) \otimes \F)$/$\mathcal{B}(\cE) $-measurable.
\end{lemma}

\begin{proof}[Proof of Lemma~\ref{lem:productMeasurable_new}]
	Throughout this proof let
	$ (e_m)_{m \in \N} \subseteq E $ be a sequence which satisfies that 
	$\overline{\{ e_m \colon m \in \N \}} = E $,  
	let 
	$ P_n \colon E \rightarrow E $, $ n \in \N $,
	be the functions which satisfy for every $ n \in \N $, $ x \in E $ that
	\begin{align}
	P_n(x)
	=
	e_{
		\min\{ 
		k \in \{1,2,\ldots,n\} 
		\colon 
		d(x,e_k) 
		= 
		\min\{
		d(x,e_1), d(x,e_2), \ldots, d(x,e_n)
		\} 
		\}
	} ,
	\end{align}
	and let $\mathcal{X}_n \colon E \times \Omega \rightarrow \cE $, $ n \in \N $,
	be the functions which satisfy for every $ n \in \N $, $ x \in E $, $ \omega \in \Omega $ that
	\begin{equation}
	\label{eq:curlyXn}
	\mathcal{X}_n(x, \omega)
	=
	X( P_n(x), \omega) .
	\end{equation}Note that 
	\eqref{eq:curlyXn} shows that for all $ n \in \N $, $ B \in \mathcal{B}(\cE) $
	it holds that
	\begin{align}
	\begin{split}
	(\mathcal{X}_n)^{-1}(B)
	&=
	\left\{ 
	(x, \omega)
	\in E \times \Omega
	\colon
	\mathcal{X}_n(x, \omega) \in B
	\right\}
	\\&=
	\bigcup_{ y \in \operatorname{Im}(P_n) }
	\Big(
	\left[ 
	(\mathcal{X}_n)^{-1}(B)
	\right]
	\cap 
	\left[ 
	(P_n)^{-1}(\{y\}) \times \Omega
	\right] 
	\Big)
	\\&=
	\bigcup_{ y \in \operatorname{Im}(P_n) }
	\left\{ 
	(x, \omega)
	\in E \times \Omega
	\colon 
	\Big[
	\mathcal{X}_n(x, \omega) \in B
	\,\, \text{and} \,\,
	x \in (P_n)^{-1}(\{y\})
	\Big]
	\right\}
	\\&=
	\bigcup_{ y \in \operatorname{Im}(P_n) }
	\left\{ 
	(x, \omega)
	\in E \times \Omega
	\colon
	\Big[
	X(P_n(x), \omega) \in B
	\,\, \text{and} \,\,
	x \in (P_n)^{-1}(\{y\})
	\Big]
	\right\} 
	.
	\end{split}
	\end{align}
	Item~\eqref{it:projection_1_new} in Lemma~\ref{lem:projection_new}
	hence implies that for all $ n \in \N $, $ B \in \mathcal{B}(\cE) $
	it holds that
	\begin{align}
	\begin{split}
	(\mathcal{X}_n)^{-1}(B)
	&=
	\bigcup_{ y \in \operatorname{Im}(P_n) }
	\left\{ 
	(x, \omega)
	\in E \times \Omega
	\colon
	\Big[
	X(y, \omega) \in B
	\,\, \text{and} \,\,
	x \in (P_n)^{-1}(\{y\})
	\Big]
	\right\} 
	\\&=
	\bigcup_{ y \in \operatorname{Im}(P_n) }
	\Big(
	\left\{ 
	(x, \omega)
	\in E \times \Omega
	\colon
	X(y, \omega) \in B
	\right\}
	\cap 
	\left[ 
	(P_n)^{-1}(\{y\}) \times \Omega
	\right]
	\Big)
	\\&=
	\bigcup_{ y \in \operatorname{Im}(P_n) }
	\Big(
	\big[ 
	\underbrace{
		E
		\times \left(
		\left(X(y, \cdot)\right)^{-1}(B)\right)
	}_{\in (\mathcal{B}(E) \otimes \F)}
	\big] 
	\cap 
	\big[ 
	\underbrace{
		(P_n)^{-1}(\{y\})
		\times 
		\Omega
	}_{\in (\mathcal{B}(E) \otimes \F)}
	\big]
	\Big)
	\in (\mathcal{B}(E) \otimes \F) .
	\end{split}
	\end{align}
	This proves that for every $ n \in \N $ it holds that
	the function 
	$
	\mathcal{X}_n
	$
	is
	$
	(\mathcal{B}(E) \otimes \F)
	$/$
	\mathcal{B}(\cE)
	$-measurable.
	In addition, note that 
	item~\eqref{it:projection_2_new} in Lemma~\ref{lem:projection_new}
	and the hypothesis that for every $\omega\in\Omega$ it holds that the 
	function $E\ni x\mapsto X(x,\omega)\in\cE$ is continuous 
	imply that for every $ x \in E $, $ \omega \in \Omega $
	it holds that
	\begin{equation}
	\label{eq:productMeasurable_5_new}
	\lim_{ n \rightarrow \infty }
	\mathcal{X}_n(x, \omega)
	=
	\lim_{ n \rightarrow \infty }
	X( P_n(x), \omega)
	=
	X(x, \omega) .
	\end{equation}
	Combining this with the fact that for every $n\in\N$ it holds that 
	the function 
	$\X_n\colon E\times\Omega\to\cE$ is 
	$(\B(E)\otimes\F)$/$\B(\cE)$-measurable 
	shows that the function $X\colon E\times\Omega\to \cE$ 
	is $(\B(E)\otimes\F)$/$\B(\cE)$-measurable. 
	The proof of Lemma~\ref{lem:productMeasurable_new} is thus 
	completed.
\end{proof}

\begin{lemma}
	\label{lemma:stochastic_convergence_under_continuous_transformations}
	Let $ (\Omega, \F, \P) $ be a probability space, 
	let $ ( E, d ) $ and $ ( \cE, \delta ) $ 
	be separable metric spaces, 
	let $ X_n \colon \Omega \to E $, $ n \in \N_0 $, 
	be random variables which satisfy 
	for every $\varepsilon \in (0,\infty)$ that 
	\begin{equation}
	\label{eq:weakConvMetric1}
	\limsup_{n\to\infty} 
	\P(d(X_n,X_0)\geq \varepsilon) 
	= 0, 
	\end{equation}
	and let $\Phi\colon E \to \cE$ be a continuous function. 
	Then it holds for every $\varepsilon \in (0,\infty)$ that 
	\begin{equation}
	\label{eq:weakConvMetric2}
	\limsup_{n\to\infty} 
	\P(\delta(\Phi(X_n),\Phi(X_0))\geq\varepsilon)
	= 0 . 
	\end{equation}
\end{lemma}
 \begin{proof}[Proof of Lemma \ref{lemma:stochastic_convergence_under_continuous_transformations}]
 	Note that \eqref{eq:weakConvMetric1}, e.g., Cox et al.~\cite[Lemma 2.4]{CoxJentzenKurniawanPusnik}, and, e.g., Hutzenthaler et al.~\cite[Lemma 4.2]{HutzenthalerJentzenSalimova2016} establish \eqref{eq:weakConvMetric2}. The proof of Lemma~\ref{lemma:stochastic_convergence_under_continuous_transformations} is thus completed.
 \end{proof}

\begin{lemma}
	\label{lemma:equivalence_of_expectation_representations}
	Let 
	$d, m\in\N$, 
	$T\in (0,\infty)$, 
	$L, a\in\R$, 
	$b\in (a,\infty)$, 
	let 
	$\mu\colon \R^d\to\R^d$
	and
	$\sigma\colon \R^d\to\R^{d\times m}$ 
	be functions which satisfy for every $x,y\in\R^d$ that $\max\{\|\mu(x)-\mu(y)\|_{\R^d},\|\sigma(x)-\sigma(y)\|_{HS(\R^m,\R^d)}\}\le L\|x-y\|_{\R^d}$, 
	let 
	$\Phi \colon C([0,T],\R^d) \to \R$ be 
	an at most polynomially growing  
	continuous function, 
	let 
	$(\Omega,\F,\P)$ be a probability space with a normal filtration 
	$(\bF_t)_{t\in [0,T]}$, 
	let 
	$\xi\colon\Omega\to [a,b]^d$ be a continuous uniformly distributed 
	$\bF_0$/$\B([a,b]^d)$-measurable random variable,   
	let 
	$W\colon [0,T]\times\Omega\to\R^m$ be a standard
	$(\bF_t)_{t\in [0,T]}$-Brownian motion, 
	for every $x\in [a,b]^d$ let 
	$X^x=(X^x_t)_{t\in [0,T]}\colon [0,T]\times\Omega\to\R^d$ 
	be an $(\bF_t)_{t\in [0,T]}$-adapted stochastic process with 
	continuous sample paths which satisfies that for every 
	$t\in [0,T]$ it holds $\P$-a.s.~that 
	\begin{equation}\label{eq:lemma_equivalence_sde_for_x}
	X^x_t = x + \int_0^t \mu(X^x_s)\,ds + \int_0^t \sigma(X^x_s)\,dW_s, 
	\end{equation}
	and 
	let 
	$\bX\colon [0,T]\times\Omega\to\R^d$ be an 
	$(\bF_t)_{t\in [0,T]}$-adapted stochastic process with continuous 
	sample paths which satisfies that for every $t\in [0,T]$ it
	holds $\P$-a.s.~that 
	\begin{equation}
	\label{eq:lemma_equivalence_sde_for_bx}
	\bX_t 
	= 
	\xi 
	+ 
	\int_0^t \mu(\bX_s)\,ds 
	+
	\int_0^t \sigma(\bX_s)\,dW_s. 
	\end{equation}
	Then 
	\begin{enumerate}[(i)]
		\item\label{it:measurability}
		it holds for every 
		$x\in [a,b]^d$ that the functions $\Omega\ni\omega\mapsto\Phi((X^x_t(\omega))_{t\in [0,T]})\in\R$
		and $\Omega\ni\omega\mapsto\Phi((\bX_t(\omega))_{t\in [0,T]})\in\R$
		are 
		$\F$/$\B(\R)$-measurable,  
		\item\label{it:GWbound}
		it holds for every $p\in[2,\infty)$, $x,y\in[a,b]^d$ that 
		\begin{equation}
		\left(\E\left[\sup_{t\in[0,T]}\|X^x_t-X^y_t\|_{\R^d}\right]^p\right)^{\!\nicefrac{1}{p}}\le \sqrt{2}\exp\!\left(L^2T\big[p+\sqrt{T}\big]^2\right)\|x-y\|_{\R^d},
		\end{equation}
		\item\label{it:rhs_well_defined_for_every_x} 
		it holds for every 
		$x\in [a,b]^d$ that 
		$
		\E\!\left[|\Phi((X^x_t)_{t\in [0,T]})|
		+ 
		|\Phi((\bX_t)_{t\in [0,T]})| 
		\right] 
		< 
		\infty, 
		$
		\item\label{it:rhs_continuity_wrt_x} 
		it holds that the function 
		$[a,b]^d\ni x\mapsto \E[\Phi((X^x_t)_{t\in [0,T]})]\in\R$ 
		is continuous, and
		\item\label{it:identity_of_interest}
		it holds that 
		\begin{equation}
		\E\big[ \Phi((\bX_t)_{t\in [0,T]}) \big]
		= 
		\tfrac{1}{(b-a)^d}
		\left( 
		\int_{[a,b]^d} \E\big[ \Phi((X^x_t)_{t\in [0,T]}) \big]\,dx
		\right). 
		\end{equation}
	\end{enumerate}
\end{lemma}

\begin{proof}[Proof of Lemma~\ref{lemma:equivalence_of_expectation_representations}]
	Throughout this proof let $c\in[1,\infty)$ be a real number which satisfies for every $w\in C([0,T],\R^m)$ that 
	\begin{equation}
	|\Phi(w)| \leq c\left[1+\sup_{t\in [0,T]} \|w_t\|_{\R^d}\right]^c	\label{eq:lemma_equivalence_at_most_polynomial_growth}, 
	\end{equation}
	let $p_t\colon C([0,T],\R^m) \to \R^m$, $t\in [0,T]$, be the functions which satisfy for every $t\in [0,T]$, $w=(w_s)_{s\in [0,T]}\in C([0,T],\R^m)$ that $p_t(w)=w_t$, and let $\Psi^N_{x,w}\colon[0,T]\to\R^d$, $N\in \N$, $x\in\R^d$, $w\in C([0,T],\R^m),$
	be the functions which satisfy for every $w\in C([0,T],\R^m)$, $x\in\R^d$, $N\in \N$, $ n\in\{0,1,\ldots,N-1\}$, $t\in[\frac{nT}{N},\frac{(n+1)T}{N}]$ that $\Psi^N_{x,w}(0) = x$ and
	\begin{equation}
	\Psi^N_{x,w}(t) = \Psi^N_{x,w}\!\left(\tfrac{nT}{N}\right)+\left(\tfrac{nt}{T}-n\right)\left[\mu\!\left(\Psi^N_{x,w}\!\left(\tfrac{nT}{N}\right)\right)\tfrac{T}{N}+\sigma\!\left(\Psi^N_{x,w}\!\left(\tfrac{nT}{N}\right)\right)	\big( w_{\frac{(n+1)T}{N}} - w_{\frac{nT}{N}} \big)\right].\label{eq:PsiDef}
	\end{equation}
	Observe that the fact that the Borel sigma-algebra $\B(C([0,T],\R^m))$ is generated by the set 
	$\cup_{t\in[0,T]}\cup_{A\in\B(\R^m)}\{(p_t)^{-1}(A)\}$ (cf., for example, Klenke~\cite[Theorem~21.31]{Klenke_2014}), the hypothesis that for every $x\in[a,b]^d$ it holds that $X^x\colon [0,T]\times\Omega\to\R^d$ is an $(\bF_t)_{t\in [0,T]}$-adapted stochastic process with 
	continuous sample paths, and the hypothesis that  $\bX\colon [0,T]\times\Omega\to\R^d$ is an
	$(\bF_t)_{t\in [0,T]}$-adapted stochastic process with continuous sample paths demonstrate that the functions
	\begin{equation}
	\Omega\ni\omega\mapsto(X^x_t(\omega))_{t\in[0,T]}\in C([0,T],\R^d)\label{eq:omegaToX}
	\end{equation}
	and
	\begin{equation}
	\Omega\ni\omega\mapsto(\bX_t(\omega))_{t\in[0,T]}\in C([0,T],\R^d)\label{eq:omegaTobX}
	\end{equation}
	are 
	$\F$/$\B(C([0,T],\R^d))$-measurable.
	Combining this with the fact that the function 
	$\Phi\colon\allowbreak C([0,T],\R^d)\to\R$ is 
	$\B(C([0,T],\R^d))$/$\B(\R)$-measurable 
	implies that for every $x\in [a,b]^d$ it holds that 
	the functions $\Omega\ni\omega\mapsto\Phi((X^x_t(\omega))_{t\in[0,T]})\in\R$ and 
	$\Omega\ni\omega\mapsto\Phi((\bX_t(\omega))_{t\in[0,T]})\in\R$
	are $\F$/$\B(\R)$-measurable. This proves 
	item~\eqref{it:measurability}. 
	Next observe that \eqref{eq:PsiDef}, the hypothesis that $\mu\colon\R^d\to\R^d$ and $\sigma\colon\R^d\to\R^{d\times m}$ are globally Lipschitz continuous, and the fact that for every $p\in(0,\infty)$ it holds that $\E[\|\xi\|_{\R^d}^p]<\infty$ ensure that for every $p\in(0,\infty)$ it holds that
	\begin{equation}\label{eq:lemma_equivalence_lp_sup_numerics}
	\begin{split}
	&\sup_{N\in\N}
	\sup_{x\in [a,b]^d}
	\left( 
	\E\!\left[ 
	\sup_{t\in [0,T]} \| \Psi^N_{x, W}(t) \|_{\R^d}^p
	\right] 
	+
	\E\!\left[ 
	\sup_{t\in [0,T]} \| \Psi^N_{\xi, W}(t) \|_{\R^d}^p
	\right] 
	\right)\\
	&=\sup_{N\in\N}
	\sup_{x\in [a,b]^d}
	\left( 
	\E\!\left[ 
	\max_{n\in \{0,1,\ldots,N\}} \left\| \Psi^N_{x, W}\!\left(\tfrac{nT}{N}\right) \right\|_{\R^d}^p
	\right] 
	+
	\E\!\left[ 
	\max_{n\in \{0,1,\ldots,N\}}  \left\| \Psi^N_{\xi, W}\!\left(\tfrac{nT}{N}\right) \right\|_{\R^d}^p
	\right] 
	\right)<\infty
	\end{split}
	\end{equation}
	(cf., for example, Kloeden \& Platen~\cite[Section 10.6]{KloedenPlaten1992}). Next note that~\eqref{eq:lemma_equivalence_sde_for_x},  \eqref{eq:lemma_equivalence_sde_for_bx}, \eqref{eq:PsiDef}, the fact that $\mu\colon\R^d\to\R^d$ and $\sigma\colon\R^d\to\R^{d\times m}$ are locally Lipschitz continuous functions, and e.g., Hutzenthaler \& Jentzen~\cite[Theorem 3.3]{HutzenthalerJentzen2015_Memoirs} ensure that for every $x\in[a,b]^d$, $\varepsilon\in(0,\infty)$ it holds that
	\begin{equation}
	\label{eq:supDiscret1}
	\limsup_{N\to\infty}\P\!\left(\sup_{t\in[0,T]}\|X^x_t-\Psi^N_{x,W}(t)\|_{\R^d}\ge\varepsilon\right)=0
	\end{equation}
	and
	\begin{equation}
	\label{eq:supDiscret2}
	\limsup_{N\to\infty}\P\!\left(\sup_{t\in[0,T]}\|\bX_t-\Psi^N_{\xi,W}(t)\|_{\R^d}\ge\varepsilon\right)=0.
	\end{equation}
	Combining \eqref{eq:omegaToX}, \eqref{eq:omegaTobX}, and, e.g., Hutzenthaler \& Jentzen~\cite[Lemma 3.10]{HutzenthalerJentzen2015_Memoirs} hence demonstrates that for every $x\in[a,b]^d, p\in(0,\infty)$ it holds that
	\begin{equation}
	\begin{split}
	\E\!\left[\sup_{t\in[0,T]}\|X^x_t\|^p\right]&\le \liminf_{N\to\infty} \E\!\left[\sup_{t\in[0,T]}\|\Psi^N_{x,W}(t)\|^p\right]\\
	&\le\sup_{N\in\N}\E\!\left[\sup_{t\in[0,T]}\|\Psi^N_{x,W}(t)\|^p\right]
	\end{split}
	\end{equation}
	and
	\begin{equation}
	\begin{split}
	\E\!\left[\sup_{t\in[0,T]}\|\bX_t\|^p\right]&\le \liminf_{N\to\infty} \E\!\left[\sup_{t\in[0,T]}\|\Psi^N_{\xi,W}(t)\|^p\right]\\
	&\le\sup_{N\in\N}\E\!\left[\sup_{t\in[0,T]}\|\Psi^N_{\xi,W}(t)\|^p\right].
	\end{split}
	\end{equation}
	This and \eqref{eq:lemma_equivalence_lp_sup_numerics} assure that for every $ p \in (0,\infty) $ it holds that
	\begin{equation}\label{eq:lemma_equivalence_lp_sup_exact}
	\sup_{x\in [a,b]^d}
	\left(
	\E\!\left[
	\sup_{t\in [0,T]} \| X^x_t \|_{\R^d}^p
	\right]  
	+
	\E\!\left[ 
	\sup_{t\in [0,T]} \| \bX_t \|_{\R^d}^p
	\right]
	\right)
	< \infty .
	\end{equation}
	Combining 
	\eqref{eq:lemma_equivalence_at_most_polynomial_growth} and the fact that $\forall\,r \in(0,\infty)$, $ a,b\in\R\colon |a+b|^r\le 2^r(|a|^r+|b|^r)$ therefore demonstrates that for every $p\in(0,\infty)$ it holds that
	\begin{equation}\label{eq:lemma_equivalence_uniform_lp_boundedness_Phi}
	\begin{split}
	& 
	\sup_{x\in [a,b]^d}
	\left( \E\!\left[ 
	|\Phi((X^x_t)_{t\in [0,T]})|^p
	\right]
	+
	\E\!\left[ 
	|\Phi((\bX_t)_{t\in [0,T]})|^p
	\right]
	\right) 
	\\&\leq 
	\sup_{x\in [a,b]^d} 
	\left( 
	c^p\,\E\!\left[ 
	|1+\sup\nolimits_{t\in [0,T]} \|X^x_t\|_{\R^d}|^{cp}
	\right]
	+
	c^p\,\E\!\left[ 
	|1+\sup\nolimits_{t\in [0,T]} \|\bX_t\|_{\R^d}|^{cp}
	\right]
	\right) 
	\\&\leq 
	2^{cp}\,c^p
	\left(
	\sup_{x\in [a,b]^d}
	\E\!\left[
	2 
	+ \sup\nolimits_{t\in [0,T]} \|X^x_t\|_{\R^d}^{cp} 
	+ \sup\nolimits_{t\in [0,T]} \|\bX_t\|_{\R^d}^{cp}
	\right]
	\right)
	<\infty .
	\end{split}
	\end{equation}
	This establishes item~\eqref{it:rhs_well_defined_for_every_x}. 
	In the next step we observe that 
	\eqref{eq:lemma_equivalence_sde_for_x} 
	ensures that for every 
	$ x,y \in [a,b]^d$, $ t \in [0,T]$ 
	it holds $\P$-a.s.~that 
	\begin{equation}
	X^x_t - X^y_t 
	= 
	x-y 
	+ 
	\int_0^t \left(\mu(X^x_s)-\mu(X^y_s)\right) ds 
	+ 
	\int_0^t \left(\sigma(X^x_s)-\sigma(X^y_s)\right) dW_s. 
	\end{equation}
	The triangle inequality hence ensures that for every 
	$ x,y \in [a,b]^d$, $t\in[0,T]$ 
	it holds $\P$-a.s.~that
	\begin{equation}
	\begin{split}
	&\sup_{s\in [0,t]}\|X^x_s - X^y_s \|_{\R^d}
	\\
	&\leq 
	\|x-y \|_{\R^d}
	+ 
	\sup_{s\in [0,t]} \int_0^s \|\mu(X^x_r)-\mu(X^y_r)\|_{\R^d} \,dr
	+ 
	\sup_{s\in [0,t]}\left\|\int_0^t
	\left(\sigma(X^x_r)-\sigma(X^y_r)\right) dW_r\right\|_{\R^d}\\
	&\leq \|x-y \|_{\R^d}
	+ 
	L \left[\sup_{s\in [0,t]} \int_0^s \|X^x_r-X^y_r\|_{\R^d} \,dr \right]
	+\sup_{s\in [0,t]}\left\|\int_0^s 
	\left(\sigma(X^x_r)-\sigma(X^y_r)\right) dW_r\right\|_{\R^d}\\
	&= \|x-y \|_{\R^d}
	+ 
	L \int_0^t \|X^x_r-X^y_r\|_{\R^d} \,dr 
	+\sup_{s\in [0,t]}\left\|\int_0^s
	\left(\sigma(X^x_r)-\sigma(X^y_r)\right) dW_r\right\|_{\R^d}.
	\end{split}
	\end{equation}
	Therefore, we obtain for every 
	$ p \in [1,\infty)$, $x,y \in [a,b]^d$, $t\in[0,T]$
	that 
	\begin{equation}\label{eq:supDiffusion}
	\begin{split}
	\left(\E\!\left[\sup_{s\in [0,t]} \|X^x_s-X^y_s\|^p_{\R^d}\right]\right)^{\!\nicefrac{1}{p}}
	& \leq 
	\|x-y\|_{\R^d} 
	+ 
	L \int_0^t \left(\E\!\left[\|X^x_r-X^y_r\|^p_{\R^d}\right]\right)^{\!\nicefrac{1}{p}}dr 
	\\
	&
	+
	\left(\E\!\left[\sup_{s\in [0,t]} 
	\left\|\int_0^s 
	\left(\sigma(X^x_r)-\sigma(X^y_r)\right)\,dW_r\right\|^p_{\R^d}
	\right]\right)^{\!\nicefrac{1}{p}} . 
	\end{split}
	\end{equation}
	The Burkholder-Davis-Gundy type inequality in Da Prato \& Zabczyk~\cite[Lemma 7.2]{DaPratoZabczyk2008} hence shows that for every 
	$ p \in [2,\infty)$, $ x,y \in [a,b]^d$, $t\in[0,T]$ 
	it holds that 
	\begin{equation}
	\begin{split}
	\left(\E\!\left[\sup_{s\in [0,t]} \|X^x_s-X^y_s\|^p_{\R^d}\right]\right)^{\!\nicefrac{1}{p}}
	& \leq  
	\|x-y\|_{\R^d} 
	+ 
	L \int_0^t \left(\E\big[\|X^x_r-X^y_r\|^p_{\R^d}\big]\right)^{\!\nicefrac{1}{p}}dr 
	\\
	&
	\quad +p
	\left[\int_0^t 
	\left(\E\big[\|\sigma(X^x_r)-\sigma(X^y_r)\|_{HS(\R^d,\R^m)}^p\big]\right)^{\!\nicefrac{2}{p}}
	dr \right]^{\nicefrac{1}{2}} . 
	\end{split}
	\end{equation}
	This demonstrates that for every $p\in[2,\infty)$, $x,y\in[a,b]^d$, $t\in[0,T]$ it holds that
	\begin{equation}
	\begin{split}
	\left(\E\!\left[\sup_{s\in [0,t]} \|X^x_s-X^y_s\|^p_{\R^d}\right]\right)^{\!\nicefrac{1}{p}} 
	& \le
	\|x-y\|_{\R^d} 
	+ 
	L \int_0^t \left(\E\big[\|X^x_r-X^y_r\|^p_{\R^d}\big]\right)^{\!\nicefrac{1}{p}}  dr 
	\\
	&
	\quad + Lp
	\left[\int_0^t 
	\left(\E\big[\|X^x_r-X^y_r\|_{\R^d}^p\big]\right)^{\!\nicefrac{2}{p}}
	dr \right]^{\nicefrac{1}{2}}. 
	\end{split}
	\end{equation}
	H{\"o}lder's inequality hence proves that for every $p\in[2,\infty)$, $x,y\in[a,b]^d$, $t\in[0,T]$ it holds that
	\begin{equation}
	\begin{split}
	\left(\E\!\left[\sup_{s\in [0,t]} \|X^x_s-X^y_s\|^p_{\R^d}\right]\right)^{\!\nicefrac{1}{p}} 
	& \le
	\|x-y\|_{\R^d} 
	+ 
	L\sqrt{t}
	\left[\int_0^t 
	\left(\E\big[\|X^x_r-X^y_r\|_{\R^d}^p\big]\right)^{\!\nicefrac{2}{p}}
	dr \right]^{\nicefrac{1}{2}}
	\\
	&
	\quad + Lp
	\left[\int_0^t 
	\left(\E\big[\|X^x_r-X^y_r\|_{\R^d}^p\big]\right)^{\!\nicefrac{2}{p}}
	dr \right]^{\nicefrac{1}{2}}
	\\
	&\le
	\|x-y\|_{\R^d} 
	+ 
	L\big[p+\sqrt{T}\big]
	\left[\int_0^t 
	\left(\E\big[\|X^x_r-X^y_r\|_{\R^d}^p\big]\right)^{\!\nicefrac{2}{p}}
	dr \right]^{\nicefrac{1}{2}}. 
	\end{split}
	\end{equation}
	The fact that $\forall \, v,w\in\R\colon|v+w|^2\le 2v^2+2w^2$ therefore shows that for  every
	$ p \in [2,\infty)$, 
	$ x,y \in [a,b]^d$, 
	$t\in[0,T]$ it holds that 
	\begin{equation}
	\begin{split}
	&\left(\E\!\left[\sup_{s\in [0,t]} \|X^x_s-X^y_s\|^p_{\R^d}\right]\right)^{\!\nicefrac{2}{p}}\\
	&\leq  
	2\|x-y\|_{\R^d}^2 
	+ 
	2L^2\big[p+\sqrt{T}\big]^2\int_0^t 
	\left(\E\big[\|X^x_r-X^y_r\|_{\R^d}^p\big]\right)^{\nicefrac{2}{p}}
	\,dr\\
	&\leq 2\|x-y\|_{\R^d}^2 + 
	2L^2\big[p+\sqrt{T}\big]^2\int_0^t \left(\E\!\left[\sup_{s\in [0,r]}\|X^x_s-X^y_s\|^p_{\R^d}\right]\right)^{\!\nicefrac{2}{p}}dr
	.
	\end{split}
	\end{equation}
	Combining the Gronwall inequality (cf., e.g., Andersson et al.~\cite[Lemma 2.6]{Andersson2015} (with $\alpha=0$, $\beta=0$, $a=2\|x-y\|_{\R^d}^2$, $b=3L^2\big[p+\sqrt{T}\big]$, $e=([0,T]\ni t \mapsto  \left(\E\!\left[\sup_{s\in [0,t]} \|X^x_s-X^y_s\|^p_{\R^d}\right]\right)^{2/p}\in[0,\infty])$ in the notation of Lemma 2.6)) and \eqref{eq:lemma_equivalence_lp_sup_exact} hence establishes that for every 
	$ p \in [2,\infty)$, $ x,y \in [a,b]^d$, $t\in[0,T]$ it holds that 
	\begin{equation}
	\left(\E\!\left[\sup_{s\in [0,t]} \|X^x_s-X^y_s\|^p_{\R^d}\right]\right)^{\!\nicefrac{2}{p}}
	\leq 
	2\|x-y\|_{\R^d}^2 \exp\!\left(2L^2t\big[p+\sqrt{T}\big]^2\right).
	\end{equation}
	Therefore, we obtain that for every $p\in[2,\infty)$, $x,y\in[a,b]^d$ it holds that
	\begin{equation}
	\left(\E\left[\sup_{t\in[0,T]}\|X^x_t-X^y_t\|^p\right]\right)^{\!\nicefrac{1}{p}}\le \sqrt{2}\exp\!\left(L^2T\big[p+\sqrt{T}\big]^2\right)\|x-y\|_{\R^d}.
	\end{equation}
	This establishes item \eqref{it:GWbound}. Next observe that item \eqref{it:GWbound} and Jensen's inequality imply that for every 
	$ p \in (0,\infty)$ it holds that 
	\begin{equation}
	\label{eq:uniform_integrability}
	\sup_{x,y\in [a,b]^d, x\neq y}
	\left( 
	\frac{\left(\E\!\left[\sup_{t\in [0,T]} \|X^x_t-X^y_t\|_{\R^d}^p\right]\right)^{\!\nicefrac{1}{p}}}
	{\|x-y\|_{\R^d}}
	\right)
	< \infty . 
	\end{equation}
	The hypothesis that the function 
	$\Phi\colon C([0,T],\R^d)\to\R$ is continuous and Lemma
	\ref{lemma:stochastic_convergence_under_continuous_transformations} 
	hence ensure that for every $\varepsilon\in (0,\infty)$, $(x_n)_{n\in \N_0}\subseteq  [a,b]^d$ with $\limsup_{n\to \infty} \|x_0-x_n\|_{\R^d}=0$ it holds that 
	\begin{equation}
	\limsup_{n\to \infty} 
	\P\!
	\left( 
	|\Phi((X^{x_0}_t)_{t\in [0,T]}) - \Phi((X^{x_n}_t)_{t\in [0,T]})| \geq \varepsilon 
	\right) 
	= 
	0 . 
	\end{equation}
	Combining \eqref{eq:lemma_equivalence_uniform_lp_boundedness_Phi} with, e.g., Hutzenthaler et al.~\cite[Proposition 4.5]{HutzenthalerJentzenSalimova2016} therefore implies that for every $(x_n)_{n\in \N_0}\subseteq  [a,b]^d$ with $\limsup_{n\to \infty} \|x_0-x_n\|_{\R^d}=0$ it holds that 
	\begin{equation}
	\limsup_{n\to \infty} 
	\E\big[ 
	|\Phi((X^{x_0}_t)_{t\in [0,T]}) - \Phi((X^{x_n}_t)_{t\in [0,T]})|
	\big]  
	= 0.
	\end{equation}
	This establishes item~\eqref{it:rhs_continuity_wrt_x}. In the next step we observe that \eqref{eq:supDiscret1} and \eqref{eq:supDiscret2} ensure that for every 
	$ \varepsilon \in (0,\infty) $, 
	$ x\in [a,b]^d $
	it holds that 
	\begin{equation}\label{eq:euler_scheme_convergence_in_probability}
	\limsup_{N\to\infty} \left[
	\P\!\left(
	\sup_{t\in [0,T]} \|
	\Psi^N_{x,W}(t)
	- 
	X^x_t  \|_{\R^d} 
	\geq \varepsilon \right)
	+
	\P\!\left(
	\sup_{t\in [0,T]} \|
	\Psi^N_{\xi,W}(t)
	- 
	\bX_t  \|_{\R^d} 
	\geq \varepsilon \right)
	\right]
	= 0.
	\end{equation}
	The hypothesis that the function $\Phi$ is continuous and 
	Lemma~\ref{lemma:stochastic_convergence_under_continuous_transformations}
	therefore demonstrate that for every 
	$ \varepsilon\in (0,\infty) $, $ x \in [a,b]^d $ it holds that 
	\begin{equation}\label{eq:stochastic_convergence_combined}
	\limsup_{N\to\infty} 
	\left[
	\P\!
	\left(
	|\Phi(\Psi^N_{x,W}) - \Phi((X^x_t)_{t\in [0,T]})| 
	+ 
	|\Phi(\Psi^N_{\xi,W})-\Phi((\bX_t)_{t\in [0,T]})| 
	\geq \varepsilon
	\right)
	\right]
	= 0 . 
	\end{equation}
	Next observe that 
	\eqref{eq:lemma_equivalence_at_most_polynomial_growth}
	assures that for every 
	$ N \in \N$, 
	$ x \in [a,b]^d $, 
	$ p \in (0,\infty) $ 
	it holds that
	\begin{equation}\label{eq:lemma_equivalence_p_integrability}
	\begin{split}
	&  \E\!\left[ 
	|\Phi(\Psi^N_{ x, W} ) 
	- 
	\Phi( (X^x_t)_{t\in [0,T]} ) |^p   
	+
	|\Phi( \Psi^N_{ \xi, W} ) 
	- 
	\Phi( (\bX_t)_{t\in [0,T]} ) |^p  
	\right] 
	\\[1.5 ex]
	&\leq 
	2^p\,\E\!\left[
	| \Phi( \Psi^N_{ x, W } ) |^p
	+ 
	|\Phi( (X^x_t)_{t\in [0,T]} ) |^p
	\right]
	+
	2^p\,\E\!\left[| \Phi(\Psi^N_{ \xi, W} ) |^p
	+ 
	| \Phi( (\bX_t)_{t\in [0,T]} ) |^p
	\right]
	\\[1.5 ex]
	&\leq 
	2^pc^p \,
	\E\Big[
	\left|
	1+\sup\nolimits_{t\in [0,T]} 
	\| \Psi^N_{ x, W} ( t )  \|_{\R^d}
	\right|^{cp}
	+
	\left|
	1+\sup\nolimits_{t\in [0,T]} \|X^x_t\|_{\R^d}
	\right|^{cp}
	\Big]
	\\[1.5 ex]
	&+ 
	2^pc^p \,
	\E\Big[
	\left|
	1+\sup\nolimits_{t\in [0,T]} 
	\| \Psi^N_{ \xi, W} (t) \|_{\R^d}
	\right|^{cp}
	+
	\left|
	1+\sup\nolimits_{t\in [0,T]} \|\bX_t\|_{\R^d}
	\right|^{cp}
	\Big]
	\\[1.5 ex] 
	&\leq 
	4^pc^p \,
	\E\!\left[
	2 
	+ 
	\sup\nolimits_{t\in [0,T]} \| \Psi^N_{ x, W} ( t ) \|_{\R^d}^{cp}
	+ 
	\sup\nolimits_{t\in [0,T]} \|X^x_t\|_{\R^d}^{cp}
	\right]  \\[1.5 ex]
	&+
	4^pc^p \,
	\E\!\left[
	2 
	+ 
	\sup\nolimits_{t\in [0,T]} \| \Psi^N_{ \xi, W} ( t ) \|_{\R^d}^{cp}
	+ 
	\sup\nolimits_{t\in [0,T]} \|\bX_t\|_{\R^d}^{cp}
	\right].
	\end{split}
	\end{equation} 
	Combining \eqref{eq:lemma_equivalence_lp_sup_numerics} and 
	\eqref{eq:lemma_equivalence_lp_sup_exact}
	hence shows that for every $ p \in (0,\infty) $ it holds that 
	\begin{equation}
	\label{eq:lemma_equivalence_l2_boundedness}
	\sup_{N\in\N}\sup_{x\in [a,b]^d} 
	\left(
	\E\!\left[ 
	|\Phi(\Psi^N_{ x, W} ) 
	- 
	\Phi( (X^x_t)_{t\in [0,T]} ) |^p   
	+
	|\Phi(\Psi^N_{ \xi, W}) 
	- 
	\Phi( (\bX_t)_{t\in [0,T]} ) |^p 
	\right] 
	\right)
	< \infty . 
	\end{equation}
	This, \eqref{eq:stochastic_convergence_combined}, and, e.g., Hutzenthaler et al.~\cite[Proposition 4.5]{HutzenthalerJentzenSalimova2016} imply that for every $ x \in [a,b]^d $ 
	it holds that 
	\begin{equation}\label{eq:lemma_equivalence_first_l1_convergence}
	\limsup_{N\to\infty} 
	\left(
	\E\big[|\Phi(\Psi^N_{x,W}) 
	- 
	\Phi( (X^x_t)_{t\in [0,T]})|
	\big]
	+ 
	\E\big[|\Phi(\Psi^{N}_{\xi, W}) 
	- 
	\Phi( (\bX_t)_{t\in [0,T]})|\big]
	\right) = 0.
	\end{equation}
	Combining \eqref{eq:lemma_equivalence_l2_boundedness}
	with Lebesgue's dominated convergence theorem therefore demonstrates that 
	\begin{equation}\label{eq:lemma_equivalence_second_l1_convergence}
	\limsup_{N\to\infty} 
	\left(\int_{[a,b]^d}  
	\E\big[| \Phi( \Psi^{N}_{ x, W} )
	- 
	\Phi( (X^x_t)_{t\in [0,T]})|\big]\,dx 
	\right)
	= 0. 
	\end{equation}
	In addition, observe that \eqref{eq:lemma_equivalence_p_integrability}, \eqref{eq:lemma_equivalence_lp_sup_numerics}, and \eqref{eq:lemma_equivalence_lp_sup_exact} prove that for all $p\in(0,\infty)$ it holds that
	\begin{equation}\label{eq:Phi_Psi_integrability}
	\sup_{N\in\N} \sup_{x\in[a,b]^d} \E\big[|\Phi(\Psi^N_{x,W})|^p\big]<\infty.
	\end{equation}
	Next observe that~\eqref{eq:lemma_equivalence_first_l1_convergence}
	and the fact that $\xi$ and $W$ are independent imply that 
	\begin{align}
	\label{eq:relationBetweenExpectationsOfbXAndXx_1_new}
	\begin{split}
	\E\!\left[ \Phi( (\bX_t)_{t\in [0,T]} ) \right] 
	&=
	\lim_{ N \rightarrow \infty }
	\E\Big[ \Phi\big(\Psi^{N}_{\xi, W}\big) \Big]
	\\&=
	\lim_{ N \rightarrow \infty }\left[
	\int_{ \Omega }
	\Phi\!\left(\Psi^{N}_{\xi(\omega), (W_t(\omega))_{t\in[0,T]}}\right) \P(d\omega)\right]
	\\&=  
	\lim_{ N \rightarrow \infty }\left[
	\int_{ [a,b]^d \times  C([0,T],\R^d) }
	\Phi\big(\Psi^{N}_{x, W}\big) \, \big( (\xi, W)(\P) \big)(dx,dw)\right]
	\\&=  
	\lim_{ N \rightarrow \infty }\left[
	\int_{ [a,b]^d \times C([0,T],\R^d) }
	\Phi\big(\Psi^{N}_{x, w}\big) \,  \big( (\xi(\P)) \otimes (W(\P)) \big)(dx,dw)\right]. 
	\end{split}
	\end{align}
	Combining Fubini's theorem,  \eqref{eq:lemma_equivalence_second_l1_convergence}, and \eqref{eq:Phi_Psi_integrability} with Lebesgue's dominated convergence theorem
	therefore assures that 
	\begin{align}
	\label{eq:relationBetweenExpectationsOfbXAndXx_2_new}
	\begin{split}
	\E\!\left[ \Phi((\bX_t)_{t\in [0,T]}) \right] 
	&=  
	\lim_{ N \rightarrow \infty }\left[
	\int_{ [a,b]^d }
	\left(
	\int_{  C([0,T],\R^d) }
	\Phi\big(\Psi^{N}_{x, w}\big) \, (W(\P))(dw)
	\right) (\xi(\P))(dx)\right]
	\\&=
	\lim_{ N \rightarrow \infty }\left[
	\int_{ [a,b]^d }
	\left(
	\int_{ \Omega }
	\Phi\big(\Psi^{N}_{x, w}\big) \, \P(dw)
	\right) (\xi(\P))(dx)\right]
	\\&=  
	\lim_{ N \rightarrow \infty }\left[
	\int_{ [a,b]^d }
	\E\Big[ \Phi\big(\Psi^{N}_{x, W}\big) \Big] \, (\xi(\P))(dx)\right]
	\\&=  
	\int_{ [a,b]^d }
	\lim_{ N \rightarrow \infty }
	\E\Big[  \Phi\big(\Psi^{N}_{x, W}\big) \Big] \, (\xi(\P))(dx)
	\\&=
	\int_{[a,b]^d}
	\E\!\left[ \Phi( (X^x_t)_{t\in [0,T]} ) \right]
	(\xi(\P))(dx)
	\\&
	=
	\frac{1}{(b-a)^d} \left(\int_{[a,b]^d} \E\!\left[ \Phi( (X^x_t)_{t\in [0,T]} ) \right] dx\right). 
	\end{split}
	\end{align}
	This establishes item \eqref{it:identity_of_interest}. The proof of Lemma~\ref{lemma:equivalence_of_expectation_representations} 
	is thus completed. 
\end{proof}

\begin{proposition}
	\label{proposition:minimizingPropertyApplied_new}
	Let 
	$d, m\in\N$, 
	$T\in (0,\infty)$, 
	$a\in\R$, 
	$b\in (a,\infty)$, 
	let 
	$\mu\colon \R^d\to\R^d$ and 
	$\sigma\colon \R^d\to\R^{d\times m}$ 
	be globally Lipschitz continuous functions,
	let 
	$\varphi\colon\R^d\to\R$ be a function,
	let 
	$u=(u(t,x))_{(t,x)\in [0,T]\times\R^d}\in C^{1,2}([0,T]\times\R^d,\R)$ 
	be a function with at most polynomially growing partial derivatives which satisfies 
	for every $t\in [0,T]$, $x\in\R^d$ that 
	$u(0,x) = \varphi(x)$ and 
	\begin{equation}
	\label{eq:differentialu}
	\tfrac{\partial u}{\partial t}(t,x) 
	= 
	\tfrac12 \operatorname{Trace}_{\R^d}\!\big(
	\sigma(x)[\sigma(x)]^{*}(\operatorname{Hess}_x u)(t,x)\big)
	+ 
	\langle 
	\mu(x),(\nabla_x u)(t,x)
	\rangle_{\R^d},  
	\end{equation}
	let 
	$(\Omega,\F,\P)$ be a probability space with a normal filtration $(\bF_t)_{t\in [0,T]}$, 
	let
	$W\colon [0,T]\times\Omega\to\R^m$ 
	be a standard $(\bF_t)_{t\in [0,T]}$-Brownian motion, 
	let 
	$\xi\colon\Omega\to [a,b]^d$ be 
	a continuous uniformly distributed 
	$\bF_0$/$\B([a,b]^d)$-measurable random variable, 
	and let 
	$\bX=(\bX_t)_{t\in [0,T]}\colon [0,T]\times\Omega\to\R^d$ 
	be an $(\bF_t)_{t\in [0,T]}$-adapted stochastic process with 
	continuous sample paths which satisfies that for every 
	$t\in [0,T]$ it holds $\P$-a.s.~that
	\begin{equation}\label{eq:SDEforXstartingAtXi_new}
	\bX_t = \xi + \int_0^t \mu(\bX_s)\,ds + \int_0^t \sigma(\bX_s)\,dW_s. 
	\end{equation}  
	Then 
	\begin{enumerate}[(i)]
		\item\label{it:varphiContinuous_new} it holds that the function
		$\varphi \colon \R^d \rightarrow \R$ is twice continuously differentiable
		with at most polynomially growing derivatives,
		\item \label{it:exuniqueSolutions} it holds that there exists a unique continuous function 
		$U\colon [a,b]^d\to\R$ such that 
		\begin{equation}
		\E\big[ | \varphi(\bX_T) - U(\xi) |^2 \big] 
		= 
		\inf_{v\in C([a,b]^d,\R)} \E\big[ | \varphi(\bX_T) - v(\xi) |^2 \big] ,
		\end{equation} 
		and
		\item \label{it:UboundaryCond}it holds for every $x\in [a,b]^d$ that $U(x)=u(T,x)$.
	\end{enumerate}
\end{proposition}

\begin{proof}[Proof of Proposition~\ref{proposition:minimizingPropertyApplied_new}]
	Throughout this proof let
	$X^x=(X^x_t)_{t\in [0,T]}\colon [0,T]\times\Omega\to\R^d$, $x\in [a,b]^d$, 
	be $(\bF_t)_{t\in [0,T]}$-adapted stochastic processes with continuous 
	sample paths 
	\begin{enumerate}[a)]
		\item which satisfy that for every $t\in [0,T]$, $x\in [a,b]^d$ 
		it holds $\P$-a.s.~that 
		\begin{equation}
		\label{eq:SDEPropertyAppliedForx}
		X^x_t = x + \int_0^t \mu(X^x_s)\,ds + \int_0^t \sigma(X^x_s)\,dW_s
		\end{equation}
		and
		\item which satisfy that for every 
		$\omega\in\Omega$ it holds that the function
		$[a,b]^d \ni x \mapsto X^x_T(\omega) \in \R^d$ is continuous (cf., for example, Cox et al.~\cite[Theorem 3.5]{CoxJentzenHutzenthaler2013} and item~\eqref{it:GWbound} in Lemma~\ref{lemma:equivalence_of_expectation_representations}).
	\end{enumerate}
	Note that the assumption that $ \forall \, x \in \R^d \colon u(0,x) = \varphi(x) $
	and the assumption that $u\in C^{1,2}([0,T]\times\R^d,\R)$ has at most polynomially
	growing partial derivatives establish item~\eqref{it:varphiContinuous_new}.
	Next note that item \eqref{it:varphiContinuous_new} and the fact that for every  
	$ p \in (0,\infty)$, $x\in[a,b]^d$ it holds that
	\begin{equation}
		\sup_{t\in [0,T]}
		\E[\|X^x_t\|_{\R^d}^p] 
		< \infty
	\end{equation}
	assure that for every $x\in [a,b]^d$ it holds that 
	\begin{equation}
	\label{eq:minimizingPropertyApplied_1_new}
	\E\big[|\varphi(X^x_T)|^2] < \infty .
	\end{equation}
	Item~\eqref{it:varphiContinuous_new}, the 
	assumption that 
	for every $\omega \in \Omega$ 
	it holds that the function 
	$
	[a,b]^d \ni x 
	\mapsto X^x_T(\omega) 
	\in \R^d 
	$ 
	is continuous, and, e.g., Hutzenthaler et al.~\cite[Proposition 4.5]{HutzenthalerJentzenSalimova2016} hence ensure that the function
	\begin{equation}
	\label{eq:minimizingPropertyApplied_2_new}
	[a,b]^d \ni x 
	\mapsto 
	\E[ \varphi(X^x_T) ] 
	\in \R 
	\end{equation}
	is continuous. In the next step we combine the fact that
	for every 
	$ 
	x \in [a,b]^d 
	$
	it holds that the function 
	$
	\Omega \ni \omega 
	\mapsto 
	\varphi(X^x_T(\omega)) 
	\in \R 
	$
	is $\F$/$\mathcal{B}(\R)$-measurable, the 
	fact that for every 
	$ 
	\omega \in \Omega 
	$
	it holds that the function 
	$
	[a,b]^d \ni x 
	\mapsto \varphi(X^x_T(\omega)) 
	\in \R 
	$ 
	is continuous, and Lemma~\ref{lem:productMeasurable_new}
	to obtain that the function 
	\begin{equation}
	\label{eq:minimizingPropertyApplied_3_new}
	[a,b]^d \times \Omega 
	\ni 
	(x, \omega) 
	\mapsto 
	\varphi(X^x_T(\omega)) 
	\in \R 
	\end{equation}
	is $(\mathcal{B}([a,b]^d) \otimes \F)$/$\mathcal{B}(\R)$-measurable.
	Combining this, 
	\eqref{eq:minimizingPropertyApplied_1_new},  
	\eqref{eq:minimizingPropertyApplied_2_new}, and
	Proposition~\ref{proposition:minimizingProperty}
	demonstrates 
	\begin{enumerate}[A)]
		\item that there exists a unique continuous function 
		$U\colon [a,b]^d\to\R$ which satisfies that 
		\begin{equation}
		\label{eq:minimizingPropertyEqApl_new}
		\int_{[a,b]^d} \E\big[ | \varphi(X^x_T) - U(x) |^2 \big] \,dx
		= 
		\inf_{v\in C([a,b]^d,\R)} 
		\left( 
		\int_{ [a,b]^d } 
		\E\big[ | \varphi(X^x_T) - v(x) |^2 \big] \, dx
		\right)
		\end{equation}
		and
		\item that it holds for every $x\in[a,b]^d$ that
		\begin{equation}
			\label{eq:feynmkac2}
			U(x)=\E[\varphi(X^x_T)].
		\end{equation}
	\end{enumerate}
	Next note that for every continuous function $V\colon[a,b]^d\to\R$ it holds that
	\begin{equation}
		\sup_{x\in[a,b]^d}|V(x)|<\infty.
	\end{equation}
	Item \eqref{it:varphiContinuous_new} hence implies that for every continuous function $V\colon[a,b]^d\to\R$ it holds that
	\begin{equation}
		C([0,T],\R^d) \ni (z_t)_{t\in [0,T]} \mapsto |\varphi(z_T)-V(z_0)|^2 \in \R
	\end{equation}
	is an at most polynomially growing continuous function. Combining Lemma~\ref{lemma:equivalence_of_expectation_representations} (with $\Phi=(C([0,T],\R^d) \ni (z_t)_{t\in [0,T]} \mapsto |\varphi(z_T)-V(z_0)|^2 \in \R)$ for $V\in C([a,b]^d,\R^d)$ in the notation of Lemma~\ref{lemma:equivalence_of_expectation_representations}), \eqref{eq:SDEforXstartingAtXi_new}, item~\eqref{it:varphiContinuous_new}, and \eqref{eq:SDEPropertyAppliedForx} hence ensures that for every continuous function $V\colon[a,b]^d\to\R$ it holds that
	\begin{equation}
		\label{eq:nrRR}
		\E\big[|\varphi(\bX_T)-V(\xi)|^2\big]=\frac{1}{(b-a)^d}\left[\int_{[a,b]^d}\E\big[|\varphi(X^x_T)-V(x)|^2\big]\,dx\right].
	\end{equation}
	Hence, we obtain that for every continuous function $V\colon[a,b]^d\to\R$ with $\E[|\varphi(\bX_T)-V(\xi)|^2]=\inf_{v\in C([a,b]^d,\R)}\E[|\varphi(\bX_T)-v(\xi)|^2]$ it holds that
	\begin{equation}
		\begin{split}
			&\int_{[a,b]^d}\E\big[|\varphi(X^x_T)-V(x)|^2\big]\,dx\\
			&=(b-a)^d\left(\frac{1}{(b-a)^d}\left[\int_{[a,b]^d}\E\big[|\varphi(X^x_T)-V(x)|^2\big]\,dx\right]\right)\\
			&=(b-a)^d\left(\E\big[|\varphi(\bX_T)-V(\xi)|^2\big]\right)\\
			&=(b-a)^d\left(\inf_{v\in C([a,b]^d,\R)}\E\big[|\varphi(\bX_T)-v(\xi)|^2\big]\right)\\
			&=(b-a)^d\left(\inf_{v\in C([a,b]^d,\R)}\left(\frac{1}{(b-a)^d}\left[\int_{[a,b]^d}\E\big[|\varphi(X^x_T)-v(x)|^2\big]\,dx\right]\right)\right)\\
			&=\inf_{v\in C([a,b]^d,\R)}\left[\int_{[a,b]^d}\E\big[|\varphi(X^x_T)-v(x)|^2\big]\,dx\right].
		\end{split}
	\end{equation}
	Combining this with \eqref{eq:minimizingPropertyEqApl_new} proves that for every continuous function $V\colon[a,b]^d\to\R$ with $\E[|\varphi(\bX_T)-V(\xi)|^2]=\inf_{v\in C([a,b]^d,\R)}\E[|\varphi(\bX_T)-v(\xi)|^2]$ it holds that
	\begin{equation}
		\label{eq:uniqueUisV}
		U=V.
	\end{equation}
	Next observe that \eqref{eq:nrRR} and \eqref{eq:minimizingPropertyEqApl_new} demonstrate that 
	\begin{equation}
	\begin{split}
	\E\big[ | \varphi(\bX_T) - U(\xi) |^2 \big] 
	&
	= 
	\frac{1}{(b-a)^d} \left(\int_{[a,b]^d} \E\big[ | \varphi(X^x_T) - U(x) |^2 \big] \,dx \right)
	\\
	& = 
	\inf_{v\in C([a,b]^d,\R)} \bigg[ \frac{1}{(b-a)^d} \left( \int_{[a,b]^d} \E\big[ | \varphi(X^x_T) - v(x) |^2 \big]\,dx \right) \bigg]
	\\
	& = 
	\inf_{v\in C([a,b]^d,\R)} \E\big[ | \varphi(\bX_T) - v(\xi) |^2 \big]. 
	\end{split}
	\end{equation} 
	Combining this with \eqref{eq:uniqueUisV} proves item~\eqref{it:exuniqueSolutions}. 	
	Next note that \eqref{eq:differentialu},  \eqref{eq:feynmkac2}, and the Feynman-Kac formula (cf., for example, Hairer et al.~\cite[Corollary 4.17]{HairerHutzenthalerJentzen_LossOfRegularity2015}) imply that for every $x\in[a,b]^d$ it holds that $U(x)=\E[\varphi(X^x_T)]=u(T,x)$. This establishes item \eqref{it:UboundaryCond}. 
	The proof of Proposition~\ref{proposition:minimizingPropertyApplied_new} is thus completed. 
\end{proof}

In the next step we use Proposition~\ref{proposition:minimizingPropertyApplied_new} 
to obtain a minimization problem which is uniquely solved by the function 
$[a,b]^d\ni x\mapsto u(T,x)\in\R$. More specifically, let 
  $\xi\colon\Omega\to [a,b]^d$ be a continuously 
  uniformly distributed $\bF_0$/$\B([a,b]^d)$-measurable 
  random variable, 
and 
let 
  $\bX\colon [0,T]\times\Omega\to\R^d$ be an $(\bF_t)_{t\in [0,T]}$-adapted 
  stochastic process with continuous sample paths which satisfies that for 
  every $t\in [0,T]$ it holds $\P$-a.s.~that 
  \begin{equation}\label{eq:SDEforXStartingAtXiInSloppyDerivation}
   \bX_t = \xi + \int_0^t \mu(\bX_s)\,ds + \int_0^t \sigma(\bX_s)\,dW_s. 
  \end{equation}  
Proposition~\ref{proposition:minimizingPropertyApplied_new} then guarantees that 
the function $[a,b]^d\ni x\mapsto u(T,x)\in\R$ is the unique global
minimizer of the function 
\begin{equation}\label{eq:functionMinimizedByU}
 C([a,b]^d,\R) \ni v \mapsto \E\big[ | \varphi(\bX_T) - v(\xi) |^2 \big] \in \R. 
\end{equation}
In the following two subsections we derive an 
approximated minimization problem 
by discretizing the stochastic process $\bX\colon [0,T]\times\Omega\to\R^d$ 
(see Subsection~\ref{subsec:discretizationOfX} below)
and 
by employing a deep neural network approximation for the 
function $\R^d\ni x\mapsto u(T,x)\in\R$ 
(see Subsection~\ref{subsec:DNNapproximations} below). 

%%%%%%%%%%%%%%%%%%%%%%%%%%%%%%%%%%%%%%%%%%%%%%%%%%%%%%%%%%%%%%%%%%%%%%%%%%%%
%%%%%%%%%%%%% DISCRETIZATION OF THE AUXILIARY PROCESS %%%%%%%%%%%%%%%%%%%%%%
%%%%%%%%%%%%%%%%%%%%%%%%%%%%%%%%%%%%%%%%%%%%%%%%%%%%%%%%%%%%%%%%%%%%%%%%%%%%

\subsection{Discretization of the stochastic differential equation}
\label{subsec:discretizationOfX}

In this subsection we use the Euler-Maruyama scheme 
(cf., for example, Kloeden \& Platen~\cite{KloedenPlaten1992}
and Maruyama~\cite{Maruyama_ContinuousMarkovProcessesAndStochasticEquations1955}) 
to temporally discretize 
the solution process $\bX$ of the SDE~\eqref{eq:SDEforXStartingAtXiInSloppyDerivation}. 

More specifically, let $N\in\N$, 
let $t_0,t_1,\ldots,t_N\in [0,\infty)$ be real numbers 
which satisfy that 
\begin{equation}
 0 = t_0 < t_1 < \ldots < t_N = T. 
\end{equation}
Note that \eqref{eq:SDEforXStartingAtXiInSloppyDerivation} implies that for every $n\in\{0,1,\ldots,N-1\}$ 
it holds $\P$-a.s.~that
\begin{equation}
 \bX_{t_{n+1}} = \bX_{t_n} + 
 \int_{t_n}^{t_{n+1}} \mu(\bX_s)\,ds + 
 \int_{t_n}^{t_{n+1}} \sigma(\bX_s)\,dW_s.
\end{equation}
This suggests that for sufficiently small mesh size 
$\sup_{n\in\{0,1,\ldots,N-1\}} (t_{n+1}-t_n)$ 
it holds that 
\begin{equation}\label{eq:SDEforXApproximated}
 \bX_{t_{n+1}} \approx \bX_{t_n} + 
 \mu(\bX_{t_n})\,(t_{n+1}-t_n) + 
 \sigma(\bX_{t_n})\,(W_{t_{n+1}}-W_{t_n}).
\end{equation}
Let $\X\colon\{0,1,\ldots,N\}\times\Omega\to\R^d$ 
be the stochastic 
process which satisfies for every $n\in\{0,1,\ldots,N-1\}$ that $\X_0 = \xi$ 
and 
\begin{equation}\label{eq:definitionOfXapproximation}
 \X_{n+1} = \X_n + \mu(\X_n)\,(t_{n+1}-t_n) + \sigma(\X_n)\,(W_{t_{n+1}}-W_{t_n}).
\end{equation}
Observe that \eqref{eq:SDEforXApproximated} and 
\eqref{eq:definitionOfXapproximation} 
suggest, in turn, that for every $n\in\{0,1,2,\ldots,N\}$ it holds 
that 
\begin{equation}
 \X_n \approx \bX_{t_n}
\end{equation}
(cf., for example, Theorem~\ref{thm:strong_convergence_euler_method} below 
for a strong convergence result for the Euler-Maruyama scheme). 

\begin{theorem}[Strong convergence rate for the Euler-Maruyama scheme]
\label{thm:strong_convergence_euler_method}
 Let 
  $T\in (0,\infty)$, $d\in\N$, $p\in [2,\infty)$,
 let 
  $(\Omega,\F,\P)$ be a probability space 
  with a normal filtration $(\bF_t)_{t\in [0,T]}$, 
 let 
  $W\colon [0,T]\times\Omega\to\R^d$ 
  be a standard $(\bF_t)_{t\in [0,T]}$-Brownian 
  motion, 
 let 
  $\xi\colon\Omega\to\R^d$ be a random variable which satisfies 
  that $\E[ \| \xi \|_{\R^d}^p]<\infty$, 
 let 
  $\mu\colon\R^d\to\R^d$ and $\sigma\colon\R^d\to\R^{d\times d}$ 
  be Lipschitz continuous functions, 
 let 
  $\bX\colon [0,T]\times\Omega\to\R^d$ be an $(\bF_t)_{t\in [0,T]}$-adapted 
  stochastic process with continuous sample paths which satisfies that
  for every $t\in [0,T]$ it holds $\P$-a.s.~that 
  \begin{equation}
   \bX_t = \xi + \int_0^t \mu(\bX_s)\,ds + \int_0^t \sigma(\bX_s)\,dW_s, 
  \end{equation}
 for every $N\in\N$ let 
  $t^{N}_0, t^{N}_1, \ldots, t^{N}_N \in [0,T]$ be 
  real numbers which satisfy that 
  \begin{equation}
   0 = t^{N}_0 < t^{N}_1 < \ldots < t^{N}_N = T, 
  \end{equation}
 and for every $N\in\N$ let 
  $\X^{N}\colon \{0,1,\ldots,N\} \times\Omega\to\R^d$ 
  be the stochastic process which satisfies for every $n\in\{0,1,\ldots,N-1\}$ 
  that $\X^{N}_0 = \xi_0$ and 
 \begin{equation}
  \X^{N}_{n+1} 
  = 
  \X^{N}_n 
  + 
  \mu\big(\X^{N}_n\big)(t^{N}_{n+1} - t^{N}_n) 
  + 
  \sigma\big(\X^{N}_n\big)\big(W_{t^{N}_{n+1}} - W_{t^{N}_n}\big). 
 \end{equation}
 Then there exists a real number $C\in (0,\infty)$ such that 
 for every $N\in\N$ it holds that 
 \begin{equation}
  \sup_{n \in \{0,1,\ldots,N\}} 
  \Big( \E\big[ \|\bX_{t_n^N} - \X^{N}_n \|_{\R^d}^p \big]\Big)^{\!\nicefrac{1}{p}}
  \leq 
  C
  \left[ 
    \max_{ n \in \{0,1,\ldots,N-1\} }
    | t_{n+1} - t_n |
  \right]^{\nicefrac{1}{2}}.
 \end{equation}

\end{theorem}

The proof of Theorem~2.4 is well-known in the literature
(cf., for instance, Kloeden \& Platen~\cite{KloedenPlaten1992},
Milstein~\cite{Milstein1995}, 
Hofmann, M{\"u}ller-Gronbach, \& Ritter~\cite{HofmannGronbachRitter2000},
M{\"u}ller-Gronbach \& Ritter~\cite{GronbachRitterMinimal2008},
and the references mentioned therein).

%%%%%%%%%%%%%%%%%%%%%%%%%%%%%%%%%%%%%%%%%%%%%%%%%%%%%%%%%%%%%%%%%%%%%%%%%%%%
%%%%%%%%%%%% DEEP NEURAL NETWORK APPROXIMATIONS %%%%%%%%%%%%%%%%%%%%%%%%%%%%
%%%%%%%%%%%%%%%%%%%%%%%%%%%%%%%%%%%%%%%%%%%%%%%%%%%%%%%%%%%%%%%%%%%%%%%%%%%%

\subsection{Deep artificial neural network approximations}
\label{subsec:DNNapproximations}

In this subsection we employ suitable approximations for the solution
$
 \R^d\ni x\mapsto u(T,x)\in\R
$ of the PDE~\eqref{eq:kolmogorovPDE} at time $T$. 

More specifically, let $ \nu \in \N $ 
and let 
$\bU = (\bU(\theta,x))_{(\theta,x)\in\R^{\nu}\times\R^d} \colon \R^{\nu} \times \R^d \to \R$
be a continuous function. 
For every
  \emph{suitable} 
  $
    \theta\in\R^{\nu}
  $
  and every $x\in[a,b]^d$ 
  we think of 
  $\bU(\theta,x)\in\R$ 
  as an appropriate approximation
  \begin{equation}\label{eq:VThetaApproxVz}
  \bU(\theta,x) \approx u(T,x)
  \end{equation}
  of $u(T,x)$. 
We suggest to choose the function $\bU\colon\R^{\nu}\times\R^d\to\R$ 
as a deep neural network (cf., for example, Bishop~\cite{Bishop_PatternRecognition2016}). 
For instance, 
let $ \mathcal{L}_d \colon \R^d \to \R^d $ be the function 
which satisfies for every $ x = ( x_1, x_2, \dots, x_d ) \in \R^d $ that
\begin{equation}
  \mathcal{L}_d( x ) 
  =
  \left(
    \frac{\exp(x_1)}{\exp(x_1)+1},
    \frac{\exp(x_2)}{\exp(x_2)+1},
    \dots
    ,
    \frac{\exp(x_d)}{\exp(x_d)+1}
  \right)
\end{equation}
(multidimensional version of the standard logistic function),
for every 
  $ k, l \in \N $,
  $ v \in \N_0 = \{0\} \cup \N $,
  $ \theta = ( \theta_1, \dots, \theta_{ \nu } ) \in \R^{ \nu } $
with 
$
  v + l (k + 1 ) \leq \nu
$
let 
$ A^{ \theta, v }_{ k, l } \colon \R^k \to \R^l $ 
be the function which satisfies for every 
$ x = ( x_1, \dots, x_k )\in\R^k $ that
\begin{equation}
 A^{ \theta, v }_{ k, l }( x )
  =
  \left(
    \begin{array}{cccc}
      \theta_{ v + 1 }
    &
      \theta_{ v + 2 }
    &
      \dots
    &
      \theta_{ v + k }
    \\
      \theta_{ v + k + 1 }
    &
      \theta_{ v + k + 2 }
    &
      \dots
    &
      \theta_{ v + 2 k }
    \\
      \theta_{ v + 2 k + 1 }
    &
      \theta_{ v + 2 k + 2 }
    &
      \dots
    &
      \theta_{ v + 3 k }
    \\
      \vdots
    &
      \vdots
    &
      \vdots
    &
      \vdots
    \\
      \theta_{ v + ( l - 1 ) k + 1 }
    &
      \theta_{ v + ( l - 1 ) k + 2 }
    &
      \dots
    &
      \theta_{ v + l k }
    \end{array}
  \right)
  \left(
    \begin{array}{c}
      x_1
    \\
      x_2
    \\
      x_3
    \\
      \vdots 
    \\
      x_k
    \end{array}
  \right)
  +
  \left(
    \begin{array}{c}
      \theta_{ v + k l + 1 }
    \\
      \theta_{ v + k l + 2 }
    \\
      \theta_{ v + k l + 3 }
    \\
      \vdots 
    \\
      \theta_{ v + k l + l }
    \end{array}
  \right),
\end{equation}
let $ s \in \{ 3, 4, 5, 6, \ldots \} $,
assume that $ (s-1) d (d+1) + d + 1 \leq \nu $,
and let 
  $\bU\colon\R^{\nu}\times\R^d\to\R$
  be the function 
  which satisfies for every 
  $\theta\in\R^{\nu}$, 
  $x\in\R^d$
  that  
\begin{equation}\label{eq:example_net}
% \begin{split}
  \bU(\theta,x)
  =
  \big( A^{ \theta, (s-1)d(d+1) }_{ d, 1 } 
  \circ 
  \mathcal{L}_d
  \circ
  A^{ \theta, (s-2)d(d+1) }_{ d, d } 
  \circ
  \ldots
  \circ 
  \mathcal{L}_d
  \circ 
  A^{ \theta, d(d+1) }_{ d, d } 
  \circ 
  \mathcal{L}_d
  \circ 
  A^{ \theta, 0 }_{ d, d } \big)(x)
  .
%  \end{split}
\end{equation}
The function $\bU\colon \R^{\nu}\times\R^d \to \R$ in \eqref{eq:example_net}
describes an artificial neural network with 
$s+1$ layers (1 input layer with $d$ neurons, 
$s-1$ hidden layers with $d$ neurons each, and 
$1$ output layer with $d$ neurons) 
and standard logistic functions as activation functions 
(cf., for instance, Bishop~\cite{Bishop_PatternRecognition2016}).  

\subsection{Stochastic gradient descent-type minimization}
\label{subsec:sgd}

As described in Subsection~\ref{subsec:DNNapproximations} for every 
\emph{suitable} $\theta\in\R^{\nu}$ and every $x\in [a,b]^d$ 
we think of $\bU(\theta,x)\in\R$ as an appropriate approximation 
of $u(T,x)\in\R$. 
In this subsection we intend to find a \emph{suitable} 
  $\theta\in\R^{\nu}$ 
as an approximate minimizer of the function 
\begin{equation}
\label{eq:toMinimizeApprox}
 \R^{\nu} \ni \theta \mapsto 
 \E\big[| \varphi(\X_N) - \bU(\theta,\xi)|^2\big]\in\R. 
\end{equation}
To be more specific, we intend to find an approximate 
minimizer of the function in \eqref{eq:toMinimizeApprox} 
through a stochastic gradient descent-type minimization algorithm
(cf., for instance, Ruder~\cite[Section~4]{Ruder2016}, 
Jentzen et al.~\cite{JentzenKuckuckNeufeldWurstemberger2018},
and the references mentioned therein). 
For this we approximate the derivative of the function 
in \eqref{eq:toMinimizeApprox} by means of the Monte Carlo method. 

More precisely, 
let 
  $\xi^{(m)}\colon \Omega\to [a,b]^d$, $m\in\N_0$, 
  be independent continuously uniformly distributed $\bF_0$/$\B([a,b]^d)$-measurable 
  random variables, 
let 
$W^{(m)}\colon [0,T]\times\Omega \to \R^d$, $m\in\N_0$, 
be independent standard $(\bF_t)_{t\in [0,T]}$-Brownian motions, 
for every $m\in\N_0$ let
$
\X^{(m)} = (\X^{(m)}_n)_{n\in\{0,1,\ldots,N\}}
\colon 
\{0,1,\ldots,N\} \times \Omega \to \R^d
$ 
be the stochastic process which satisfies for 
every $n\in\{0,1,\ldots,N-1\}$ that 
$\X^{(m)}_0 = \xi^{(m)}$ and
\begin{equation}
 \X^{(m)}_{n+1} = \X^{(m)}_n 
 + \mu(\X^{(m)}_n)\,(t_{n+1}-t_n) 
 + \sigma(\X^{(m)}_n)\,(W^{(m)}_{t_{n+1}}-W^{(m)}_{t_n}), 
\end{equation}
let $\gamma \in (0,\infty)$,
and let 
$\Theta\colon\N_0\times\Omega\to\R^{\nu}$ 
be a stochastic process which satisfies for 
every $m\in\N_0$ that 
\begin{equation}
\label{eq:plainGradientDescent}
 \Theta_{m+1} 
 = 
 \Theta_m 
 -
 2\gamma 
 \cdot 
 \big(\bU(\Theta_m, \xi^{(m)})-\varphi(\X^{(m)}_N)\big)
 \cdot 
 (\nabla_{\theta}\bU)(\Theta_m,\xi^{(m)}). 
\end{equation}
Under appropriate hypotheses we think 
for every sufficiently large $ m \in \N $
of the random variable $ \Theta_m \colon \Omega \rightarrow \R^{\nu} $
as a suitable approximation of a local minimum point of the 
function~\eqref{eq:toMinimizeApprox} and we think
for every sufficiently large $ m \in \N $
of the random function
$ [a,b]^d \ni x \mapsto \bU(\Theta_n, x) \in \R $
as a suitable approximation of the function
$ [a,b]^d \ni x \mapsto u(T,x) \in \R $.

\subsection{Description of the algorithm in a special case}
\label{subsec:desc_algo}
In this subsection we give a description of the proposed 
approximation method in a special case, that is, 
we describe the proposed approximation method in the 
specific case where a particular neural network approximation 
is chosen and where the plain-vanilla stochastic gradient 
descent method with a constant learning rate is the employed stochastic
minimization algorithm (cf.\ \eqref{eq:plainGradientDescent} above). 
For a more general description of the proposed approximation 
method we refer the reader to Subsection~\ref{subsec:generalOptDescription} below. 

\begin{algo}
	\label{algo:special}
Let 
  $T,\gamma\in (0,\infty)$, 
  $a\in\R$, $b\in (a,\infty)$, 
  $d,N\in\N$,
  $s\in\{3,4,5,\ldots\}$,
let 
  $\nu = sd(d+1)$,
let 
  $t_0,t_1,\ldots,t_N\in [0,T]$ 
 be real numbers with 
 \begin{equation}
  0 = t_0 < t_1 < \ldots < t_N = T,
 \end{equation}
let 
 $
  \mu \colon \R^d\to\R^d
 $
 and
 $
  \sigma \colon \R^d\to\R^{d\times d}
 $
 be continuous functions, 
let 
 $(\Omega,\F,\P,(\bF_t)_{t\in [0,T]})$ 
 be a filtered probability space, 
let  
  $\xi^{(m)}\colon\Omega\to[a,b]^d$, 
  $m\in\N_0$,
  be independent continuously uniformly distributed 
  $\bF_0$/$\B([a,b]^d)$-measurable random variables,
let
  $W^{(m)}\colon [0,T]\times\Omega\to\R^d$, 
  $m\in\N_0$, be i.i.d.~standard
  $(\bF_t)_{t\in [0,T]}$-Brownian motions, 
for every 
  $m\in\N_0$ 
  let 
  $\X^{(m)}\colon \{0,1,\ldots,N\}\times\Omega\to\R^d$ 
  be the stochastic process which satisfies for every $n\in\{0,1,\ldots,N-1\}$ 
  that 
  $\X^{(m)}_0 = \xi^{(m)}$ and 
  \begin{equation}
   \X^{(m)}_{n+1} 
   = 
   \X^{(m)}_n 
   + 
   \mu(\X^{(m)}_n)\,(t_{n+1}-t_{n})
   + 
   \sigma(\X^{(m)}_n)\,(W^{(m)}_{t_{n+1}}-W^{(m)}_{t_{n}}), 
  \end{equation}  
let $ \mathcal{L}_d \colon \R^d \to \R^d $ be the function 
which satisfies for every $ x = ( x_1, x_2, \dots, x_d ) \in \R^d $ that
\begin{equation}
\label{eq:activation}
  \mathcal{L}_d( x ) 
  =
  \left(
    \frac{\exp(x_1)}{\exp(x_1)+1},
    \frac{\exp(x_2)}{\exp(x_2)+1},
    \dots
    ,
    \frac{\exp(x_d)}{\exp(x_d)+1}
  \right)
  ,
\end{equation}
for every 
  $ k, l \in \N $,
  $ v \in \N_0 = \{0\} \cup \N $,
  $ \theta = ( \theta_1, \dots, \theta_{ \nu } ) \in \R^{ \nu } $
with 
$
  v + l (k + 1 ) \leq \nu
$
let 
$ A^{ \theta, v }_{ k, l } \colon \R^k \to \R^l $ 
be the function which satisfies for every 
$ x = ( x_1, \dots, x_k )\in\R^k $ that
\begin{equation}
 A^{ \theta, v }_{ k, l }( x )
  =
  \left(
    \begin{array}{cccc}
      \theta_{ v + 1 }
    &
      \theta_{ v + 2 }
    &
      \dots
    &
      \theta_{ v + k }
    \\
      \theta_{ v + k + 1 }
    &
      \theta_{ v + k + 2 }
    &
      \dots
    &
      \theta_{ v + 2 k }
    \\
      \theta_{ v + 2 k + 1 }
    &
      \theta_{ v + 2 k + 2 }
    &
      \dots
    &
      \theta_{ v + 3 k }
    \\
      \vdots
    &
      \vdots
    &
      \vdots
    &
      \vdots
    \\
      \theta_{ v + ( l - 1 ) k + 1 }
    &
      \theta_{ v + ( l - 1 ) k + 2 }
    &
      \dots
    &
      \theta_{ v + l k }
    \end{array}
  \right)
  \left(
    \begin{array}{c}
      x_1
    \\
      x_2
    \\
      x_3
    \\
      \vdots 
    \\
      x_k
    \end{array}
  \right)
  +
  \left(
    \begin{array}{c}
      \theta_{ v + k l + 1 }
    \\
      \theta_{ v + k l + 2 }
    \\
      \theta_{ v + k l + 3 }
    \\
      \vdots 
    \\
      \theta_{ v + k l + l }
    \end{array}
  \right),
\end{equation}
let 
  $\bU\colon\R^{\nu}\times\R^d\to\R$
  be the function 
  which satisfies for every 
  $\theta\in\R^{\nu}$, 
  $x\in\R^d$
  that  
\begin{equation}
% \begin{split}
\label{eq:neural_network_for_a_generalCase}
  \bU(\theta,x)
  =
  \big( A^{ \theta, (s-1)d(d+1) }_{ d, 1 } 
  \circ 
  \mathcal{L}_d
  \circ
  A^{ \theta, (s-2)d(d+1) }_{ d, d } 
  \circ
  \ldots
  \circ 
  \mathcal{L}_d
  \circ 
  A^{ \theta, d(d+1) }_{ d, d } 
  \circ 
  \mathcal{L}_d
  \circ 
  A^{ \theta, 0 }_{ d, d } \big)(x)
  ,
%  \end{split}
\end{equation}
and let 
  $\Theta\colon \N_0\times\Omega\to\R^{\nu}$ 
  be a stochastic process which satisfies for 
  every $m\in\N_0$ 
  that 
  \begin{equation}
   \Theta_{m+1} 
   = 
   \Theta_m 
   - 
   2\gamma \cdot 
   \big(\bU(\Theta_m,\xi^{(m)})-\varphi(\X_N^{(m)})\big) \cdot (\nabla_{\theta} \bU)(\Theta_m,\xi^{(m)})
  \end{equation}
\end{algo}
Under appropriate hypotheses we think for every
sufficiently large $ m \in \N $ and every $ x \in [a,b]^d $
of the random variable $\bU(\Theta_m,x) \colon \Omega \rightarrow \R $ in Framework~\ref{algo:special} 
as a suitable approximation $\bU(\Theta_m,x)\approx u(T,x)$ 
of $u(T,x)\in\R$ where $u=u(t,x)_{(t,x)\in[0,T]\times\R^d}\in C^{1,2}([0,T]\times\R^d,\R)$ is a function with at most polynomially growing partial derivatives which satisfies for every $t\in[0,T]$, $x\in \R^d$ that $u(0,x)=\varphi(x)$ and
\begin{equation}
\tfrac{\partial u}{\partial t}(t,x) 
= 
\tfrac12 \operatorname{Trace}_{\R^d}\!\big(
\sigma(x)[\sigma(x)]^{*}(\operatorname{Hess}_x u)(t,x)\big)
+ 
\langle 
\mu(x),(\nabla_x u)(t,x)
\rangle_{\R^d} 
\end{equation}
(cf.~\eqref{eq:kolmogorovPDE} above).

\subsection{Description of the algorithm in the general case}
\label{subsec:generalOptDescription}

In this subsection we provide a general framework 
which covers the approximation method derived in 
Subsections 
\ref{subsec:kolmogorovEq}--\ref{subsec:sgd} above and 
which allows, in addition, to incorporate other 
minimization algorithms (cf., for example, 
Kingma \& Ba~\cite{KingmaBa2015}, Ruder~\cite{Ruder2016},
E et al.~\cite{EHanJentzen2017a}, Han et al.~\cite{EHanJentzen2017b}, 
and Beck et al.~\cite{BeckEJentzen2017}) than just the plain vanilla 
stochastic gradient descent method. The proposed approximation 
algorithm is an extension of the approximation algorithm 
in E et al.~\cite{EHanJentzen2017a},
Han et al.~\cite{EHanJentzen2017b}, and Beck et al.~\cite{BeckEJentzen2017} 
in the special case of linear 
Kolmogorov partial differential equations. 
\begin{algo}
\label{algo:general_algorithm}
Let 
  $T \in (0,\infty)$,  
  $N, d, \varrho, \nu, \varsigma \in \N$, 
let
  $H\colon [0,T]^2\times\R^d\times\R^d\to\R^d$, $\varphi\colon \R^d\to\R$ 
  be functions,
let 
$ 
  ( \Omega, \F, \P, ( \bF_t )_{ t \in [0,T] } ) 
$ 
be a filtered probability space, 
let 
  $
    W^{m,j} \colon [0,T] \times \Omega \to \R^d 
  $, 
  $
    m\in\N_0
  $, 
  $
    j\in \N
  $,
  be independent standard $ ( \bF_t )_{ t \in [0,T] } $-Brownian motions on 
  $(\Omega,\F,\P)$, 
  let 
  $
    \xi^{m,j}\colon\Omega\to\R^d
  $, 
  $
    m \in \N_0 
  $, 
  $ j \in \N $,
  be i.i.d.\ $ \bF_0 $/$ \B(\R^d) $-measurable random variables,
let
  $t_0, t_1, \ldots, t_N\in [0,T]$ be 
  real numbers with 
  $0 = t_0 < t_1 < \ldots < t_N = T$, 
for every 
  $\theta\in\R^{\nu}$, 
  $j\in\N$,
  ${\bf s}\in\R^{\varsigma}$
  let
  $ \bU^{\theta,j,{\bf s}}\colon\R^d \to\R $ 
  be a function, 
for every 
  $ m \in \N_0 $, 
  $ j \in \N $ 
  let 
  $\X^{m,j} = (\X^{m,j}_n)_{n\in\{0,1,\ldots,N\}}\colon \{0,1,\ldots,N\}\times\Omega\to\R^d$ 
  be a stochastic process which satisfies for every 
  $n\in\{0,1,\ldots,N-1\}$ that 
  $\X^{m,j}_0 = \xi^{ m, j } $ and 
  \begin{equation}\label{eq:FormalXapprox}
   \X^{m,j}_{n+1} 
   = 
   H(t_n,t_{n+1},\X^{m,j}_n,W^{m,j}_{t_{n+1}}-W^{m,j}_{t_{n}}), 
  \end{equation}  
let $(J_m)_{m\in\N_0}\subseteq\N$ be a sequence, 
for every 
  $m\in\N_0$, 
  ${\bf s}\in\R^{\varsigma}$ 
  let 
  $\phi^{m,{\bf s}}\colon\R^{\nu}\times\Omega\to\R$ be 
  the function which satisfies for every 
  $(\theta,\omega)\in\R^{\nu}\times\Omega$ that 
  \begin{equation}
   \phi^{m,{\bf s}}(\theta,\omega) 
   = 
   \frac{1}{J_m}\sum_{j=1}^{J_m}
   \left[ 
    \bU^{\theta,j,{\bf s}}\big(\xi^{m,j}(\omega))
    - 
    \varphi\big(\X^{m,j}_N(\omega)\big) 
   \right]^2,  
  \end{equation}  
for every 
  $m\in\N_0$, 
  ${\bf s}\in\R^{\varsigma}$ 
  let 
  $\Psi^{m,{\bf s}}\colon\R^{\nu}\times\Omega\to\R^{\nu}$ be a function 
  which satisfies for every 
  $\omega\in\Omega$,  
  $\theta\in\{\eta\in\R^{\nu}\colon \phi^{m,{\bf s}}(\cdot,\omega)\colon\R^{\nu}\to\R~\text{is differentiable at}~\eta\}$ 
  that
  \begin{align}
   \Psi^{m,{\bf s}}(\theta,\omega) = (\nabla_{\theta}\phi^{m,{\bf s}})(\theta,\omega),
  \end{align}
let 
  $
    \S
      \colon
    \R^{\varsigma}
      \times
    \R^{\nu}
      \times
    (\R^d)^{\N}\to\R^{\varsigma}
   $ 
  be a function, 
for every 
  $m\in\N_0$
  let $\Phi_m\colon\R^{\varrho}\to\R^{\nu}$ 
  and $\Psi_m\colon\R^{\varrho}\times\R^{\nu}\to\R^{\varrho}$
  be functions, 
let
  $\Theta\colon\N_0\times\Omega\to\R^{\nu}$, 
  $\bS\colon\N_0\times\Omega\to\R^{\varsigma}$, 
  and 
  $\Xi\colon\N_0\times\Omega\to\R^{\varrho}$ 
  be stochastic processes 
  which satisfy for every $m\in\N_0$ that 
  \begin{equation}\label{eq:general_batch_normalization} 
   \bS_{m+1} = \S\bigl(\bS_m, \Theta_{m}, 
   (\X_N^{m,i})_{i\in\N}\bigr),  
   \qquad 
   \Xi_{m+1} = \Psi_{m}(\Xi_{m},\Psi^{m,\bS_{m+1}}(\Theta_m)), 
  \end{equation}
  \begin{equation}
   \text{and}
   \qquad
   \Theta_{m+1} = \Theta_{m} - \Phi_{m}(\Xi_{m+1}) 
   \label{eq:general_gradient_step}. 
  \end{equation}
\end{algo}
Under appropriate hypotheses we think for every
sufficiently large $ m \in \N $ and every $ x \in [a,b]^d $
of the random variable $\bU^{\Theta_m,1,\bS_m}(x) \colon \Omega \rightarrow \R $ in Framework~\ref{algo:general_algorithm} 
as a suitable approximation $\bU^{\Theta_m,1,\bS_m}(x)\approx u(T,x)$ 
of $u(T,x)\in\R$ where $u=u(t,x)_{(t,x)\in[0,T]\times\R^d}\in C^{1,2}([0,T]\times\R^d,\R)$ is a function with at most polynomially growing partial derivatives which satisfies for every $t\in[0,T]$, $x\in \R^d$ that $u(0,x)=\varphi(x)$ and
\begin{equation}
\tfrac{\partial u}{\partial t}(t,x) 
= 
\tfrac12 \operatorname{Trace}_{\R^d}\!\big(
\sigma(x)[\sigma(x)]^{*}(\operatorname{Hess}_x u)(t,x)\big)
+ 
\langle 
\mu(x),(\nabla_x u)(t,x)
\rangle_{\R^d} ,
\end{equation}where $\mu\colon \R^d\to\R^d$ and $\sigma\colon \R^d\to\R^{d\times d}$ are sufficiently regular functions 
(cf.~\eqref{eq:kolmogorovPDE} above).

\section{Examples}
\label{sec:examples}

In this section we test the proposed approximation algorithm (see Section~\ref{sec:derivationalgo} above) 
in the case of several examples of SDEs and Kolmogorov PDEs, respectively. In particular, in this section we apply the proposed approximation algorithm to the heat equation (cf.\ Subsection~\ref{sec:paraboliceq} below), to  independent geometric Brownian motions (cf.\ Subsection~\ref{sec:geometric} below), to the Black-Scholes model (cf.\ Subsection~\ref{sec:black_scholes} below), to stochastic Lorenz equations (cf.\ Subsection~\ref{sec:lorenz} below), and to the Heston model (cf.\ Subsection~\ref{sec:heston} below). In the case of each of the examples below we employ the general approximation algorithm in Framework~\ref{algo:general_algorithm} above 
  in conjunction with the Adam optimizer (cf.\ Kingma \& Ba~\cite{KingmaBa2015}) with mini-batches of size 8192 in each iteration step (see Subsection~\ref{sec:adam} below for a precise description). 
  Moreover, we employ a fully-connected feedforward neural network with one input layer, two hidden layers, and one one-dimensional output layer in our implementations in the case of each of these examples. 
  We also use batch normalization (cf.\ Ioffe \& Szegedy~\cite{IoffeSzegedy}) just before the first linear transformation, just before each of the two nonlinear activation functions in front of the hidden layers as well as just after the last linear transformation.
  For the two nonlinear activation functions we employ the multidimensional version of the function $\R\ni x\mapsto \tanh(x)\in (-1,1)$. All weights in the neural network are initialized by means of the Xavier initialization (cf.\ Glorot \& Bengio~\cite{GlorotBengio10}). 
All computations were performed in single precision (float32) on a NVIDIA GeForce GTX 1080 GPU with 1974 MHz core clock and 8 GB GDDR5X memory with 1809.5 MHz clock rate. The underlying system consisted of an
Intel Core i7-6800K CPU with 64 GB DDR4-2133 memory running Tensorflow 1.5 on Ubuntu 16.04.

\subsection{Setting}
\label{sec:adam}
\begin{framework}
	\label{framework:adam}
	Assume Framework~\ref{algo:general_algorithm}, let $\varepsilon=10^{-8}$, $\beta_1=\frac{9}{10}$, $\beta_1=\frac{999}{1000}$, $(\gamma_m)_{m\in\N_0}\subseteq (0,\infty)$, let  $\operatorname{Pow}_r\colon \R^\nu \to \R^\nu$, $r\in(0,\infty)$, be the functions which satisfy for every $r\in(0,\infty)$, $x=(x_1,\ldots,x_\nu)\in\R^\nu$ that
	\begin{equation}
		\operatorname{Pow}_r(x)=(|x_1|^r, \ldots, |x_\nu|^r),
	\end{equation}
	assume for every $m\in\N_0$, $i\in\{0,1,\ldots,N\}$ that $J_m=8192$, $t_i=\frac{iT}{N}$, $\varrho=2\nu$, $ T = 1$, $\gamma_m =10^{-3}\mathbbm{1}_{[0,250000]}(m)+10^{-4}\mathbbm{1}_{(250000,500000]}(m)+10^{-5}\mathbbm{1}_{(500000,\infty)}(m)$, assume for every $m\in\N_0$, $x=(x_1,\ldots,x_\nu)$, $y=(y_1,\ldots,y_\nu)$, $\eta=(\eta_1,\ldots,\eta_\nu)\in\R^\nu$ that
	\begin{equation}\label{eq:Adam1}
		\Psi_m(x,y,\eta) = (\beta_1 x + (1-\beta_1)\eta, \beta_2 y + (1-\beta_2)\operatorname{Pow}_2(\eta))
	\end{equation}
	and
	\begin{equation}\label{eq:Adam2}
		\psi_m(x,y) = \left(\left[\sqrt{\tfrac{|y_1|}{1-(\beta_2)^m}}+\varepsilon\right]^{-1}\frac{\gamma_mx_1}{1-(\beta_1)^m},\ldots,\left[\sqrt{\tfrac{|y_\nu|}{1-(\beta_2)^m}}+\varepsilon\right]^{-1}\frac{\gamma_mx_\nu}{1-(\beta_1)^m}\right).
	\end{equation}
\end{framework}
	Equations~\eqref{eq:Adam1} and \eqref{eq:Adam2} in Framework~\ref{framework:adam} describe the Adam optimizer (cf.\ Kingma \& Ba~\cite{KingmaBa2015}, e.g., E et al.~\cite[(32)--(33) in Section 4.2 and (90)--(91) in Section 5.2]{EHanJentzen2017b}, and line 84 in {\sc{Python}} code \ref{code:common} in Section~\ref{sec:source_code} below). 

\subsection{Heat equation}
\label{sec:paraboliceq}
In this subsection we apply the proposed approximation algorithm 
to the heat equation (see \eqref{eq:example_constant_coeffs_kolmogorovPDE} below).

Assume 
  Framework~\ref{algo:general_algorithm}, assume for every 
  $s,t\in [0,T]$, 
  $x,w\in\R^d$, 
  $m\in\N_0$
  that 
  $ N = 1$, 
  $ d = 100$, 
  $ \nu = d(2d)+(2d)^2+2d= 2d(3d+1)$,  
  $ \varphi(x) = \|x\|_{\R^d}^2 $, 
  $ H(s,t,x,w) = x + \sqrt{2}\operatorname{Id}_{\R^d}w $, 
  assume that
  $ 
  \xi^{0,1}\colon\Omega\to\R^d
  $
  is continuous uniformly distributed on $[0,1]^d$,
and let 
  $u = (u(t,x))_{(t,x)\in [0,T]\times\R^d}
  \in C^{1,2}([0,T]\times\R^d,\R)$
  be an at most polynomially growing function which satisfies 
  for every $t\in [0,T]$, $x\in\R^d$ that 
  $u(0,x) = \varphi(x)$ and 
  \begin{equation}\label{eq:example_constant_coeffs_kolmogorovPDE}
   (\tfrac{\partial u}{\partial t})(t,x) 
   = 
 (\Delta_x u)(t,x). 
  \end{equation}
Combining, 
e.g., Lemma~\ref{lemma:example_constant_coeffs_exact_solution} below with, 
e.g., Hairer et al.~\cite[Corollary 4.17 and Remark 4.1]
  {HairerHutzenthalerJentzen_LossOfRegularity2015} shows that for every 
$t\in [0,T]$, $x\in\R^d$ it holds that
\begin{equation}
\label{eq:constant_coeffs_concrete_PDE}
 u(t,x) = \|x\|_{\R^d}^2 + t\,d. 
\end{equation}
Table~\ref{tab:constant_coeffs_linfty_l1_l2} approximately presents  
the relative $L^1(\lambda_{[0,1]^d};\R)$-approximation error associated to 
$(\bU^{\Theta_m,1,\bS_m}(x))_{x\in[0,1]^d}$ (see \eqref{eq:relL1fehler} below),
the relative $L^2(\lambda_{[0,1]^d};\R)$-approximation error associated to 
$(\bU^{\Theta_m,1,\bS_m}(x))_{x\in[0,1]^d}$ (see \eqref{eq:relL2fehler} below), 
and 
the relative $L^{\infty}(\lambda_{[0,1]^d};\R)$-approximation error associated to 
$(\bU^{\Theta_m,1,\bS_m}(x))_{x\in[0,1]^d}$ (see \eqref{eq:relLinftyfehler} below)  against 
$ m \in \{0,10000, \allowbreak50000,\allowbreak 100000, 150000, 200000, 500000, \allowbreak 750000\}$ (cf.\ \textsc{Python} code~\ref{code:paraboliceq} in Subsection~\ref{sec:code_paraboliceq} below). 
Figure~\ref{fig:constant_coeffs} approximately depicts 
the relative $L^1(\lambda_{[0,1]^d};\R)$-approximation error associated to 
$(\bU^{\Theta_m,1,\bS_m}(x))_{x\in[0,1]^d}$ (see \eqref{eq:relL1fehler} below),
the relative $L^2(\lambda_{[0,1]^d};\R)$-approximation error associated to 
$(\bU^{\Theta_m,1,\bS_m}(x))_{x\in[0,1]^d}$ (see \eqref{eq:relL2fehler} below), 
and 
the relative $L^{\infty}(\lambda_{[0,1]^d};\R)$-approximation error associated to 
$(\bU^{\Theta_m,1,\bS_m}(x))_{x\in[0,1]^d}$ (see \eqref{eq:relLinftyfehler} below)  against 
$ m \in \{0,100, 200, 300,\ldots,\allowbreak 299800, 299900, 300000\}$ (cf.\ \textsc{Python} code~\ref{code:paraboliceq} in Subsection~\ref{sec:code_paraboliceq} below). 
In our numerical simulations for  Table~\ref{tab:constant_coeffs_linfty_l1_l2} and Figure \ref{fig:constant_coeffs} we calculated the exact solution 
of the PDE~\eqref{eq:example_constant_coeffs_kolmogorovPDE}
by means of Lemma~\ref{lemma:example_constant_coeffs_exact_solution} 
below (see \eqref{eq:constant_coeffs_concrete_PDE} above), we approximately calculated the relative $L^1(\lambda_{[0,1]^d};\R)$-approximation error
\begin{equation}
\label{eq:relL1fehler}
\int_{[0,1]^d} \left|\frac{u(T,x) - \bU^{\Theta_m,1,\bS_m}(x)}{u(T,x)}\right|dx
\end{equation}
for $ m \in \{0,10000, 50000, 100000, 150000, 200000, 500000, 750000\}$ by means of Monte Carlo approximations with 10240000 samples, we approximately calculated the relative $L^2(\lambda_{[0,1]^d};\R)$-approximation error
\begin{equation}
\label{eq:relL2fehler}
\sqrt{\int_{[0,1]^d} \left|\frac{u(T,x) - \bU^{\Theta_m,1,\bS_m}(x)}{u(T,x)}\right|^2dx}
\end{equation}
for $ m \in \{0,10000, 50000, 100000, 150000, 200000, 500000, 750000\}$ by means of Monte Carlo approximations with 10240000 samples, and we approximately calculated the relative $L^\infty(\lambda_{[0,1]^d};\R)$-approximation error
\begin{equation}
\label{eq:relLinftyfehler}
 \sup_{x\in [0,1]^d} \left|\frac{u(T,x) -  \bU^{\Theta_m,1,\bS_m}(x)}{u(T,x)}\right|
\end{equation}
for $ m \in \{0,10000, 50000, 100000, 150000, 200000, 500000, 750000\}$ by means of Monte Carlo approximations with 10240000 samples (see Lemma~\ref{lemma:MonteCarloSupEstimate} below).
Table~\ref{tab:constant_coeffs_linfty_l1_l2_2} 
approximately presents
the relative $L^1(\P;L^1(\lambda_{[0,1]^d};\R))$-approximation error associated to 
$(\bU^{\Theta_m,1,\bS_m(x)})_{x\in[0,1]^d}$ (see \eqref{eq:relL1fehler2} below),
the relative $L^2(\P;L^2(\lambda_{[0,1]^d};\R))$-approximation error associated to 
$(\bU^{\Theta_m,1,\bS_m(x)})_{x\in[0,1]^d}$ (see \eqref{eq:relL2fehler2} below),
and 
the relative $L^2(\P;L^\infty(\lambda_{[0,1]^d};\R))$-approximation error associated to 
$(\bU^{\Theta_m,1,\bS_m(x)})_{x\in[0,1]^d}$ (see \eqref{eq:relLinftyfehler2} below), against 
$ m \in \{0, 10000, 50000,\allowbreak 100000,\allowbreak  150000, 200000,\allowbreak 500000,\allowbreak 750000\}$ (cf.\ \textsc{Python} code~\ref{code:paraboliceq} in Subsection~\ref{sec:code_paraboliceq} below). 
In our numerical simulations for Table~\ref{tab:constant_coeffs_linfty_l1_l2_2} we calculated the exact solution 
of the PDE~\eqref{eq:example_constant_coeffs_kolmogorovPDE}
by means of Lemma~\ref{lemma:example_constant_coeffs_exact_solution} 
below (see \eqref{eq:constant_coeffs_concrete_PDE} above), we approximately calculated the relative $L^1(\P;L^1(\lambda_{[0,1]^d};\R))$-approximation error
\begin{equation}
\label{eq:relL1fehler2}
\E\Bigg[\int_{[0,1]^d} \left|\frac{u(T,x) - \bU^{\Theta_m,1,\bS_m}(x)}{u(T,x)}\right|dx\Bigg]
\end{equation}
for $ m \in \{0, 10000, 50000, 100000, 150000, 200000, 500000, 750000\}$ by means of Monte Carlo approximations with 10240000 samples for the Lebesgue integral and 5 samples for the expectation, we approximately calculated the relative $L^2(\P;L^2(\lambda_{[0,1]^d};\R))$-approximation error
\begin{equation}
\label{eq:relL2fehler2}
\left(\E\!\left[\int_{[0,1]^d} \left|\frac{u(T,x) - \bU^{\Theta_m,1,\bS_m}(x)}{u(T,x)}\right|^2dx\right]\right)^{\!1/2}
\end{equation}
for $ m \in \{0, 10000, 50000, 100000, 150000, 200000, 500000, 750000\}$ by means of Monte Carlo approximations with 10240000 samples for the Lebesgue integral and 5 samples for the expectation, and we approximately calculated the relative $L^2(\P;L^\infty(\lambda_{[0,1]^d};\R))$-approximation error
\begin{equation}
\label{eq:relLinftyfehler2}
\left(\E\!\left[\sup_{x\in [0,1]^d} \left|\frac{u(T,x) - \bU^{\Theta_m,1,\bS_m}(x)}{u(T,x)}\right|^2\right]\right)^{\!1/2}
\end{equation}
for $ m \in \{0, 10000, 50000, 100000, 150000, 200000, 500000, 750000\}$ by means of Monte Carlo approximations with 10240000  samples for the supremum (see Lemma~\ref{lemma:MonteCarloSupEstimate} below) and 5 samples for the expectation. The following elementary result, Lemma~\ref{lemma:example_constant_coeffs_exact_solution} below, specifies the explicit solution of the PDE~\eqref{eq:example_constant_coeffs_kolmogorovPDE} above (cf.\ \eqref{eq:constant_coeffs_concrete_PDE} above). For completeness we also provide here a proof for  Lemma~\ref{lemma:example_constant_coeffs_exact_solution}.
\begin{table}[H]	\begin{center}
		\begin{tabular}{|c|c|c|c|c|}
			\hline
			{\begin{tabular}{@{}c@{}}Number \\ of steps\end{tabular}}  & {\begin{tabular}{@{}c@{}}Relative \\ $L^1(\lambda_{[0,1]^d};\R)$-error\end{tabular}} &  {\begin{tabular}{@{}c@{}}Relative \\ $L^2(\lambda_{[0,1]^d};\R)$-error\end{tabular}} & {\begin{tabular}{@{}c@{}}Relative \\ $L^\infty(\lambda_{[0,1]^d};\R)$-error\end{tabular}} & {\begin{tabular}{@{}c@{}}Runtime \\ in seconds\end{tabular}}\\
			\hline
			0  & 0.998253 &  0.998254 &  1.003524 & 0.5\\
			\hline
			10000  & 0.957464 &  0.957536 &  0.993083 &  44.6\\
			\hline
			50000  &  0.786743 &  0.786806 &  0.828184 &   220.8 \\
			\hline
			100000 &  0.574013 &  0.574060 &  0.605283 & 440.8\\
			\hline
			150000 &   0.361564 &  0.361594 &  0.384105 & 661.0 \\
			\hline
			200000 &  0.150346  &  0.150362 &   0.164140 & 880.8  \\
			\hline
			500000 &  0.000882 &  0.001112 &  0.007360 & 2200.7 \\
			\hline
			750000 &  0.000822 &  0.001036 & 0.007423 & 3300.6 \\
			\hline
		\end{tabular}
		\caption{Approximative presentations of the relative approximation errors in \eqref{eq:relL1fehler}--\eqref{eq:relLinftyfehler} for the heat equation in \eqref{eq:example_constant_coeffs_kolmogorovPDE}.}
		\label{tab:constant_coeffs_linfty_l1_l2}
	\end{center}
\end{table}
\begin{figure}
	\centering
	\includegraphics[scale=0.65]{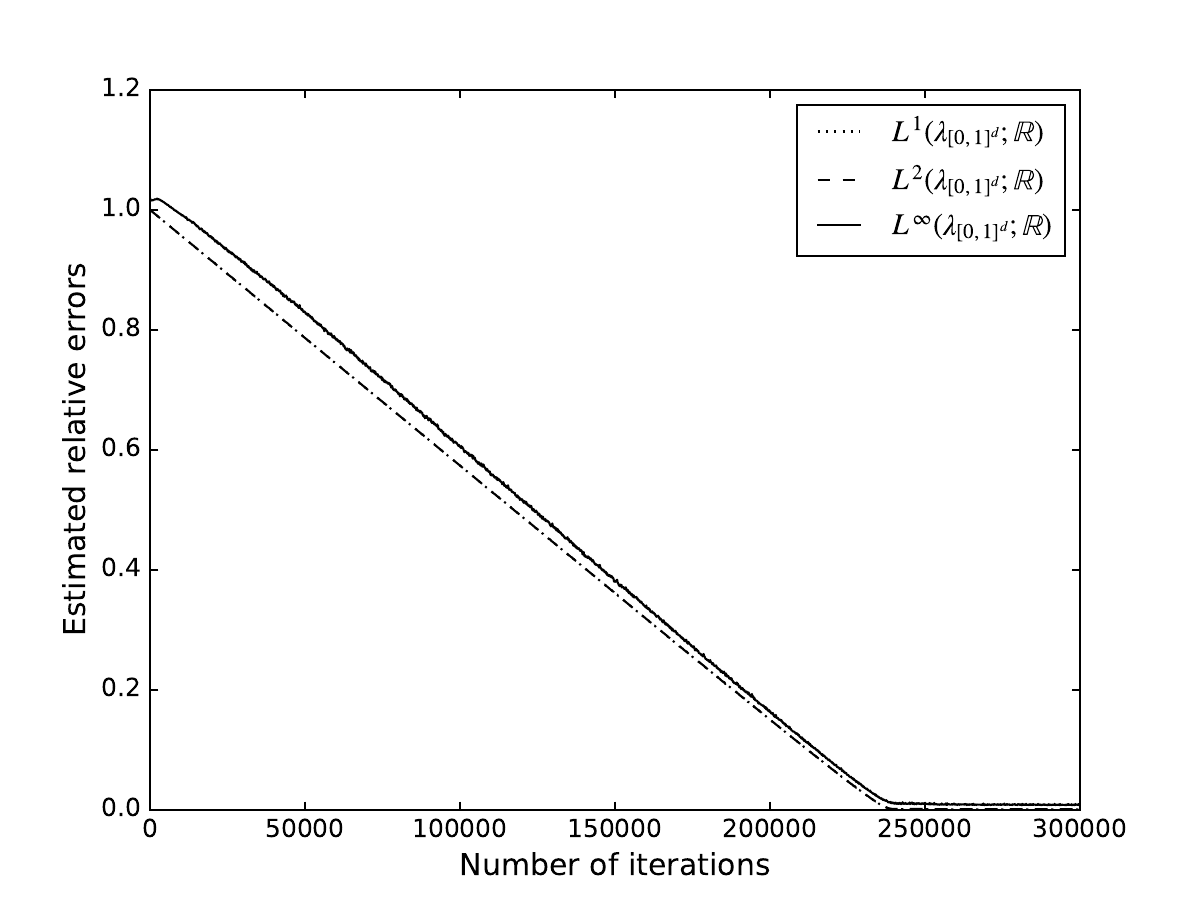}
	\caption{Approximative plots of the relative approximation errors in \eqref{eq:relL1fehler}--\eqref{eq:relLinftyfehler} for the heat equation in \eqref{eq:example_constant_coeffs_kolmogorovPDE}.}
	\label{fig:constant_coeffs}
\end{figure}
\begin{table}[H]	\begin{center}
	\begin{tabular}{|c|c|c|c|c|}
		\hline
		{\begin{tabular}{@{}c@{}}Number \\ of \\steps\end{tabular}}  & {\begin{tabular}{@{}c@{}}Relative \\ $L^1(\P;L^1(\lambda_{[0,1]^d};\R))$-\\error\end{tabular}} &  {\begin{tabular}{@{}c@{}}Relative \\ $L^2(\P;L^2(\lambda_{[0,1]^d};\R))$-\\error\end{tabular}} & {\begin{tabular}{@{}c@{}}Relative \\ $L^2(\P;L^\infty(\lambda_{[0,1]^d};\R))$-\\error\end{tabular}} & {\begin{tabular}{@{}c@{}}Mean\\  runtime\\ in\\ seconds\end{tabular}}\\
		\hline
		0  &  1.000310 &  1.000311 &  1.005674 &   0.6\\
		\hline
		10000  & 0.957481 &  0.957554 &  0.993097 &  44.7\\
		\hline
		50000  & 0.786628 &  0.786690 &  0.828816 &   220.4 \\
		\hline
		100000 & 0.573867 &  0.573914 &  0.605587 & 440.5\\
		\hline
		150000 &  0.361338 &  0.361369 &  0.382967  & 660.8 \\
		\hline
		200000 &  0.001387 &  0.001741 &  0.010896 & 880.9  \\
		\hline
		500000 &   0.000883 &  0.001112 &  0.008017  & 2201.0 \\
		\hline
		750000 & 0.000822 &  0.001038 &  0.007547 & 3300.4 \\
		\hline
	\end{tabular}
		\caption{Approximative presentations of the relative approximation errors in \eqref{eq:relL1fehler2}--\eqref{eq:relLinftyfehler2} for the heat equation in \eqref{eq:example_constant_coeffs_kolmogorovPDE}.}
		\label{tab:constant_coeffs_linfty_l1_l2_2}
	\end{center}
\end{table}
\begin{lemma}\label{lemma:example_constant_coeffs_exact_solution}
 Let 
  $T\in (0,\infty)$, 
  $d\in \N$,
 let 
  $C\in \R^{d\times d}$ be a 
  strictly positive and symmetric matrix, 
 and let 
  $u\colon [0,T]\times\R^d \to \R$ 
  be the function which satisfies for every $(t,x)\in [0,T]\times\R^d$ 
  that 
  \begin{equation}\begin{split}
     u(t,x) = \|x\|_{\R^d}^2 + t\operatorname{Trace}_{\R^d}(C).
     \end{split}
  \end{equation}
 Then 
 \begin{enumerate}[(i)]
  \item\label{it:exact_solution_polynomially} it holds that $u\in C^{\infty}([0,T]\times\R^d,\R)$ is 
  at most polynomially growing
  and 
  \item \label{it:exact_solution_derivative} it holds for every $t\in [0,T]$, $x\in\R^d$ that 
  \begin{equation}
   (\tfrac{\partial u}{\partial t})(t,x) 
   = 
   \tfrac12 \operatorname{Trace}_{\R^d}\!\big(C(\operatorname{Hess}_x u)(t,x)\big). 
  \end{equation}
 \end{enumerate}
\end{lemma}
\begin{proof}[Proof of Lemma~\ref{lemma:example_constant_coeffs_exact_solution}]
 First, note that $u$ is a polynomial. This establishes item~\eqref{it:exact_solution_polynomially}. 
 Moreover, note that for every $(t,x)\in [0,T]\times\R^d$ it holds that
 \begin{equation}
  (\tfrac{\partial u}{\partial t})(t,x) = \operatorname{Trace}_{\R^d}(C), 
  \qquad
  (\nabla_x u)(t,x) = 2x,
  \end{equation}
  \begin{equation}
  \text{and}\qquad 
  (\operatorname{Hess}_x u)(t,x) = \big(\tfrac{\partial }{\partial x}(\nabla_x u)\big)(t,x) = 2\operatorname{Id}_{\R^d}. 
 \end{equation}
 Hence, we obtain for every $(t,x)\in [0,T]\times\R^d$ that 
 \begin{equation}
  (\tfrac{\partial u}{\partial t})(t,x) 
  - 
  \tfrac12 \operatorname{Trace}_{\R^d}\!\big(C(\operatorname{Hess}_x u)(t,x)\big)
  = 
  \operatorname{Trace}_{\R^d}(C) 
  - 
  \tfrac12 \operatorname{Trace}_{\R^d}(2C)
  = 
  0.
 \end{equation}
 This proves item \eqref{it:exact_solution_derivative}. The proof of Lemma~\ref{lemma:example_constant_coeffs_exact_solution} 
 is thus completed.
\end{proof}
Lemma~\ref{lemma:MonteCarloSupEstimate} below discloses the strategy how we approximatively calculate the $L^\infty(\lambda_{[0,1]^d};\R)$-errors in \eqref{eq:relLinftyfehler} and \eqref{eq:relLinftyfehler2} above. Our proof of Lemma~\ref{lemma:MonteCarloSupEstimate} employs the elementary auxiliary results in Lemma~\ref{lem:capone} and Lemma~\ref{lem:almost_sure_convergence} below. For completeness we also include proofs for Lemma~\ref{lem:capone} and Lemma~\ref{lem:almost_sure_convergence} here.
\begin{lemma}
	\label{lem:capone}
	Let $ (\Omega, \F, \P) $ be a probability space and let $A, \mathcal{O}\in\F$ satisfy that $\P(\mathcal{O})=1$. Then it holds that
	\begin{equation}
		\P(A)=\P(A\cap \mathcal{O}).
	\end{equation}
\end{lemma}
\begin{proof}[Proof of  Lemma~\ref{lem:capone}]\
	Observe that the monotonicity of $\P$ ensures that
	\begin{equation}
		\begin{split}
			\P(A\cap\mathcal{O}) \le \P(A) &= \P\big([A\cap \mathcal{O}] \cup [A\backslash(A\cap \mathcal{O})]\big)\\
			&=\P(A\cap \mathcal{O})+\P(A\backslash(A\cap\mathcal{O}))\\
			&=\P(A\cap \mathcal{O})+\P(A\backslash \mathcal{O})\\
			&\le \P(A\cap \mathcal{O})+\P(\Omega\backslash \mathcal{O})\\
			&=\P(A\cap \mathcal{O})+\P(\Omega)-\P(\mathcal{O})=\P(A\cap \mathcal{O}).
		\end{split}
	\end{equation}
	Hence, we obtain that for every $A,\mathcal{O}\in\F$ with $\P(\mathcal{O})=1$ it holds that $\P(A)=\P(A\cap \mathcal{O})$. The Proof of Lemma~\ref{lem:capone} is thus completed.
\end{proof}

\begin{lemma}
	\label{lem:almost_sure_convergence}
	Let $ (\Omega, \F, \P) $ be a probability space, 
	let $ X_n \colon \Omega \to [0,\infty) $, $ n \in \N_0 $, 
	be random variables, assume for every $n\in\N$ that $\P(X_n\ge X_{n+1})=1$, and assume for every $\varepsilon\in(0,\infty)$ that
	\begin{equation}
	\label{eq:schwachgg0}
	\limsup_{n\to\infty}\P(X_n>\varepsilon)=0.
	\end{equation}
	Then 
	\begin{equation}
	\label{eq:fastsichergg0}
	\P\bigg(\limsup_{n\to\infty} X_n = 0\bigg)=1.
	\end{equation}
\end{lemma}
\begin{proof}[Proof of Lemma~\ref{lem:almost_sure_convergence}]
	Throughout this proof let $\mathcal{O}\subseteq \Omega$ be the set given by
	\begin{equation}
		\mathcal{O}=\cap_{n=1}^\infty \{X_n\ge X_{n+1}\}=\{\forall\, n\in\N\colon X_n\ge X_{n+1}\}.
	\end{equation}	
	Observe that Lemma~\ref{lem:capone} and the hypothesis that for every $n\in\N$ it holds that $\P(X_n \ge X_{n+1})=1$ assure that for every $N\in\N$ it holds that
	\begin{equation}		\begin{split}
		\P(\cap_{n=1}^N\{X_n \ge X_{n+1}\})&=\P([\cap_{n=1}^{N-1}\{X_n\ge X_{n+1}\}]\cap \{X_N \ge X_{N+1}\})\\
		&=\P(\cap_{n=1}^{N-1}\{X_n \ge X_{n+1}\}).		\end{split}
	\end{equation}
	This implies that for every $N\in\N$ it holds that
	\begin{equation}\begin{split}
		\P(\cap_{n=1}^N\{X_n \ge X_{n+1}\}) &= \P(\cap_{n=1}^0\{X_n \ge X_{n+1}\})\\
		&=\P(\Omega)=1.\end{split}
	\end{equation}
	The fact that the measure $\P$ is continuous from above hence demonstrates that
	\begin{equation}
	\label{eq:PO}
		\begin{split}
		\P(\mathcal{O})&=\P(\cap_{n=1}^\infty\{X_n \ge X_{n+1}\})=\P(\cap_{N=1}^\infty[\cap_{n=1}^N\{X_n \ge X_{n+1}\}])\\
		&=\lim_{N\to\infty}\P(\cap_{n=1}^N\{X_n \ge X_{n+1}\})=1.
		\end{split}
	\end{equation}
	Next note that
	\begin{equation}
		\begin{split}
		\label{eq:darstellungVonGegenWkeitneu}
		&\P\bigg( \limsup_{n\to\infty} X_n > 0 \bigg) 
		\\
		&=  \P\bigg(\exists\, k \in \N\colon\bigg[\limsup_{n\to\infty}X_n >\tfrac{1}{k}\bigg]\bigg)\\
		&=  \P\Big(\exists\, k \in \N\colon\forall\, m\in\N\colon\exists\, n\in\N\cap[m,\infty)\colon\big[X_n>\tfrac{1}{k}\big]\Big)\\
		&=  \P\Big(\cup_{k\in\N}\cap_{m\in \N}\cup_{n\in\N\cap[M,\infty)}\big\{X_n>\tfrac{1}{k}\big\}\Big)\\
		&\le \sum_{k=1}^\infty \P\Big(\cap_{m=1}^\infty \cup_{n=m}^\infty\big\{X_n>\tfrac{1}{k}\big\}\Big)\\
		&=\sum_{k=1}^\infty\left[\limsup_{m\to\infty}\P(\cup_{n=m}^\infty\{X_n>\tfrac1k\})\right].
		\end{split}
	\end{equation}
	Lemma \ref{lem:capone} and \eqref{eq:PO} therefore ensure that
	\begin{equation}
		\begin{split}
		&\P\bigg( \limsup_{n\to\infty} X_n > 0 \bigg) \\
		&\le \sum_{k=1}^\infty\bigg[ \limsup_{m\to\infty}\P\Big([\cup_{n=m}^\infty\{X_n>\tfrac1k\}]\cap\mathcal{O}\Big)\bigg]\\
		&\le \sum_{k=1}^\infty\bigg[ \limsup_{m\to\infty}\P\Big([\cup_{n=m}^\infty\{X_n>\tfrac1k\}]\cap\{\forall\, n\in\N\cap[m,\infty)\colon X_m\ge X_n\}\Big)\bigg]\\
		&=\sum_{k=1}^\infty\bigg[ \limsup_{m\to\infty}\P\Big(\{X_m>\tfrac1k\}\cap\{\forall\, n\in\N\cap[m,\infty)\colon X_m\ge X_n\}\Big)\bigg]\\
		&\le \sum_{k=1}^\infty\left[\limsup_{m\to\infty}\P\big(X_m>\tfrac1k\big)\right].
		\end{split}
	\end{equation}
	Combining this and \eqref{eq:schwachgg0} establishes \eqref{eq:fastsichergg0}. The proof of Lemma~\ref{lem:almost_sure_convergence} is thus completed.
\end{proof}

\begin{lemma}\label{lemma:MonteCarloSupEstimate}
 Let 
  $d\in\N$, $a\in\R$, $b\in(a,\infty)$, 
 let 
  $f\colon [a,b]^d\to \R$ 
  be a continuous function, 
 let 
  $(\Omega,\F,\P)$ be a probability space, 
 let 
  $X_n\colon \Omega\to [a,b]^d$, $n\in\N$,  
  be i.i.d.\ random variables, 
 and assume that 
  $X_1$ is continuous uniformly distributed on $[a,b]^d$. 
 Then 
 \begin{enumerate}[(i)]
  \item \label{it:weakapprox} it holds that
  \begin{equation}\label{eq:lemma_as_convergence}
   \P\!\left(
     \limsup_{N\to\infty} 
     \left|
       \left[ 
         \max\limits_{1\leq n\leq N} f(X_n)
       \right] 
       - 
       \left[  
         \sup\limits_{x\in [a,b]^d} f(x)
       \right]
     \right| 
     = 
     0
   \right) 
   = 
   1
  \end{equation}
  and 
  \item \label{it:pthmomentapprox}
  it holds for every $p\in (0,\infty)$ that 
  \begin{equation}\label{eq:lemma_lp_convergence}
   \limsup_{N\to\infty} 
   \E\!\left[\rule{0cm}{0.9cm}
     \left|
       \left[ 
         \max\limits_{1\leq n\leq N} f(X_n)
       \right]
       - 
       \left[ 
         \sup\limits_{x\in [a,b]^d} f(x)
       \right] 
     \right|^p 
   \right] = 0. 
  \end{equation}
 \end{enumerate}
\end{lemma}

\begin{proof}[Proof of Lemma~\ref{lemma:MonteCarloSupEstimate}]
First, observe that the fact that $f\colon [a,b]^d\to\R$ is a continuous 
 function and the fact that $[a,b]^d\subseteq\R^d$ is a compact set 
 demonstrate that there exists $\xi\in [a,b]^d$ 
 which satisfies that $f(\xi) = \sup_{x\in [a,b]^d} f(x)$. Next 
 note that the fact that for every $N\in\N$, $n\in\{1,2,\ldots,N\}$ it holds that 
 $f(X_n) \leq \sup_{x\in [a,b]^d} f(x)$ implies that for every $N\in\N$ it holds that 
 $\max_{1\leq n\leq N} f(X_n) \leq \sup_{x\in [a,b]^d} f(x)$. Hence, 
 we obtain that for every $N\in\N$ it holds that
 \begin{equation}
 \left|
   \big[ 
     \max\nolimits_{1\leq n\leq N} f(X_n)
   \big] 
   - 
   \big[
     \sup\nolimits_{x\in [a,b]^d} f(x)
   \big]
 \right| 
 = 
 \big[ 
   \sup\nolimits_{x\in [a,b]^d} f(x) 
 \big]  
 - 
 \big[
   \max\nolimits_{1\leq n\leq N} f(X_n)
 \big] . 
 \end{equation}
Combining this with the fact that $f(\xi) = \sup_{x\in\R^d} f(x)$ 
 ensures that for every $\varepsilon\in (0,\infty)$, $N\in\N$ it holds that 
 \begin{equation} \label{eq:set_relations}
 \begin{split}
  & 
  \big\{|\!\max\nolimits_{1\leq n\leq N} f(X_n) - \sup\nolimits_{x\in [a,b]^d} f(x)|\leq \varepsilon\big\}\\
  &= \big\{\max\nolimits_{1\leq n\leq N} f(X_n) \geq \sup\nolimits_{x\in [a,b]^d} f(x) - \varepsilon\big\}
  \\
  &= 
  \mathop{\cup}_{n=1}^N \big\{f(X_n)\geq\sup\nolimits_{x\in [a,b]^d} f(x)-\varepsilon\big\} 
  = \mathop{\cup}_{n=1}^N \big\{|f(X_n) - \sup\nolimits_{x\in [a,b]^d} f(x)| \leq \varepsilon\big\} 
  \\
  &= 
  \mathop{\cup}_{n=1}^N \big\{|f(X_n) - f(\xi)|\leq\varepsilon\big\}. 
 \end{split}
 \end{equation}
 In the next step we observe that the fact that $f\colon [a,b]^d\to\R$ is continuous 
 ensures that for every $\varepsilon\in (0,\infty)$ there exists $\delta\in (0,\infty)$ such that for 
 every $x\in [a,b]^d$ with $\|x-\xi\|_{\R^d} \leq \delta$ it 
 holds that $|f(x)-f(\xi)| \leq \varepsilon$. 
 Combining this and \eqref{eq:set_relations} shows that for every $\varepsilon \in (0,\infty)$ there 
 exists $\delta \in (0,\infty)$ such that for every $N\in\N$ it holds that
 \begin{equation}\label{eq:estimatingProbabilites}
  \begin{split}
  & 
  \P\big(|\!\max\nolimits_{1\leq n\leq N} f(X_n) - \sup\nolimits_{x\in [a,b]^d} f(x)| \leq \varepsilon \big) 
  = 
  \P\big( \cup_{n=1}^N \{|f(X_n) - f(\xi)| \leq \varepsilon \}\big) 
  \\
  &\geq 
  \P\big( \cup_{n=1}^N \{\|X_n-\xi\|_{\R^d} \leq \delta\}\big) 
  = 
  1 - \P\big( \cap_{n=1}^N \{\|X_n-\xi\|_{\R^d} > \delta\}\big). 
  \end{split}
 \end{equation}
 Hence, we obtain that for every $\varepsilon\in(0,\infty)$ there exists $\delta\in(0,\infty)$ such that 
 \begin{equation}
 \label{eq:144b}
 \begin{split}
 &\liminf_{N\to \infty} \P\big(|\!\max\nolimits_{1\le n \le N}f(X_n)-\sup\nolimits_{x\in[a,b]}f(x)|\le \varepsilon\big)\\
 &\ge 1- \liminf_{N\to\infty}\P\big(\cap_{n=1}^N\{\|X_n-\xi\|_{\R^d}>\delta\}\big).	
 \end{split}
 \end{equation}
Next observe that the fact that the random variables $X_n\colon \Omega\to [a,b]^d$, $n\in\{1,2,\ldots,N\}$, 
are i.i.d.\ ensures that for every $\delta\in (0,\infty)$, $N\in\N$ it holds that 
\begin{equation}
\label{eq:independentproduct}
 \begin{split}
 \P\big( \cap_{n=1}^N \{\|X_n-\xi\|_{\R^d} > \delta\}\big) 
 &= 
 \prod_{n=1}^N \P(\|X_n-\xi\|_{\R^d} > \delta) \\
 &= 
 \big[\P(\|X_1-\xi\|_{\R^d} > \delta)\big]^N. 
\end{split}
\end{equation}
In addition, note that the fact that for every $\delta\in(0,\infty)$ it holds that the set $\{x\in[a,b]^d\colon\|x-\xi\|_{\R^d}\le\delta\}\subseteq \R^d$ has strictly positive $d$-dimensional Lebesgue measure and the fact that $X_1$ is continuous uniformly distributed on $[a,b]^d$ ensure that for every $\delta\in(0,\infty)$ it holds that
\begin{equation}
	\P\!\left(\|X_1-\xi\|_{\R^d}>\delta\right)=1-\P\!\left(\|X_1-\xi\|_{\R^d}\le\delta\right)<1.
\end{equation}
Hence, we obtain that for every $\delta\in(0,\infty)$ it holds that
\begin{equation}
\limsup_{N\to\infty}\Big(\big[\P(\|X_1-\xi\|_{\R^d}>\delta)\big]^N\Big)=0.
\end{equation}
Combining this with \eqref{eq:independentproduct} demonstrates that for every $\delta\in(0,\infty)$ it holds that
\begin{equation}
\limsup_{N\to\infty}\P\big( \cap_{n=1}^N \{\|X_n-\xi\|_{\R^d} > \delta\}\big) =0.
\end{equation}
This and \eqref{eq:144b} assure that for every $\varepsilon\in(0,\infty)$ it holds that
\begin{equation}
	\liminf_{N\to\infty}\P\big(|\!\max\nolimits_{1\leq n\leq N} f(X_n) - \sup\nolimits_{x\in [a,b]^d} f(x) | \leq \varepsilon \big) =1.
\end{equation}
Therefore, we obtain that for every $\varepsilon\in(0,\infty)$ it holds that
\begin{equation}
	\limsup_{N\to\infty}\P\big( |\!\max\nolimits_{1\leq n\leq N} f(X_n) - \sup\nolimits_{x\in [a,b]^d} f(x) | >\varepsilon\big) =0.
\end{equation}
Combining this with with Lemma~\ref{lem:almost_sure_convergence} establishes item~\eqref{it:weakapprox}. It thus remains to prove item~\eqref{it:pthmomentapprox}. For this note that the fact that $f\colon [a,b]^d\to\R$ is globally bounded, item~\eqref{it:weakapprox}, and Lebesgue's dominated convergence theorem 
ensure that for every $p\in (0,\infty)$ it holds that
\begin{equation}
 \limsup_{N\to\infty} \E\Big[ \big|\!\max\nolimits_{1\leq i\leq N} f(X_i) - \sup\nolimits_{x\in [a,b]^d} f(x) \big|^p\Big] = 0. 
\end{equation} This establishes item~\eqref{it:pthmomentapprox}. 
The proof of Lemma~\ref{lemma:MonteCarloSupEstimate} is thus completed.
\end{proof}

\subsection{Geometric Brownian motions}
\label{sec:geometric}
In this subsection we apply the proposed approximation algorithm 
to a Black-Scholes PDE with independent underlying geometric Brownian 
motions. 

Assume 
  Framework \ref{algo:general_algorithm}, let $r=\tfrac{1}{20}$, $\mu=r-\tfrac{1}{10}=-\tfrac{1}{20}$, $\sigma_1=\tfrac{1}{10}+\tfrac{1}{200}$, $\sigma_2=\tfrac{1}{10}+\tfrac{2}{200}$,\ldots, $\sigma_{100}=\tfrac{1}{10}+\tfrac{100}{200}$, 
assume for every $s,t \in[0,T]$, $x=(x_1,x_2,\ldots,x_d)$, $w=(w_1,w_2,\ldots,w_d)\in\R^d$, $m\in\N_0$ that 
  $ N = 1$, 
  $ d = 100$,
  $ \varphi(x) = \exp(-r T)\max\!\big\{[\max_{i\in\{1,2,\ldots,d\}} x_i]-100,0\big\} $, and
  \begin{equation}
  \begin{split}
    	&H(s,t,x,w) = \\
    	&\Big( x_1 \exp\!\big(\big(\mu_1-\tfrac{|\sigma_1|^2}{2}\big)(t-s) + \sigma_1w_1\big), \ldots, x_d \exp\!\big(\big(\mu_d-\tfrac{|\sigma_d|^2}{2}\big)(t-s) + \sigma_dw_d\big)\Big),
  \end{split}
    \end{equation}assume that 
  $ \xi^{0,1}\colon\Omega\to\R^d$ is continuous uniformly distributed on $[90,110]^d$,
and 
  let $u = (u(t,x))_{(t,x)\in [0,T]\times\R^d} \in C^{1,2}([0,T]\times\R^d,\R)$ 
  be an at most polynomially growing function which satisfies 
  for every $t\in [0,T]$, $x\in\R^d$ that 
  $u(0,x) = \varphi(x)$ and 
\begin{equation}\label{eq:example_geometric_brownian_motion_kolmogorovPDE}
 (\tfrac{\partial u}{\partial t})(t,x)  
 = 
 \tfrac12 \sum_{i=1}^d 
 |\sigma_i x_i|^2 (\tfrac{\partial^2 u}{\partial x_i^2})(t,x)
 + 
 \sum_{i=1}^d \mu_i x_i(\tfrac{\partial u}{\partial x_i})(t,x). 
\end{equation}
The Feynman-Kac formula 
(cf., for example, Hairer et al.~\cite[Corollary 4.17]{HairerHutzenthalerJentzen_LossOfRegularity2015})
shows that for every standard Brownian motion $\mathcal{W} =(\mathcal{W}^{(1)},\ldots,\mathcal{W}^{(d)})\colon[0,T]\times \Omega\to\R^d$ and every $t\in [0,T]$, $x=(x_1,\ldots,x_d)\in\R^d$ it holds that 
\begin{equation}\label{eq:example_geometric_brownian_motion_stochastic_representation}
\begin{split}
 &u(t,x) 
 = \\
 & \E\!\left[
  \varphi\!\left(x_1 \exp\!\left(\sigma_1 \mathcal{W}^{(1)}_t 
  + 
  \left(\mu_1 - \tfrac{|\sigma_1|^2}{2}\right)t\right), \ldots, x_d \exp\!\left(\sigma_d \mathcal{W}^{(d)}_t 
  + 
  \left(\mu_d - \tfrac{|\sigma_d|^2}{2}\right)t\right)\right)
 \right]. 
\end{split}
\end{equation}
Table~\ref{tab:geometric_linfty_l1_l2} 
approximately presents the relative $L^1(\lambda_{[0,1]^d};\R)$-approximation error associated to 
$(\bU^{\Theta_m,1,\bS_m}(x))_{x\in[0,1]^d}$ (see \eqref{eq:relL1fehler_bm} below),
the relative $L^2(\lambda_{[0,1]^d};\R)$-approximation error associated to 
$(\bU^{\Theta_m,1,\bS_m}(x))_{x\in[0,1]^d}$ (see \eqref{eq:relL2fehler_bm} below), 
and 
the relative $L^{\infty}(\lambda_{[0,1]^d};\R)$-approximation error associated to 
$(\bU^{\Theta_m,1,\bS_m}(x))_{x\in[0,1]^d}$  (see \eqref{eq:relLinftyfehler_bm} below)  against 
$ m \in \{0, 25000, \allowbreak50000,\allowbreak 100000, 150000, 250000, 500000, \allowbreak 750000\}$ (cf.\ \textsc{Python} code~\ref{code:geometric} in Subsection~\ref{sec:code_geometric} below).
In our numerical simulations for Table~\ref{tab:geometric_linfty_l1_l2}  we approximately calculated the exact 
solution of the PDE~\eqref{eq:example_geometric_brownian_motion_kolmogorovPDE} by means  of~\eqref{eq:example_geometric_brownian_motion_stochastic_representation}
and Monte Carlo approximations with 1048576 samples, we approximately calculated the relative $L^1(\lambda_{[0,1]^d};\R)$-approximation error
\begin{equation}
\label{eq:relL1fehler_bm}
\int_{[0,1]^d} \left|\frac{u(T,x) - \bU^{\Theta_m,1,\bS_m}(x)}{u(T,x)}\right|dx
\end{equation}
for $ m \in \{0, 25000, \allowbreak50000,\allowbreak 100000, 150000, 250000, 500000, \allowbreak 750000\}$ 
 by means of Monte Carlo approximations with 81920 samples, we approximately calculated the relative $L^2(\lambda_{[0,1]^d};\R)$-approximation error 
 \begin{equation}
 \label{eq:relL2fehler_bm}
 \sqrt{\int_{[0,1]^d} \left|\frac{u(T,x) - \bU^{\Theta_m,1,\bS_m}(x)}{u(T,x)}\right|^2dx}
 \end{equation} for $ m \in \{0, 25000, \allowbreak50000,\allowbreak 100000, 150000, 250000, 500000, \allowbreak 750000\}$  
 by means of Monte Carlo approximations with 81920 samples, and  we approximately calculated the relative $L^\infty(\lambda_{[0,1]^d};\R)$-approximation error
\begin{equation}
\label{eq:relLinftyfehler_bm}
\sup_{x\in [0,1]^d} \left|\frac{u(T,x) -  \bU^{\Theta_m,1,\bS_m}(x)}{u(T,x)}\right|
\end{equation} for $ m \in \{0, 25000, \allowbreak50000,\allowbreak 100000, 150000, 250000, 500000, \allowbreak 750000\}$  
by means of Monte Carlo approximations with 81920 samples (see Lemma~\ref{lemma:MonteCarloSupEstimate} above).

	\begin{table}[H]	\begin{center}
\begin{tabular}{|c|c|c|c|c|}
		\hline
		{\begin{tabular}{@{}c@{}}Number \\ of steps\end{tabular}}  & {\begin{tabular}{@{}c@{}}Relative \\ $L^1(\lambda_{[0,1]^d};\R)$-error\end{tabular}} &  {\begin{tabular}{@{}c@{}}Relative \\ $L^2(\lambda_{[0,1]^d};\R)$-error\end{tabular}} & {\begin{tabular}{@{}c@{}}Relative \\ $L^\infty(\lambda_{[0,1]^d};\R)$-error\end{tabular}} & {\begin{tabular}{@{}c@{}}Runtime \\ in seconds\end{tabular}}\\
		\hline
		0 &  1.004285&1.004286&1.009524 & 1 \\
		\hline
		25000 &  0.842938&0.843021&0.87884& 110.2 \\
		\hline
		50000 &  0.684955&0.685021&0.719826 &  219.5\\
		\hline
		100000 &  0.371515&0.371551&0.387978 &  437.9\\
		\hline
		150000 &  0.064605&0.064628&0.072259 &  656.2\\
		\hline
		250000 &  0.001220&0.001538&0.010039 &  1092.6\\
		\hline
		500000 & 0.000949&0.001187&0.005105&  2183.8 \\
		\hline
		750000 & 0.000902&0.001129&0.006028 &  3275.1\\
		\hline
	\end{tabular}
	\caption{Approximative presentations of the relative approximation errors in \eqref{eq:relL1fehler_bm}--\eqref{eq:relLinftyfehler_bm} for the  Black-Scholes PDE with independent underlying geometric Brownian motions in  \eqref{eq:example_geometric_brownian_motion_kolmogorovPDE}.}
	\label{tab:geometric_linfty_l1_l2}\end{center}
	\end{table}

\subsection{Black-Scholes model with correlated noise}
\label{sec:black_scholes}
In this subsection we apply the proposed approximation algorithm 
to a Black-Scholes PDE with correlated noise. 

Assume 
  Framework~\ref{algo:general_algorithm}, 
  let $r=\tfrac{1}{20}$, $\mu=r-\tfrac{1}{10}=-\frac{1}{20}$,  $\beta_1=\tfrac{1}{10}+\tfrac{1}{200}$, $\beta_2=\tfrac{1}{10}+\tfrac{2}{200}$,\ldots, $\beta_{100}=\tfrac{1}{10}+\tfrac{100}{200}$, $Q=(Q_{i,j})_{(i,j)\in\{1,2,\ldots,100\}}$, $ \Sigma=(\Sigma_{i,j})_{(i,j)\in\{1,2,\ldots,100\}}\in\R^{100\times100}$, $\varsigma_1,\varsigma_2,\ldots,\varsigma_{100}\in\R^{100}$, assume for every $s,t \in[0,T]$, $x=(x_1,x_2,\ldots,x_d)$, $w=(w_1,w_2,\ldots,w_d)\in\R^d$, $m\in\N_0$, $i,j,k\in\{1,2,\ldots,100\}$ with $i<j$  
 that
  $ N = 1$, 
  $ d = 100$,
  $ \nu = d(2d)+(2d)^2+2d=2d(3d+1)$, $Q_{k,k}=1$, $Q_{i,j}=Q_{j,i}=\tfrac12$, $\Sigma_{i,j}=0$, $\Sigma_{k,k}> 0$, $\Sigma\Sigma^\ast=Q$ (cf., for example, Golub \& Van Loan~\cite[Theorem 4.2.5]{GolubVanLoan}), $\varsigma_k=(\Sigma_{k,1},\ldots,\Sigma_{k,100})$, $
  \varphi(x) = \exp(-\mu T)\max\!\big\{110-[\min_{i\in\{1,2,\ldots,d\}} x_i],0\big\} 
  $, and
  \begin{multline}
  %\begin{split}
 	H(s,t,x,w) =\Big(x_1 \exp\!\left((\mu-\tfrac12{\|\beta_1\varsigma_1\|_{\R^d}^2})(t-s) + \langle \varsigma_1,w\rangle_{\R^d}\right),
 	\ldots,\\
 	x_d\exp\!\left((\mu-\tfrac12{\|\beta_d\varsigma_d\|_{\R^d}^2})(t-s) + \langle \varsigma_d,w\rangle_{\R^d}\right)\Big),
  %\end{split}
  \end{multline} 
  assume that 
  $ \xi^{0,1}\colon\Omega\to\R^d$ is continuous uniformly distributed on $[90,110]^d$, 
and let $u=(u(t,x))_{t\in [0,T],x\in\R^d}\in C^{1,2}([0,T]\times\R^d,\R)$ 
be an at most polynomially growing continuous function which satisfies 
for every $t\in [0,T]$, $x\in\R^d$ that $u(0,x)=\varphi(x)$ and 
\begin{equation}
\label{eq:example_correlated_geometric_brownian_motion_kolmogorovPDE}
 (\tfrac{\partial u}{\partial t})(t,x)  
 = 
 \tfrac12 \sum_{i,j=1}^d 
 x_i x_j \beta_i \beta_j \langle \varsigma_i,\varsigma_j \rangle_{\R^d}
 (\tfrac{\partial^2 u}{\partial x_i^2})(t,x)
 + 
 \sum_{i=1}^d \mu_i x_i (\tfrac{\partial u}{\partial x_i})(t,x). 
\end{equation}
The Feynman-Kac formula 
(cf., for example, Hairer et al.~\cite[Corollary 4.17]{HairerHutzenthalerJentzen_LossOfRegularity2015})
shows that for every standard Brownian motion $\mathcal{W} =(\mathcal{W}^{(1)},\ldots,\mathcal{W}^{(d)})\colon[0,T]\times \Omega\to\R^d$ and every $t\in [0,T]$, $x=(x_1,\ldots,x_d)\in\R^d$ it holds that 
\begin{multline}\label{eq:example_correlated_geometric_brownian_motion_stochastic_representation}
%\begin{split}
u(t,x) 
= \E\!\left[
\varphi\!\left(x_1 \exp\!\left(\big\langle \varsigma_1,\mathcal{W}^{(1)}_t\big\rangle_{\R^d}  
+ 
\Big(\mu_1 - \tfrac{\|\beta_1\varsigma_1\|_{\R^d}^2}{2}\Big)t\right),\ldots,\right.\right.\\
\qquad~\left. \left. x_d \exp\!\left(\big\langle \varsigma_d,\mathcal{W}^{(d)}_t\big\rangle_{\R^d}  
+ 
\Big(\mu_d - \tfrac{\|\beta_d\varsigma_d\|_{\R^d}^2}{2}\Big)t\right)\right)
\right]. 
%\end{split}
\end{multline}
Table~\ref{tab:brownian_linfty_l1_l2} approximately presents 
 the relative $L^1(\lambda_{[0,1]^d};\R)$-approximation error associated to 
$(\bU^{\Theta_m,1,\bS_m}(x))_{x\in[0,1]^d}$ (see \eqref{eq:relL1fehler_bs} below),
the relative $L^2(\lambda_{[0,1]^d};\R)$-approximation error associated to 
$(\bU^{\Theta_m,1,\bS_m}(x))_{x\in[0,1]^d}$ (see \eqref{eq:relL2fehler_bs} below), 
and 
the relative $L^{\infty}(\lambda_{[0,1]^d};\R)$-approximation error associated to 
$(\bU^{\Theta_m,1,\bS_m}(x))_{x\in[0,1]^d}$  (see \eqref{eq:relLinftyfehler_bs} below) against 
$ m \in \{0, 25000, \allowbreak50000,\allowbreak 100000, 150000, 250000, 500000, \allowbreak 750000\}$ (cf.\ \textsc{Python} code~\ref{code:black_scholes} in Subsection~\ref{sec:black_scholes} below). In our numerical simulations for Table~\ref{tab:brownian_linfty_l1_l2} we approximately calculated the exact 
solution of the PDE~\eqref{eq:example_correlated_geometric_brownian_motion_kolmogorovPDE} by means of \eqref{eq:example_correlated_geometric_brownian_motion_stochastic_representation} and 
Monte Carlo approximations with 1048576 samples, we approximately calculated the relative $L^1(\lambda_{[0,1]^d};\R)$-approximation error
\begin{equation}
\label{eq:relL1fehler_bs}
\int_{[0,1]^d} \left|\frac{u(T,x) - \bU^{\Theta_m,1,\bS_m}(x)}{u(T,x)}\right|dx
\end{equation} for $ m \in \{0, 25000, \allowbreak50000,\allowbreak 100000, 150000, 250000, 500000, \allowbreak 750000\}$ 
 by means of Monte Carlo approximations with 81920 samples, we approximately calculated the relative $L^2(\lambda_{[0,1]^d};\R)$-approximation error
 \begin{equation}
 \label{eq:relL2fehler_bs}
 \sqrt{\int_{[0,1]^d} \left|\frac{u(T,x) - \bU^{\Theta_m,1,\bS_m}(x)}{u(T,x)}\right|^2dx}
 \end{equation}  for $ m \in \{0, 25000, \allowbreak50000,\allowbreak 100000, 150000, 250000, 500000, \allowbreak 750000\}$ by means of Monte Carlo approximations with 81920 samples, and we approximately calculated the relative $L^\infty(\lambda_{[0,1]^d};\R)$-approximation error
 \begin{equation}
 \label{eq:relLinftyfehler_bs}
 \sup_{x\in [0,1]^d} \left|\frac{u(T,x) -  \bU^{\Theta_m,1,\bS_m}(x)}{u(T,x)}\right|
 \end{equation} for $ m \in \{0, 25000, \allowbreak50000,\allowbreak 100000, 150000, 250000, 500000, \allowbreak 750000\}$ 
by means of Monte Carlo approximations with 81920 samples (see Lemma~\ref{lemma:MonteCarloSupEstimate} above). 
\begin{table}[H]	\begin{center}
				\begin{tabular}{|c|c|c|c|c|}
				\hline
				{\begin{tabular}{@{}c@{}}Number \\ of steps\end{tabular}}  & {\begin{tabular}{@{}c@{}}Relative \\ $L^1(\lambda_{[0,1]^d};\R)$-error\end{tabular}} &  {\begin{tabular}{@{}c@{}}Relative \\ $L^2(\lambda_{[0,1]^d};\R)$-error\end{tabular}} & {\begin{tabular}{@{}c@{}}Relative \\ $L^\infty(\lambda_{[0,1]^d};\R)$-error\end{tabular}} & {\begin{tabular}{@{}c@{}}Runtime \\ in seconds\end{tabular}}\\
				\hline
		0 &  1.003383&1.003385&1.011662 & 0.8 \\
		\hline
		25000 &  0.631420&0.631429&0.640633 & 112.1 \\
		\hline
		50000 &  0.269053&0.269058&0.275114&  223.3\\
		\hline
		100000 &  0.000752&0.000948&0.00553 &  445.8\\
		\hline
		150000 & 0.000694&0.00087&0.004662 &  668.2\\
		\hline
		250000 &  0.000604&0.000758&0.006483 &  1119.3\\
		\hline
		500000 & 0.000493&0.000615&0.002774 &  2292.8 \\
		\hline
		750000 & 0.000471&0.00059&0.002862 &  3466.8\\
		\hline
	\end{tabular}
	\caption{Approximative presentations of the relative approximation errors in \eqref{eq:relL1fehler_bs}--\eqref{eq:relLinftyfehler_bs} for the Black-Scholes PDE with correlated noise in  \eqref{eq:example_correlated_geometric_brownian_motion_kolmogorovPDE}.}
	\label{tab:brownian_linfty_l1_l2}\end{center}
\end{table}

\subsection{Stochastic Lorenz equations}
\label{sec:lorenz}

In this subsection we apply the proposed approximation algorithm 
to the stochastic Lorenz equation.

Assume 
Framework~\ref{algo:general_algorithm}, let $\alpha_1=10$, $\alpha_2=14$, $\alpha_3=\tfrac{8}{3}$, $\beta=\tfrac{3}{20}$, let $\mu\colon \R^d\to\R^d$ be a function,  assume for every $s,t \in[0,T]$, $x=(x_1,x_2,\ldots,x_d)$, $w=(w_1,w_2,\ldots,w_d)\in\R^d$, $m\in\N_0$  
that
$ N = 100 $, 
$ d = 3 $, 
$ \nu = (d+20)d + (d+20)^2 +(d+20)= (d+20)(2d+21)$, $\mu(x) = (\alpha_1 (x_2-x_1), \alpha_2 x_1 - x_2 - x_1 x_3, x_1 x_2 - \alpha_3 x_3)
$, $\varphi(x) = \|x\|^2_{\R^d}
$, and 
\begin{equation}
\label{eq:lorenz_verfahren}
H ( s , t , x , w ) 
= 
x + 
\mu(x)(t-s)\mathbbm{1}_{[0,N/T]}(\|\mu(x)\|_{\R^d}) + \beta w
\end{equation}
(cf., for example, Hutzenthaler et al.~\cite{HutzenthalerJentzen2015_Memoirs}, Hutzenthaler et al.~\cite{HutzenthalerJentzenKloeden_StrongConvergenceExplicit2012}, Hutzenthaler et al.~\cite{HutzenthalerJentzenWang}, Milstein \& Tretyakov~\cite{MilsteinTretyakovNonglobal}, Sabanis~\cite{Sabanis13, Sabanis16}, and the references mentioned therein for related temporal numerical approximation schemes for SDEs), assume that 
$ \xi^{0,1}\colon\Omega\to\R^d$ is continuous uniformly distributed on $[\tfrac12,\tfrac32]\times[8,10]\times[10,12]$, 
and let $u=(u(t,x))_{(t,x)\in [0,T]\times\R^d}\in C^{1,2}([0,T]\times\R^d,\R)$ be an at most polynomially growing function (cf., for example, Hairer et al.~\cite[Corollary 4.17]{HairerHutzenthalerJentzen_LossOfRegularity2015} and H\"ormander~\cite[Theorem 1.1]{hoermander1967})% which satisfies 
for every $t\in [0,T]$, $x\in\R^d$ that 
$u(0,x) = \varphi(x)$ and 
\begin{equation}
\begin{split}
\label{eq:example_lorenz_kolmogorovPDE}
&(\tfrac{\partial u}{\partial t})(t,x)
= 
\tfrac{\beta^2}{2}(\Delta_x u)(t,x)
+ 
\alpha_1(x_2-x_1) (\tfrac{\partial u}{\partial x_1})(t,x) \\
&+ 
(\alpha_2 x_1 - x_2 - x_1 x_3)(\tfrac{\partial u}{\partial x_2})(t,x) 
+ 
(x_1 x_2 - \alpha_3 x_3)(\tfrac{\partial u}{\partial x_3})(t,x).
\end{split} 
\end{equation}
Table~\ref{tab:lorenz_linfty_l1_l2}  
approximately presents the relative $L^1(\lambda_{[0,1]^d};\R)$-approximation error associated to 
$(\bU^{\Theta_m,1,\bS_m}(x))_{x\in[0,1]^d}$ (see \eqref{eq:relL1fehler_lz} below),
the relative $L^2(\lambda_{[0,1]^d};\R)$-approximation error associated to 
$(\bU^{\Theta_m,1,\bS_m}(x))_{x\in[0,1]^d}$ (see \eqref{eq:relL2fehler_lz} below), 
and 
the relative $L^{\infty}(\lambda_{[0,1]^d};\R)$-approximation error associated to 
$(\bU^{\Theta_m,1,\bS_m}(x))_{x\in[0,1]^d}$  (see \eqref{eq:relLinftyfehler_lz} below) against 
$ m \in \{0, 25000, \allowbreak50000,\allowbreak 100000, 150000, 250000, 500000, \allowbreak 750000\}$ (cf.\ \textsc{Python} code~\ref{code:lorenz} in Subsection~\ref{sec:code_lorenz} below). In our numerical simulations for Table~\ref{tab:lorenz_linfty_l1_l2} we approximately calculated the exact 
solution of the 
PDE~\eqref{eq:example_lorenz_kolmogorovPDE} by means of %\eqref{eq:LZ_PDE_satisfies} and 
Monte Carlo approximations  with 104857 samples and temporal SDE-discretizations based on \eqref{eq:lorenz_verfahren} and 100 equidistant time steps, we approximately calculated the relative $L^1(\lambda_{[0,1]^d};\R)$-approximation error 
\begin{equation}
\label{eq:relL1fehler_lz}
\int_{[0,1]^d} \left|\frac{u(T,x) - \bU^{\Theta_m,1,\bS_m}(x)}{u(T,x)}\right|dx
\end{equation} for $ m \in \{0, 25000, \allowbreak50000,\allowbreak 100000, 150000, 250000, 500000, \allowbreak 750000\}$ 
by means of Monte Carlo approximations with 20480 samples, we approximately calculated the relative $L^2(\lambda_{[0,1]^d};\R)$-approximation error 
\begin{equation}
\label{eq:relL2fehler_lz}
\sqrt{\int_{[0,1]^d} \left|\frac{u(T,x) - \bU^{\Theta_m,1,\bS_m}(x)}{u(T,x)}\right|^2dx}
\end{equation} for $ m \in \{0, 25000, \allowbreak50000,\allowbreak 100000, 150000, 250000, 500000, \allowbreak 750000\}$ 
by means of Monte Carlo approximations with 20480 samples, and we approximately calculated the relative $L^\infty(\lambda_{[0,1]^d};\R)$-approximation error
\begin{equation}
\label{eq:relLinftyfehler_lz}
 \sup_{x\in [0,1]^d} \left|\frac{u(T,x) -  \bU^{\Theta_m,1,\bS_m}(x)}{u(T,x)}\right|
\end{equation} for $ m \in \{0, 25000, \allowbreak50000,\allowbreak 100000, 150000, 250000, 500000, \allowbreak 750000\}$ 
by means of Monte Carlo approximations with 20480 samples (see Lemma~\ref{lemma:MonteCarloSupEstimate} above).  

\begin{table}[H]	\begin{center}
		\begin{tabular}{|c|c|c|c|c|}
				\hline
				{\begin{tabular}{@{}c@{}}Number \\ of steps\end{tabular}}  & {\begin{tabular}{@{}c@{}}Relative \\ $L^1(\lambda_{[0,1]^d};\R)$-error\end{tabular}} &  {\begin{tabular}{@{}c@{}}Relative \\ $L^2(\lambda_{[0,1]^d};\R)$-error\end{tabular}} & {\begin{tabular}{@{}c@{}}Relative \\ $L^\infty(\lambda_{[0,1]^d};\R)$-error\end{tabular}} & {\begin{tabular}{@{}c@{}}Runtime \\ in seconds\end{tabular}}\\
				\hline
			0 & 0.995732& 0.995732& 0.996454& 1.0 \\
			\hline
			25000 & 0.905267& 0.909422& 1.247772& 750.1 \\
			\hline
			50000  &0.801935& 0.805497& 1.115690& 1461.7 \\
			\hline
			100000 &0.599847& 0.602630& 0.823042& 2932.1 \\
			\hline
			150000 &0.392394& 0.394204& 0.542209& 4423.3 \\
			\hline
			250000 &0.000732& 0.000811& 0.002865& 7327.9 \\
			\hline
			500000 &0.000312& 0.000365& 0.003158& 14753.0 \\
			\hline
			750000 &0.000187& 0.000229& 0.001264& 21987.4 \\
			\hline
		\end{tabular}
		\caption{Approximative presentations of the relative approximation errors in \eqref{eq:relL1fehler_lz}--\eqref{eq:relLinftyfehler_lz} for the stochastic Lorenz equation in  \eqref{eq:example_lorenz_kolmogorovPDE}.}
		\label{tab:lorenz_linfty_l1_l2}\end{center}
\end{table}

\subsection{Heston model}
\label{sec:heston}
In this subsection we apply the proposed approximation algorithm 
to the Heston model in \eqref{eq:example_heston_kolmogorovPDE} below.

Assume 
  Framework \ref{algo:general_algorithm}, 
let 
  $ \delta =25$, $\alpha=\tfrac{1}{20}$, $\kappa=\tfrac{6}{10}$, 
  $ \theta =\tfrac{1}{25} $, 
  $ \beta=\tfrac{1}{5}$, $\varrho =-\tfrac{4}{5}$, 
  let $e_{i}\in\R^{50}$, $i\in\{1,2,\ldots,{50}\}$, be the vectors which satisfy that 
  $e_{1} = (1,0,0,\ldots,0,0)\in\R^{50}$, $e_{2} = (0,1,0,\ldots,0,0)\in\R^{50}$, 
  \ldots, 
  $e_{{50}} = (0,0,0,\ldots,0,1)\in\R^{50}$, 
assume for every $s,t \in[0,T]$, $x=(x_1,x_2,\ldots,x_d)$, $w=(w_1,w_2,\ldots,w_d)\in\R^d$ that
  $ N = 100$, 
  $ d = 2\delta = 50 $, 
  $ \nu = (d+50)d + (d+50)^2 +(d+50)= (d+50)(2d+51)$, $
  \varphi(x) = \exp(-\alpha T)\max\!\big\{110 - [\textstyle\sum_{i=1}^\delta\tfrac{x_{2i-1}}{\delta}],0\big\} 
  $, and
  \begin{multline}
  \label{eq:heston_verfahren}
  %\begin{split}
  H( s , t , x , w )=  \sum_{i=1}^{\delta}\Bigg( \bigg[x_{2i-1}\exp\!\Big((\alpha-\tfrac{x_{2i}}{2})(t-s)+w_{2i-1}\sqrt{x_{2i}}\Big)\bigg]e_{2i-1}\Bigg.\\
  \left.+\bigg[\max\!\bigg\{\Big[
  \max\!\Big\{\tfrac\beta2\sqrt{t-s},\max\!\big\{\tfrac{\beta}{2}\sqrt{t-s},\sqrt{x_{2i}}\big\}+\tfrac{\beta}{2}\big(\rho w_{2i-1}+[{1-\rho^2}]^{1/2}w_{2i}\big)\Big\}\Big]^2\right.\\
  \Bigg.+\big(\kappa\theta-\tfrac{\beta^2}{4}-\kappa x_{2i}\big)(t-s),0\bigg\}\bigg]e_{2i}\Bigg)
 % \end{split}
  \end{multline}
  (cf.\ Hefter \& Herzwurm~\cite[Section 1]{HefterHerzwurm}), assume that 
  $ \xi^{0,1}\colon\Omega\to\R^d$ 
  is continuous uniformly distributed on $\times_{i=1}^\delta \big([90,110]\times[0.02,0.2]\big)$, 
and let $u=(u(t,x))_{(t,x)\in [0,T]\times\R^d}\in C^{1,2}([0,T]\times\R^d,\R)$ 
  be an at most polynomially growing function (cf., for example, Alfonsi~\cite[Proposition 4.1]{Alfonsi2005}) which satisfies 
  for every $t\in [0,T]$, $x\in\R^d$ that 
  $u(0,x) = \varphi(x)$ and 
\begin{multline}
%\begin{split}
\label{eq:example_heston_kolmogorovPDE}
(\tfrac{\partial u}{\partial t})(t,x)  = 
 \Bigg[\sum_{i=1}^{\delta} 
  \Big(
  \alpha x_{2i-1} (\tfrac{\partial u}{\partial x_{2i-1}})(t,x) 
  + 
  \kappa (\theta - x_{2i}) (\tfrac{\partial u}{\partial x_{2i}})(t,x)
  \Big)\Bigg]
 \\
+ \Bigg[
 \sum_{i=1}^{\delta} \frac{|x_{2i}|}{2}
 \Big( 
  |x_{2i-1}|^2 (\tfrac{\partial^2 u}{\partial x_{2i-1}^2})(t,x)  + 2 x_{2i-1} \beta\varrho (\tfrac{\partial^2 u}{\partial x_{2i-1}\partial x_{2i}})(t,x) + \beta^2 (\tfrac{\partial^2 u}{\partial x_{2i}^2})(t,x)
 \Big)\Bigg]. 
%\end{split}
\end{multline}
Table~\ref{tab:heston_linfty_l1_l2}  
approximately presents the relative $L^1(\lambda_{[0,1]^d};\R)$-approximation error associated to 
$(\bU^{\Theta_m,1,\bS_m}(x))_{x\in[0,1]^d}$ (see \eqref{eq:relL1fehler_hs} below),
the relative $L^2(\lambda_{[0,1]^d};\R)$-approximation error associated to 
$(\bU^{\Theta_m,1,\bS_m}(x))_{x\in[0,1]^d}$ (see \eqref{eq:relL2fehler_hs} below), 
and 
the relative $L^{\infty}(\lambda_{[0,1]^d};\R)$-approximation error associated to 
$(\bU^{\Theta_m,1,\bS_m}(x))_{x\in[0,1]^d}$ (see \eqref{eq:relLinftyfehler_hs} below) against 
$ m \in \{0, 25000, \allowbreak50000,\allowbreak 100000, 150000, 250000, 500000, \allowbreak 750000\}$ (cf.\ \textsc{Python} code~\ref{code:heston} in Subsection~\ref{sec:code_heston} below). In our numerical simulations for Table~\ref{tab:heston_linfty_l1_l2}  we approximately calculated the exact 
solution of the PDE~\eqref{eq:example_heston_kolmogorovPDE} by means of 
Monte Carlo approximations with 1048576 samples and temporal SDE-discretizations based on \eqref{eq:heston_verfahren} and 100 equidistant time steps, we approximately calculated the relative $L^1(\lambda_{[0,1]^d};\R)$-approximation error
\begin{equation}
\label{eq:relL1fehler_hs}
\int_{[0,1]^d} \left|\frac{u(T,x) - \bU^{\Theta_m,1,\bS_m}(x)}{u(T,x)}\right|dx
\end{equation} for $ m \in \{0, 25000, \allowbreak50000,\allowbreak 100000, 150000, 250000, 500000, \allowbreak 750000\}$ 
 by means of Monte Carlo approximations with 10240 samples, we approximately calculated the relative $L^2(\lambda_{[0,1]^d};\R)$-approximation error 
 \begin{equation}
 \label{eq:relL2fehler_hs}
 \sqrt{\int_{[0,1]^d} \left|\frac{u(T,x) - \bU^{\Theta_m,1,\bS_m}(x)}{u(T,x)}\right|^2dx}
 \end{equation} for $ m \in \{0, 25000, \allowbreak50000,\allowbreak 100000, 150000, 250000, 500000, \allowbreak 750000\}$ 
 by means of Monte Carlo approximations with 10240 samples, and we approximately calculated the relative $L^\infty(\lambda_{[0,1]^d};\R)$-approximation error
 \begin{equation}
 \label{eq:relLinftyfehler_hs}
 \sup_{x\in [0,1]^d} \left|\frac{u(T,x) -  \bU^{\Theta_m,1,\bS_m}(x)}{u(T,x)}\right|
 \end{equation} for $ m \in \{0, 25000, \allowbreak50000,\allowbreak 100000, 150000, 250000, 500000, \allowbreak 750000\}$ 
by means of Monte Carlo approximations with 10240 samples (see Lemma~\ref{lemma:MonteCarloSupEstimate} above).  
  
\begin{table}[H]	\begin{center}
		\begin{tabular}{|c|c|c|c|c|}
				\hline
				{\begin{tabular}{@{}c@{}}Number \\ of steps\end{tabular}}  & {\begin{tabular}{@{}c@{}}Relative \\ $L^1(\lambda_{[0,1]^d};\R)$-error\end{tabular}} &  {\begin{tabular}{@{}c@{}}Relative \\ $L^2(\lambda_{[0,1]^d};\R)$-error\end{tabular}} & {\begin{tabular}{@{}c@{}}Relative \\ $L^\infty(\lambda_{[0,1]^d};\R)$-error\end{tabular}} & {\begin{tabular}{@{}c@{}}Runtime \\ in seconds\end{tabular}}\\
				\hline
			0 & 1.038045&1.038686&1.210235& 1.0 \\
			\hline
			25000 &  0.005691&0.007215&0.053298& 688.4 \\
			\hline
			50000 &   0.005115&0.006553&0.036513& 1375.2 \\
			\hline
			100000 &  0.004749&0.005954&0.032411& 2746.8 \\
			\hline
			150000 &  0.006465&0.008581&0.051907& 4120.2 \\
			\hline
			250000 &  0.005075&0.006378&0.024458& 6867.5 \\
			\hline
			500000 & 0.002082&0.002704&0.019604& 13763.7 \\
			\hline
			750000 & 0.00174&0.002233&0.012466& 20758.8 \\
			\hline
		\end{tabular}
		\caption{Approximative presentations of the relative approximation errors in \eqref{eq:relL1fehler_hs}--\eqref{eq:relLinftyfehler_hs} for the Heston model in  \eqref{eq:example_heston_kolmogorovPDE}.}
		\label{tab:heston_linfty_l1_l2}\end{center}
\end{table}

\section{\textsc{{\sc Python}} source codes}\label{sec:source_code}
\subsection{{\sc Python} source code for the algorithm}
In Subsections~\ref{sec:code_paraboliceq}--\ref{sec:code_heston} below we present {\sc Python} source codes associated to the numerical simulations in Subsections~\ref{sec:paraboliceq}--\ref{sec:heston} above. The following {\sc Python} source code, {\sc Python} code~\ref{code:common} below, is employed in the case of each of the {\sc Python} source codes in Subsections~\ref{sec:code_paraboliceq}--\ref{sec:code_heston} below.

\lstset{caption={\it common.py}}

\begin{lstlisting}[label={code:common}]
import numpy as np
import tensorflow as tf
import time
from tensorflow.python.ops import init_ops
from tensorflow.contrib.layers.python.layers import initializers
from tensorflow.python.training.moving_averages import assign_moving_average


def neural_net(x, neurons, is_training, name,
               mv_decay=0.9, dtype=tf.float32):

    def _batch_normalization(_x):
        beta = tf.get_variable('beta', [_x.get_shape()[-1]],
                               dtype, init_ops.zeros_initializer())
        gamma = tf.get_variable('gamma', [_x.get_shape()[-1]],
                                dtype, init_ops.ones_initializer())
        mv_mean = tf.get_variable('mv_mean', [_x.get_shape()[-1]],
                                  dtype, init_ops.zeros_initializer(),
                                  trainable=False)
        mv_variance = tf.get_variable('mv_variance', [_x.get_shape()[-1]],
                                      dtype, init_ops.ones_initializer(),
                                      trainable=False)
        mean, variance = tf.nn.moments(_x, [0], name='moments')
        tf.add_to_collection(tf.GraphKeys.UPDATE_OPS,
                             assign_moving_average(mv_mean, mean,
                                                   mv_decay, True))
        tf.add_to_collection(tf.GraphKeys.UPDATE_OPS,
                             assign_moving_average(mv_variance, variance,
                                                   mv_decay, False))
        mean, variance = tf.cond(is_training,
                                 lambda: (mean, variance),
                                 lambda: (mv_mean, mv_variance))
        return tf.nn.batch_normalization(_x, mean, variance,
                                         beta, gamma, 1e-6)

    def _layer(_x, out_size, activation_fn):
        w = tf.get_variable('weights',
                            [_x.get_shape().as_list()[-1], out_size],
                            dtype, initializers.xavier_initializer())
        return activation_fn(_batch_normalization(tf.matmul(_x, w)))

    with tf.variable_scope(name):
        x = _batch_normalization(x)
        for i in range(len(neurons)):
            with tf.variable_scope('layer_%i_' % (i + 1)):
                x = _layer(x, neurons[i],
                           tf.nn.tanh if i < len(neurons)-1 else tf.identity)
    return x


def kolmogorov_train_and_test(xi, x_sde, phi, u_reference, neurons,
                              lr_boundaries, lr_values, train_steps,
                              mc_rounds, mc_freq, file_name,
                              dtype=tf.float32):

    def _approximate_errors():
        lr, gs = sess.run([learning_rate, global_step])
        l1_err, l2_err, li_err = 0., 0., 0.
        rel_l1_err, rel_l2_err, rel_li_err = 0., 0., 0.
        for _ in range(mc_rounds):
            l1, l2, li, rl1, rl2, rli \
                = sess.run([err_l_1, err_l_2, err_l_inf,
                            rel_err_l_1, rel_err_l_2, rel_err_l_inf],
                           feed_dict={is_training: False})
            l1_err, l2_err, li_err = (l1_err + l1, l2_err + l2,
                                      np.maximum(li_err, li))
            rel_l1_err, rel_l2_err, rel_li_err \
                = (rel_l1_err + rl1, rel_l2_err + rl2,
                   np.maximum(rel_li_err, rli))
        l1_err, l2_err = l1_err / mc_rounds, np.sqrt(l2_err / mc_rounds)
        rel_l1_err, rel_l2_err \
            = rel_l1_err / mc_rounds, np.sqrt(rel_l2_err / mc_rounds)
        t_mc = time.time()
        file_out.write('%i, %f, %f, %f, %f, %f, %f, %f, '
                       '%f, %f\n' % (gs, l1_err,  l2_err, li_err,
                                     rel_l1_err, rel_l2_err, rel_li_err, lr,
                                     t1_train - t0_train, t_mc - t1_train))
        file_out.flush()

    t0_train = time.time()
    is_training = tf.placeholder(tf.bool, [])
    u_approx = neural_net(xi, neurons, is_training, 'u_approx', dtype=dtype)
    loss = tf.reduce_mean(tf.squared_difference(u_approx, phi(x_sde)))

    err = tf.abs(u_approx - u_reference)
    err_l_1 = tf.reduce_mean(err)
    err_l_2 = tf.reduce_mean(err ** 2)
    err_l_inf = tf.reduce_max(err)
    rel_err = err / tf.maximum(u_reference, 1e-8)
    rel_err_l_1 = tf.reduce_mean(rel_err)
    rel_err_l_2 = tf.reduce_mean(rel_err ** 2)
    rel_err_l_inf = tf.reduce_max(rel_err)

    global_step = tf.get_variable('global_step', [], tf.int32,
                                  tf.constant_initializer(0),
                                  trainable=False)
    learning_rate = tf.train.piecewise_constant(global_step,
                                                lr_boundaries,
                                                lr_values)
    optimizer = tf.train.AdamOptimizer(learning_rate)
    update_ops = tf.get_collection(tf.GraphKeys.UPDATE_OPS, 'u_approx')
    with tf.control_dependencies(update_ops):
        train_op = optimizer.minimize(loss, global_step)

    file_out = open(file_name, 'w')
    file_out.write('step, l1_err, l2_err, li_err, l1_rel, '
                   'l2_rel, li_rel, learning_rate, time_train, time_mc\n')

    with tf.Session() as sess:

        sess.run(tf.global_variables_initializer())

        for step in range(train_steps):
            if step % mc_freq == 0:
                t1_train = time.time()
                _approximate_errors()
                t0_train = time.time()
            sess.run(train_op, feed_dict={is_training: True})
        t1_train = time.time()
        _approximate_errors()

    file_out.close()
\end{lstlisting}

\subsection[A {\sc Python} source code associated to Subsection~\ref{sec:paraboliceq}]{A {\sc Python} source code associated to the numerical simulations in Subsection~\ref{sec:paraboliceq}}
\label{sec:code_paraboliceq}
\lstset{caption={\it example3\_2.py}}

\begin{lstlisting}[label={code:paraboliceq}]
import numpy as np
import tensorflow as tf
from common import kolmogorov_train_and_test

tf.reset_default_graph()
dtype = tf.float32
T, N, d = 1., 1, 100
batch_size = 8192
neurons = [d + 100, d + 100, 1]
train_steps = 750000
mc_rounds, mc_freq = 1250, 100
lr_boundaries = [250001, 500001]
lr_values = [0.001, 0.0001, 0.00001]
xi = tf.random_uniform(shape=(batch_size, d), minval=0.,
                       maxval=1., dtype=dtype)
x_sde = xi + tf.random_normal(shape=(batch_size, d),
                              stddev=np.sqrt(2. * T / N), dtype=dtype)


def phi(x):
    return tf.reduce_sum(x ** 2, axis=1, keepdims=True)


u_reference = phi(xi) + 2. * T * d

kolmogorov_train_and_test(xi, x_sde, phi, u_reference, neurons,
                          lr_boundaries, lr_values, train_steps,
                          mc_rounds, mc_freq, 'example3_1.csv', dtype)
\end{lstlisting}

\subsection[A {\sc Python} source code associated to Subsection~\ref{sec:geometric}]{A {\sc Python} source code associated to the numerical simulations in Subsection~\ref{sec:geometric}}
\label{sec:code_geometric}
\lstset{caption={\it example3\_3.py}}

\begin{lstlisting}[label={code:geometric}]
import numpy as np
import tensorflow as tf
from common import kolmogorov_train_and_test

tf.reset_default_graph()
dtype = tf.float32
T, N, d = 1., 1, 100
r, c, K = 0.05, 0.1, 100.
sigma = tf.constant(0.1 + 0.5 * np.linspace(start=1. / d, stop=1., num=d,
                                            endpoint=True), dtype=dtype)
batch_size = 8192
neurons = [d + 100, d + 100, 1]
train_steps = 750000
mc_rounds, mc_freq = 10, 25000
mc_samples_ref, mc_rounds_ref = 1024, 1024
lr_boundaries = [250001, 500001]
lr_values = [0.001, 0.0001, 0.00001]
xi = tf.random_uniform((batch_size, d), minval=90., maxval=110., dtype=dtype)


def phi(x, axis=1):
    return np.exp(-r * T) \
           * tf.maximum(tf.reduce_max(x, axis=axis, keepdims=True) - K, 0.)


def mc_body(idx, p):
    _x = xi * tf.exp((r - c - 0.5 * sigma ** 2) * T + sigma
                     * tf.random_normal((mc_samples_ref, batch_size, d),
                                        stddev=np.sqrt(T / N), dtype=dtype))
    return idx + 1, p + tf.reduce_mean(phi(_x, 2), axis=0)


x_sde = xi * tf.exp((r - c - 0.5 * sigma ** 2) * T
                    + sigma * tf.random_normal((batch_size, d),
                                               stddev=np.sqrt(T / N),
                                               dtype=dtype))
_, u = tf.while_loop(lambda idx, p: idx < mc_rounds_ref, mc_body,
                     (tf.constant(0), tf.zeros((batch_size, 1), dtype)))
u_reference = u / tf.cast(mc_rounds_ref, tf.float32)

kolmogorov_train_and_test(xi, x_sde, phi, u_reference, neurons,
                          lr_boundaries, lr_values, train_steps,
                          mc_rounds, mc_freq, 'example3_2.csv', dtype)

\end{lstlisting}

\subsection[A {\sc Python} source code associated to Subsection~\ref{sec:black_scholes}]{A {\sc Python} source code associated to the numerical simulations in Subsection~\ref{sec:black_scholes}}
\label{sec:code_black_scholes}
\lstset{caption={\it example3\_4.py}}

\begin{lstlisting}[label={code:black_scholes}]
import numpy as np
import tensorflow as tf
from common import kolmogorov_train_and_test

tf.reset_default_graph()
dtype = tf.float32
T, N, d = 1., 1, 100
r, c, K = 0.05, 0.1, 110.
Q = np.ones([d, d]) * 0.5
np.fill_diagonal(Q, 1.)
L = np.linalg.cholesky(Q).transpose()
sigma_norms = tf.constant(np.linalg.norm(L, axis=0), dtype=dtype)
sigma = tf.constant(L, dtype=dtype)
beta = tf.constant(0.1 + 0.5 * np.linspace(start=1. / d, stop=1., num=d,
                                           endpoint=True), dtype=dtype)
batch_size = 8192
neurons = [d + 100, d + 100, 1]
train_steps = 750000
mc_rounds, mc_freq = 10, 25000
mc_samples_ref, mc_rounds_ref = 1024, 1024
lr_boundaries = [250001, 500001]
lr_values = [0.001, 0.0001, 0.00001]
xi = tf.random_uniform((batch_size, d), minval=90., maxval=110., dtype=dtype)


def phi(x, axis=1):
    return np.exp(-r * T) \
           * tf.maximum(K - tf.reduce_min(x, axis=axis, keepdims=True), 0.)


def mc_body(idx, p):
    _w = tf.matmul(tf.random_normal((mc_samples_ref * batch_size, d),
                                    stddev=np.sqrt(T / N), dtype=dtype),
                   sigma)
    _w = tf.reshape(_w, (mc_samples_ref, batch_size, d))
    _x = xi * tf.exp((r - c - 0.5 * (beta * sigma_norms) ** 2) * T
                     + beta * _w)
    return idx + 1, p + tf.reduce_mean(phi(_x, 2), axis=0)


x_sde = xi * tf.exp((r - c - 0.5 * (beta * sigma_norms) ** 2) * T + beta
                    * tf.matmul(tf.random_normal((batch_size, d),
                                                 stddev=np.sqrt(T / N),
                                                 dtype=dtype),
                                sigma))
_, u = tf.while_loop(lambda idx, p: idx < mc_rounds_ref, mc_body,
                     (tf.constant(0), tf.zeros((batch_size, 1), dtype)))
u_reference = u / tf.cast(mc_rounds_ref, tf.float32)

kolmogorov_train_and_test(xi, x_sde, phi, u_reference, neurons,
                          lr_boundaries, lr_values, train_steps,
                          mc_rounds, mc_freq, 'example3_3.csv', dtype)


\end{lstlisting}

\subsection[A {\sc Python} source code associated to Subsection~\ref{sec:lorenz}]{A {\sc Python} source code associated to the numerical simulations in Subsection~\ref{sec:lorenz}}
\label{sec:code_lorenz}
\lstset{caption={\it example3\_5.py}}

\begin{lstlisting}[label={code:lorenz}]
import numpy as np
import tensorflow as tf
from common import kolmogorov_train_and_test


tf.reset_default_graph()
dtype = tf.float32
batch_size = 1024

T, N, d = 1., 100, 3
alpha_1, alpha_2, alpha_3 = 10., 14., 8./3.
beta = tf.constant([0.15, 0.15, 0.15], dtype=dtype)
h = T / N
neurons = [d + 20, d + 20, 1]
train_steps = 750000
mc_rounds, mc_freq = 20, 25000
mc_samples_ref, mc_rounds_ref = 1024, 1024
lr_boundaries = [250001, 500001]
lr_values = [0.001, 0.0001, 0.00001]
xi = tf.stack([tf.random_uniform((batch_size,), minval=0.5,
                                 maxval=2.5, dtype=dtype),
               tf.random_uniform((batch_size,), minval=8.,
                                 maxval=10., dtype=dtype),
               tf.random_uniform((batch_size,), minval=10.,
                                 maxval=12., dtype=dtype)], axis=1)


def phi(x, axis=1):
    return tf.reduce_sum(x ** 2, axis=axis, keepdims=True)


def mu(x):
    x_1 = tf.expand_dims(x[:, :, 0], axis=2)
    x_2 = tf.expand_dims(x[:, :, 1], axis=2)
    x_3 = tf.expand_dims(x[:, :, 2], axis=2)
    return tf.concat([alpha_1 * (x_2 - x_1),
                      alpha_2 * x_1 - x_2 - x_1 * x_3,
                      x_1 * x_2 - alpha_3 * x_3], axis=2)


def sde_body(idx, s, samples):
    return tf.add(idx, 1), s \
           + tf.cast(T / N * tf.sqrt(phi(mu(s), 2 if samples > 1 else 1))
                     <= 1., dtype) * mu(s) * T / N \
           + beta * tf.random_normal((samples, batch_size, d),
                                     stddev=np.sqrt(T / N), dtype=dtype)


def mc_body(idx, p):
    _, _x = tf.while_loop(lambda _idx, s: _idx < N,
                          lambda _idx, s: sde_body(_idx, s,
                                                   mc_samples_ref),
                          loop_var_mc)
    return idx + 1, p + tf.reduce_mean(phi(_x, 2), axis=0)


loop_var_mc = (tf.constant(0),
               tf.ones((mc_samples_ref, batch_size, d), dtype) * xi)
loop_var = (tf.constant(0), tf.ones((1, batch_size, d), dtype) * xi)
_, x_sde = tf.while_loop(lambda idx, s: idx < N,
                         lambda idx, s: sde_body(idx, s, 1),
                         loop_var)
_, u = tf.while_loop(lambda idx, p: idx < mc_rounds_ref,
                     mc_body,
                     (tf.constant(0), tf.zeros((batch_size, 1), dtype)))
u_reference = u / tf.cast(mc_rounds_ref, tf.float32)

kolmogorov_train_and_test(xi, tf.squeeze(x_sde, axis=0), phi, u_reference,
                          neurons, lr_boundaries, lr_values, train_steps,
                          mc_rounds, mc_freq, 'example3_4.csv', dtype)


\end{lstlisting}

\subsection[A {\sc Python} source code associated to Subsection~\ref{sec:heston}]{A {\sc Python} source code associated to the numerical simulations in Subsection~\ref{sec:heston}}
\label{sec:code_heston}
\lstset{caption={\it example3\_6.py}}

\begin{lstlisting}[label={code:heston}]
import numpy as np
import tensorflow as tf
from common import kolmogorov_train_and_test

tf.reset_default_graph()
dtype = tf.float32
batch_size = 1024

T, N, d = 1., 100, 50
alpha, K = 0.05, 110.
kappa, sigma = 0.6, 0.2
theta, rho = 0.04, -0.8
S_0 = tf.random_uniform((batch_size, d / 2),
                        minval=90., maxval=110., dtype=dtype)
V_0 = tf.random_uniform((batch_size, d / 2),
                        minval=0.02, maxval=0.2, dtype=dtype)
h = T / N
neurons = [d + 50, d + 50, 1]
train_steps = 750000
mc_rounds, mc_freq = 10, 25000
mc_samples_ref, mc_rounds_ref = 256, 4096
lr_boundaries = [250001, 500001]
lr_values = [0.001, 0.0001, 0.00001]
xi = tf.reshape(tf.stack([S_0, V_0], axis=2),
                (batch_size, d))


def phi(x, axis=1):
    return np.exp(-alpha * T) \
           * tf.maximum(K - tf.reduce_mean(tf.exp(x), axis=axis,
                                           keepdims=True), 0.)


def sde_body(idx, s, v, samples):
    _sqrt_v = tf.sqrt(v)
    dw_1 = tf.random_normal(shape=(samples, batch_size, d / 2),
                            stddev=np.sqrt(h), dtype=dtype)
    dw_2 = rho * dw_1 + np.sqrt(1. - rho ** 2) \
           * tf.random_normal(shape=(samples, batch_size, d / 2),
                              stddev=np.sqrt(h), dtype=dtype)
    return tf.add(idx, 1), s + (alpha - v / 2.) * h + _sqrt_v * dw_1, \
        tf.maximum(tf.maximum(np.float32(sigma / 2. * np.sqrt(h)),
                              tf.maximum(np.float32(sigma / 2. * np.sqrt(h)),
                                         _sqrt_v)
                              + sigma / 2. * dw_2) ** 2
                   + (kappa * theta - sigma ** 2 / 4. - kappa * v) * h, 0.)


def mc_body(idx, p):
    _, _x, __v = tf.while_loop(lambda _idx, s, v: _idx < N,
                               lambda _idx, s, v: sde_body(_idx, s, v,
                                                           mc_samples_ref),
                               loop_var_mc)
    return idx + 1, p + tf.reduce_mean(phi(_x, 2), axis=0)


loop_var_mc = (tf.constant(0), tf.ones((mc_samples_ref, batch_size, d / 2),
                                       dtype) * tf.log(S_0),
               tf.ones((mc_samples_ref, batch_size, d / 2),
                       dtype) * V_0)
loop_var = (tf.constant(0), tf.ones((1, batch_size, d / 2),
                                    dtype) * tf.log(S_0),
            tf.ones((1, batch_size, d / 2), dtype) * V_0)
_, x_sde, _v = tf.while_loop(lambda idx, s, v: idx < N,
                             lambda idx, s, v: sde_body(idx, s, v, 1),
                             loop_var)
_, u = tf.while_loop(lambda idx, p: idx < mc_rounds_ref, mc_body,
                     (tf.constant(0), tf.zeros((batch_size, 1), dtype)))
u_reference = u / tf.cast(mc_rounds_ref, tf.float32)

kolmogorov_train_and_test(xi, tf.squeeze(x_sde, axis=0), phi, u_reference,
                          neurons, lr_boundaries, lr_values, train_steps,
                          mc_rounds, mc_freq, 'example3_5.csv', dtype)


\end{lstlisting}

\bibliographystyle{acm}
\bibliography{bibfile}

\end{document}